\documentclass{amsart}
\usepackage{amsmath,amssymb,amsthm,amsfonts,setspace,hyperref,dsfont,yhmath,amsmath,amscd,graphicx,enumerate,verbatim,cite,tikz-cd,color}

\usetikzlibrary{matrix,arrows,decorations.pathmorphing}

\DeclareMathAlphabet{\mathpzc}{OT1}{pzc}{m}{it}

\newtheorem{theorem}{Theorem}[section]
\newtheorem{proposition}[theorem]{Proposition}
\newtheorem{lemma}[theorem]{Lemma}
\newtheorem{corollary}[theorem]{Corollary}

\theoremstyle{definition}

\newtheorem{definition}{Definition}[section]
\newtheorem{fact}{Fact}[section]
\newtheorem{sect}{}

\theoremstyle{remark}
\newtheorem{remark}{Remark}
\newtheorem{example}{Example}[section]

\newcommand{\KK}{\mathbb{K}}
\newcommand{\NN}{\mathbb{N}}
\newcommand{\RR}{\mathbb{R}}

\newcommand{\CC}{\mathbb{C}}
\newcommand{\QQ}{\mathbb{Q}}
\newcommand{\TT}{\mathbb{T}}
\newcommand{\ZZ}{\mathbb{Z}}

\newcommand{\mbb}{\mathbb}
\newcommand{\R}{\mathbb{R}}
\newcommand{\Z}{\mathbb{Z}}
\newcommand{\C}{\mathbb{C}}
\newcommand{\Q}{\mathbb{Q}}
\newcommand{\N}{\mathbb{N}}
\newcommand{\T}{\mathbb{T}}

\newcommand{\floor}[1]{\left\lfloor #1 \right\rfloor}

\newcommand{\tth}{^{\mbox{\rm{\scriptsize{th}}}}}

\newcommand{\abs}[1]{\left| #1 \right|}
\newcommand{\of}{\circ}
\newcommand{\mbf}{\mathbf}

\newcommand{\ve}{\varepsilon}

\newcommand{\inner}[2]{\left\langle #1 , #2 \right\rangle}

\newcommand{\norm}[1]{\abs{\abs{#1}}}
\newcommand{\set}[1]{\left\lbrace #1 \right\rbrace}

\newcommand{\mc}{\mathcal}
\newcommand{\mf}{\mathfrak}

\newcommand{\seref}[1]{{\bf \ref{#1}.}}
\newcommand{\starref}{\hyperlink{star}{$(*)$}}

\def\id{\operatorname{id}}
\def\Diff{\operatorname{Diff}}
\def\Aut{\operatorname{Aut}}

\def\ad{\operatorname{ad}}
\def\Ad{\operatorname{Ad}}

\def\diag{\operatorname{diag}}

\def\Lie{\operatorname{Lie}}

\def\Hol{\operatorname{H\ddot{o}l}}

\begin{document}
\title[Local rigidity and the geometric method]
{Local Rigidity of Higher Rank Homogeneous Abelian Actions: a Complete Solution via the Geometric Method}
\author[Kurt Vinhage \and Zhenqi Jenny Wang]{Kurt Vinhage$^1$ and Zhenqi Jenny Wang$^2$}
\thanks{ $^1$ Based on research supported by NSF grant   DMS-1604796 \\
 $^2$ Based on research supported by NSF grant   DMS-1346876}

\subjclass[2013]{}
\address{Department of Mathematics, University of Chicago, Chicago, IL 60637, USA}

\email{kvinhage@uchicago.edu}

\address{Department of Mathematics, Michigan State University, East Lansing, MI 48824,USA}

\email{wangzq@math.msu.edu}
\begin{abstract}
We show local and cocycle rigidity for $\R^k \times \Z^l$ partially hyperbolic translation
actions on homogeneous spaces $\mc G/ \Lambda$. We consider a large class of actions whose geometric
properties are more complicated than previously treated cases. It is also the first time that partially hyperbolic twisted
symmetric space examples have been treated in the literature. The main new ingredient in the proof is a combination of geometric method and the theory of central extensions.

\end{abstract}

\maketitle

\section{Introduction}


Local rigidity is a phenomenon which generalizes and strengthens the notion of structural stability. Fix an acting group $\R^k \times \Z^l$. An action $\alpha : \R^k \times \Z^l \to \Diff^p(M)$ is {\it structurally stable} if whenever $\alpha'$ is an $\R^k \times \Z^l$-action sufficiently to $\alpha$ on a compact generating set in the $C^1$-topology, then there is a homeomorphism $h : M \to M$ which takes orbits of $\alpha$ to orbits of $\alpha'$. Hyperbolic systems, even those of rank one, are known to be structurally stable.

Local rigidity strengthens this notion in one way, by making the orbit equivalence a conjugacy, but weakens it in another, by allowing a (usually finite-dimensional) class of model perturbations. Consider a $C^\infty$ $\R^k \times \Z^l$-action $\alpha$ lying inside a class of actions $\mc A \subset \Diff^\infty(M)$, and suppose that $\alpha'$ is a $C^p$ action which is close to $\alpha$ on a compact generating set in the $C^q$ topology. If for every such $\alpha'$, there is a $C^r$ diffeomorphism $h$ such that $h^{-1} \of \alpha' \of h \in \mc A$, then $\alpha$ is said to be {\it $C^{p,q,r}$-rigid with respect to $\mc A$}. If $\mc A$ is understood, then we will often simply say $C^{p,q,r}$-rigid.

$\mc A$ is always chosen from a canonical set of models. This class is easily described in the case of homogeneous actions. An action $\alpha : \R^k \times \Z^l \to \Diff(M)$ is homogeneous if $M = G / \Gamma$ is the quotient of a Lie group $G$ by a lattice (or, more generally, closed subgroup) $\Gamma$ and there exists a homomorphism $i : \R^k \times \Z^l \to G$ such that $\alpha(a)g\Gamma = (i(a)g)\Gamma$. Then $\mc A$ may be chosen to be all homogeneous actions determined by homomorphisms $i'$ in a neighborhood of $i$.  When $i(\R^k \times \Z^l)$ is a split Cartan subgroup of a semisimple group, the action is called a {\it Weyl chamber flow}. 

Local rigidity for Anosov algebraic actions was first proved by Katok and Spatzier using harmonic analysis of homogeneous spaces \cite{ks97}. The method reduces the question of local rigidity to one of cocycle rigidity. Speical properties of Anosov actions are key to the analysis: that the orbit foliation and central foliations coincide give a global system of coordinates on this foliation (using the dynamics). This allows one to consider the cocycle that appears as a cocycle over the algebraic action, and also guarantees high regularity. Katok and Damjanovi{\'c} extended the local rigidity to restrictions of Weyl chamber flows on quotients of $SL(d,k)$ $k=\RR,\,\CC$ \cite{dk2011} using a new method, called the {\it geometric method}.

Instead of hard functional analysis, the solution of the cocycle problem which appears was defined along intersections of stable foliations called the {\it coarse Lyapunov foliations}. This has the benefit of not requiring high regularity of the cocycle. Since the web these foliations weave is transitive on the space $M$, one can extend the solution globally, provided well-definedness can be verified. Here, algebraic $K$-theory of the fields $\R$ and $\C$ was essential, as there is a direct correspondence between the $K_2$ groups and the web of coarse Lyapunov foliations. The scheme was further employed in \cite{damjanovic07}, \cite{zwang-1}, \cite{zwang-2} and \cite{vinhage15} to obtain local rigidity for various restrictions of Weyl chamber flows, which replaced the $K$-theory component with the more general theory of central extensions. However, these results were strongly limited by a genericity assumption which allowed the algebraic tools to be deduced from existing literature, most notably \cite{deodhar78}. As a result, a large class of partially hyperbolic symmetric space examples were not
covered.

The main difficulty in extending the local rigidity results for generic restrictions of Weyl chamber flows to other homogeneous partially hyperbolic systems is understanding the connection between the theory of central extensions and the web of coarse Lyapunov foliations. Since the goal of \cite{deodhar78} is to describe the universal central extension and not its relationship to abelian actions, they use the most convenient generators and relations: the coset foliations of root subgroups and the commutator relations between them. Generic restrictions of Weyl chamber flows are defined exactly so that the coarse Lyapunov foliations coincide with these coset foliations.

The progress made in \cite{zwang-1,zwang-2} used heavy computations for generic restrictions, where the algebra can be described through well-classified real forms of root systems. Since (irreducible) root systems appear in series $A_n$ through $D_n$, together with finitely many exceptional systems, the case of generic restrictions could conceivably be treated by considering each series ad-hoc.\footnote{Recently, a general, non-computational solution was found in \cite{vinhage15}, but it still relied heavily on the genericity assumption} The algebra that comes from other partially hyperbolic actions is not classfied in such a manner, so while an ad-hoc approach may be able to treat certain special actions, the broad treatment we provide requires a uniform approach.

Furthermore, our results are the first to address partially hyperbolic actions on semidirect produts (sometimes called twisted symmetric space examples). The case of semidirect products was thought to require a Fourier analysis approach to handle the torus fibers. To use the geometric method, we must describe what types of central extensions can appear in the Lie category, as well as show the coarse Lyapunov foliations correspond to a central extension. No well-developed theory for such questions existed for semidirect products, even in the Anosov setting. As a result, to obtain local rigidity, we prove results in these directions which are of indpendent interest.  

In this paper we prove the following theorem (for a precise formulation, see Theorem \ref{thm:main}), which shows local rigidity in a remarkably broad class of actions:
\begin{theorem}
\label{thm:imprecise}
Assume that $\alpha$ is a partially hyperbolic homogeneous action satisfying the assumptions laid out in Section \ref{sec:2}. Then $\alpha$ is $C^{\infty,1,\infty}$-rigid.
\end{theorem}
Theorem \ref{thm:imprecise} (Theorem \ref{thm:main}) is the first case of local rigidity for a broad
class of twisted symmetric space examples which are not Anosov. Such cases were believed to require a harmonic analysis approach in the torus fibers.
Even in the case of semisimple groups, the results are new, removing the genericity assumption. In fact, we obtain topological rigidity under a much more general assumption called the {\it genuinely higher rank condition} (see \starref). Furthermore, by employing new schemes we give a uniform approach for all partially hyperbolic algebraic systems.


\subsection{Outline of the Paper}
In Section \ref{sec:2}, we introduce the actions which we study and assumptions which will appear in our main results, and in Section \ref{sec:statement}
we state our results precisely.
In Sections \ref{sec:geom1} and \ref{sec:geom2} we study the geometric structure for general actions:  we describe the coarse Lyapunov foliations and the universal central Lie extensions (which are non-trivial for some twisted spaces), lift the action and its small perturbations from homogeneou spaces to the central Lie extensions, construct periodic cycle functional (PCF) for the twisted cocycle over general non-compact groups instead of the product of an abelian group and a compact group. These results in some ways parallel those that appear in previous papers for restrictions of full Cartan actions \cite{dk2005, dk2011, damjanovic07, zwang-1,vinhage15}.

The new ideas of the paper appear in Sections \ref{sec:alg1} through \ref{sec:alg3}, where we prove that the group generated by the coarse Lyapunov subgroups modulo the so-called {\it stable relations} is a central extension of the ambient Lie group (see Theorem \ref{thm:central}). For general actions, the corresponding root and weight systems are quite different from the standard ones. We generalize the ideas of R. Steinberg \cite{Steinberg}, \cite{Steinberg2} and V. Deodhar \cite{deodhar78} and adapt the algebraic structure to the geometric setting for general actions. This is one of the main difficulties of our approach. In the symmetric space examples the main technique is choosing suitable Weyl elements to recover undetected roots (Section \ref{sec:alg2}). For twisted spaces we make use of the irreducibility of the representation and Lie group structure to handle the difficulty from the non-trivial radical. These sections contain the main technical obstacles of the paper.

In Sections \ref{sec:cocyle-rigid} and \ref{sec:perturbed}, based on the conclusions from earlier sections,
we prove trivialization of small twisted cocycles obtained from the perturbation. Once the cocycle rigidity is obtained, the transfer function provides a continuous conjugacy. Global smoothness follows from the non-stationary normal form method. The key ingredients of in the proof of trivialization of twisted cocycles are:
(i) local transitivity of Lyapunov foliations, (ii) vanishing of PCF on stable cycles, (iii) vanishing of continuous continuous homeomorphisms from Schur multipliers to Lie groups and (iv) vanishing of small homeomorphism from $\Gamma$ to the neutral subgroup.

\section{The setting}\label{sec:2}
We consider  actions of higher rank abelian groups
 $\ZZ^k\times \RR^l, \,\, k+l\ge 2$ that come
from the following general  algebraic construction:

Let $H$ be a
connected Lie group, $A\subseteq H$ a closed abelian subgroup which is isomorphic to $\ZZ^k\times \RR^l$, $L$ a compact subgroup of
the centralizer $Z(A)$ of $A$, and $\Upsilon$ a cocompact lattice in
$H$. Then $\Z^k \times \R^l$ acts by left translations on the compact space
$M=L\backslash H/\Upsilon$. \emph{The linear part of the action} is the representation of $A$ on $\text{Lie}(H)/\text{Lie}(L)$ induced by
the adjoint representation of $A$ on $\text{Lie}(H)$.

The two specific
types of standard partially hyperbolic examples discussed below correspond to:
\begin{enumerate}
  \item [(a)] for the symmetric space examples take $H$ a semisimple connected Lie group $G$ of $\RR$-rank $\ge 2$ and $\Upsilon$ an
irreducible torsion-free cocompact lattice $\Gamma$ in $G$

  \medskip

  \item [(b)]  for the twisted symmetric space examples take $H=G\ltimes\RR^n$, a semidirect
product of a semisimple connected Lie group $G$ of $\RR$-rank $\ge 2$ with $\RR^n$, $\Upsilon$ a semidirect product of
an irreducible torsion-free cocompact lattice in $G$ with $\ZZ^n$
\end{enumerate}

\begin{definition} {\em Coarse Lyapunov distributions} are defined as the maximal non-trivial intersections of stable distributions of various action
elements.\end{definition}

For homogeneous actions these distribution are homogeneous that integrate to coset foliations called {\em coarse Lyapunov foliations} (see \cite[Section
2]{dk2005} and \cite{Kalinin} for detailed discussion in
greater generality).

We will now discuss these standard  examples in more detail. Throughout this paper $G$ will always denote a semisimple connected Lie group  of $\RR$- rank $\ge 2$ without compact factors and with finite center and $\Gamma$ an irreducible torsion-free cocompact lattice
in $G$. For any group $S$ we use $Z(S)$ to denote its center.

\subsection{Symmetric space examples}\label{sec:9}
For any $g\in G$ we have a corresponding Jordan-Chevalley normal form decomposition of $3$ commuting elements $g=s_gk_gn_g$ where $s_g$ is semisimple, $k_g$ is compact and $n_g$ is nilpotent. Set $p(g)=s_g$. For any abelian set $A \subset G$ there exists a $\RR$-split Cartan subgroup $A_0$ such that $p(a)\in A_0$ for any $a\in A $ (see Proposition \ref{prop:diagonalize}).

A  standard root system for $A_0$ comes from the decomposition of $\mathfrak{g}$, the Lie algebra of $G$, into the eigenspaces of
adjoint representation of $A_0$ whose elements are
simultaneously diagonalizable. For the subgroup $p(A)$ of $A_0$ we can also consider the
decomposition of $\mathfrak{g}$ with respect to the adjoint
representation of $p(A)$ and the resulting root system corresponding to the assignment of eigenvalues is
called the {\em restricted root system with respect to $A$}. Let $\Delta_A$
denote the restricted root system.  If $A=A_0$, $\Delta_{A_0}$ is the standard root system, and since $p(A) \subset A_0$, each $r \in \Delta_{A_0}$ restricts to an element of $\Delta_A$. $r \in \Delta_{A_0}$ is said to be {\it detected by $A$} if it does not restrict to the 0 functional on $A$. Calling $\Delta_{A}$ a restricted root system is somewhat abusive. Indeed, $\Delta_{A}$ does not carry the usual structures of a (restricted) root system, such as a canonical inner product and associated Weyl group (see Section \ref{sec:6}).

The Lie algebra
$\mathfrak{g}$ of $G$ decomposes
\begin{align}\label{for:6}
\mathfrak{g}=\mf g_0 + \sum_{\mu\in\Delta_A}\mathfrak{g}_\mu
\end{align}
where $\mathfrak{g}_\mu$ is the associated joint eigenspace of $\mu$ and
$\mf g_0$ is the Lie algebra of the centralizer $C_G(p(A))$ of $p(A)$.

Elements of
$A\setminus p^{-1}\left(\bigcup_{\phi\in\Delta_A}\ker(\phi)\right)$ are
\emph{regular} elements (for $A$). Connected components of the set of regular
elements are \emph{Weyl chambers} (for $A$). For any $\mu \in \Delta_A$ let $\mathfrak{g}^{(\mu)}=\sum_{k>0}\mathfrak{g}_{k\mu}$ and
$U_{[\mu]}$ be the corresponding subgroup of $G$. Then these subalgebra $\mathfrak{g}^{(\mu)}$ form coarse Lyapunov distributions and (double) cosets of these subgroups $U_{[\mu]}$ form coarse Lyapunov foliations of $\alpha_A$ (see Definition \ref{de:3}), which coincide with those of $\alpha_{p(A)}$ (see Proposition \ref{prop:diagonalize}).

If $p(A)=A_0$, the left translations  of $A$ on $ G/\Gamma$  is sometimes referred to as  {\em full Cartan action} (see \cite{ks97}). If the coarse Lyapunov foliations of $\alpha_A$ coincide with those of $\alpha_{A_0}$, then $p(A)$ is in a generic position (see \cite{dk2005}) and the action of $A$ on $ G/\Gamma$ is called a \emph{generic restriction}. In these two cases, $p(A)$ contains a regular element for $A_0$, which implies that every element of $A$ is semisimple (in particular, it is a subgroup of $C_G(A_0)$, which is always the product of $A_0$ and a compact group). This is why in former papers all the actions taken studies are restrictions of a split Cartan subgroup.

\begin{example}
Our first example is one which detects every root but is not generic (in the sense of \hyperlink{starp}{$(*')$}). Let $G = SL(dm,\R)$, with $d,m \ge 2$. Let $\mu_1,\dots,\mu_m \in \R$ be distinct real numbers such that $\sum \mu_i =0$, and consider the embedding of $\R^d$ defined via:

\[ \mbf{t} =  (t_1,\dots,t_d) \mapsto \diag(\mu_1 t_1,\dots,\mu_1t_d, \mu_2 t_1,\dots,\mu_2 t_d,\dots,\mu_mt_1, \dots,\mu_mt_d)  \]

Then on easily sees that the usual roots assigning the differences of entries of a diagonal matrix restrict to \[\Delta_A = \set{ \mbf t \mapsto \mu_i t_k - \mu_j t_l : 1 \le i,j \le m, 1 \le k,l \le d, (i,k) \not= (j,l)}.\] Note that if $i = j$, this is a multiple of the usual roots on $SL(d,\R)$, so many of the roots are proportional, positively and negatively. Similarly, if $k = l$, the corresponding root is a multiple of the functional which assigns the $k\tth$ entry of $\mbf t$, which is nonzero since we assume all $\mu_i$ are distinct. Notice that all of the usual roots of $SL(dm,\R)$ with respect to the diagonal subgroup are still detected by this embedding since all of the $\mu_i$ are distinct.
\end{example}

\begin{example}
\label{sec:6}
We also consider the following example to illustrate that Jordan blocks may appear. Let $G = SL(4,\R)$ and
\begin{align*}
  A &= \set{\begin{pmatrix}
e^t & e^ts & 0 & 0 \\
0 & e^t & 0 & 0 \\
0 & 0 & e^{s} & 0 \\
0 & 0 & 0 & e^{-2t-s}
\end{pmatrix}}= \exp\set{\begin{pmatrix}
t & s & 0 & 0 \\
0 & t & 0 & 0 \\
0 & 0 & s & 0 \\
0 & 0 & 0 & -2t-s
\end{pmatrix}},
\end{align*}
where $s,\,t\in\RR$. Then $A \cong \R^2$. The restricted root system for the corresponding subgroup $A' = p(A) = \set{\diag(e^t,e^t,e^{s},e^{-2t-s})}$ (the ``diagonal part'' of $A$) are
\begin{align*}
\chi_1 = \pm(t-s),\quad \chi_2 = \pm(3t+s),\quad\chi_3 = \pm(2t+2s).
\end{align*}
Both $\mathfrak{g}_{\chi_1}$ and $\mathfrak{g}_{\chi_2}$ are $2$-dimensional, but $\mf g_{\chi_3}$ is 1-dimensional. Note that $A'$ sits inside the full Cartan subgroup of diagonal matrices.
The $\ad$-action of $\log A$ on these spaces is not diagonalizable. For instance, on the subspace $\mathfrak{g}_{\chi_1}$ is given by:
\[ \mathfrak{g}_{\chi_1} = \set{\begin{pmatrix}
0 & 0 & x & 0 \\
0 & 0 & y & 0 \\
0 & 0 & 0 & 0 \\
0 & 0 & 0 & 0
\end{pmatrix} : x, y \in \R} .\]

Then the action of $\log A$ on $\mathfrak{g}_{\chi_1}$  in these coordinates is given by
\begin{align*}
 (x,y) \mapsto ((t-s)x + sy,(t-s)y).
\end{align*}
Furthermore, the neutral distributions exponentiates is $SL(2,\R)\times\RR^2$, which is the centralizer of $A'$ instead of $A$.

The restricted root system of $A$ is quite different from the standard root system corresponding to $A_0$. In our example, the Weyl group for the root system $\Delta_{A_0}$ is symmetric group $S_4$. However, $N(A')/C(A')=\{\id,(34)\} \cong \Z / 2\Z$, which does not act transitively on the Weyl chambers of $\Delta_{A}$.
Furthermore, $S_4$ (the Weyl group of $A_0$) nor does preserve the root spaces at all. For example, $\chi_1$ is the restriction of two roots: $r_{13}$ and $r_{23}$, which associate to a diagonal matrix the difference of its first and third, and second and third entries, respectively. Then $\mathfrak{g}_{\chi_1}$ under the action of $(13)$ (realized in the group $G$) in coordinates $(x,y)$ is given by
\[ \set{\begin{pmatrix}
0 & 0 & 0 & 0 \\
y & 0 & 0 & 0 \\
x & 0 & 0 & 0 \\
0 & 0 & 0 & 0
\end{pmatrix} : x, y \in \R} \]
which is split across elements of $\mf g_0$ and $\mf g_{-\chi_1}$. While this creates significant bookkeeping, the breaking of the restricted root system $\Delta_{A}$ is essential to our analysis.
\end{example}

\subsection{Twisted symmetric space examples} \label{sec:3}  Let $G$, $\Gamma$ and their associated structures be as in Section \ref{sec:9}, with the additional assumption that $\Gamma = \rho^{-1}(SL(N,\Z) \cap \rho(G))$ for some reprsentation $\rho : G \to SL(N,\R)$ such that the decomposition into irreducible subrepresentations does not contain the trivial representation.\footnote{Margulis superrigidity shows that one may expect this description for many lattices in $G$, but in general one may require some compact correction \cite[Theorem VII.5.13]{margulis91}}

We can build the associated semi-direct product $G_{\rho}=G\ltimes_\rho \RR^N$, and lattice $\Gamma_\rho = \Gamma \ltimes_\rho \Z^N$. The
multiplication of elements in $G_\rho$ is given by
\begin{align}\label{for:11}
(g_1,x_1)\cdot(g_2,x_2)=(g_1g_2,\rho(g_2^{-1})x_1+x_2),
\end{align}

This construction is equivalent to the following: consider the smooth manifold $G \times \T^N$, and the $\Gamma$ action $\gamma \cdot (g,x) = (g\gamma^{-1},\rho(\gamma)x)$. Since $\rho(\gamma) \in SL(N,\Z)$, this action is well-defined, and since $\Gamma$ is a lattice in $G$, the action is properly discontinuous. Then the topological quotient space $(G \times \T^N)/\sim$ is then canonically diffeomorphic to $G_\rho / \Gamma_\rho$.

We then consider the translation action of an abelian subgroup $A \subset G_\rho$. For any abelian set $A \subset G_\rho$ and $a=(g_a,v_a)\in A$, set $p_\rho(a)=p(g_a)$, where $p$ is the map defined in Section \ref{sec:9}. Then there exists an element $s\in G_\rho$ such that the coarse Lyapunov distributions for the action of $sAs^{-1}$ is the same as those for $p_\rho(A)$  (see Proposition \ref{prop:diagonalize1}).
Let $\Phi_{A,\rho}$ denote the \emph{restricted weights of $G$ with respect to $p_\rho(A)$}, associated with the decomposition of $\RR^N$ into eigenspaces under the representation $\rho(p_\rho(A))$. If $p_\rho(A)=A_0$, $\Phi_{A,\rho}$ is the standard weights.  Then the Lie algebra
$\mathfrak{g}_\rho$ of $G_\rho$ decomposes
\begin{align*}
\mathfrak{g}_\rho=\mf g_0 + \sum_{r \in\Delta_A}\mathfrak{g}_r +\sum_{\mu\in\Phi_{A,\rho}}\mathfrak{e}_\mu
\end{align*}
where $\mathfrak{e}_\mu$ is the weight space of $\mu$ and
$g_0$ is the Lie algebra of the centralizer $C_{G_\rho}(A)$ of $A$ in $G_\rho$.

Note that $r$ or $\mu$ may appear in the set of both restricted roots and weights. For any $r\in \Delta_A\cup \Phi_{A,\rho}$ let $\mathfrak{g}^{(r)}=\sum_{k\in \R_+} \left(\mathfrak{g}_{kr}+\mathfrak{e}_{kr}\right)$ and
$U_{[r]}$ be the corresponding subgroup of $G_\rho$. Then these subalgebra $\mathfrak{g}^{(r)}$ form coarse Lyapunov distributions and (double) cosets of these subgroups $U_{[r]}$ form coarse Lyapunov foliations of $\alpha_A$, which coincide with those of $\alpha_{p(A)}$ (see Proposition \ref{prop:diagonalize}).

Elements of
$A\backslash p^{-1}(\bigcup_{\phi\in\Delta_A\cup\Phi_{A,\rho}\setminus \set{0}}\ker(\phi))$
are \emph{regular} (for $A$) and connected components of the set of regular
elements \emph{Weyl chambers} (for $A$). If $A=A_0$, the left translations  of $A$ on $ G_\rho/\Gamma_\rho$  is sometimes referred to as  {\em full Cartan action}.

\section{Statement of Main Results}
\label{sec:statement}
\subsection{Smooth Rigidity and Factor Actions}\label{sec:13}
Let $\mc G / \Lambda$  denote $G/\Gamma$ as in
symmetric space examples or  $G_\rho/\Gamma_\rho$ as in twisted symmetric space
examples.
\begin{definition}\label{de:3}
If $i_0 : \R^k \times \Z^l \to \mc G$ is a homomrophism, the  action $\alpha$ of $\R^k \times \Z^l$ defined by $\alpha(a)(g\Lambda) = (i_0(a)g)\Lambda$  on $\mc G / \Lambda$ will be referred to as a {\em higher-rank translation} or just a {\em translation} for short.
\end{definition}

We let $A$ denote the image of $\R^k \times \Z^l$ in $\mc G$, and sometimes use identify $\R^k \times \Z^l$ with $A$. We use $\alpha_A$ or $\alpha_{i_0}$ to denote the translation action if we wish
to emphasize which translation we take.
A {\em standard perturbation} of the action  $\alpha_{i_0}$ is an action $\alpha_i$ where $i : \ZZ^k\times\RR^{l}\rightarrow N$ is a homomorphism close to $i_0$, where $N$ is the centralizer of $p(A)$ or $p_\rho(A)$, respectively.

\begin{remark}
In the hyperbolic situation any  standard perturbation is  simply a time change
corresponding to an automorphism of the acting group  but in the partially hyperbolic cases  standard perturbations are usually essentially different from each
other.
\end{remark}

\begin{definition} \label{def:rank1-factor} An action $\alpha'$ of $\R^k \times \Z^l$ on a manifold $\mathcal{M}'$ is a factor of an action $\alpha$ of $\R^k \times \Z^l$ on $\mathcal{M}$ if there exists an
epimorphism $h:\mathcal{M}\rightarrow \mathcal{M}'$ such that $h\circ\alpha=\alpha'\circ h$.

An action $\alpha'$ is a (vitrually) rank one factor if it is an factor and if $\alpha'(\R^k \times \Z^l)$
contains a rank one subgroup $S$ such that the coarse Lyapunov foliations for $\alpha'(\R^k \times \Z^l)$ is the same as for action $S$.
\end{definition}

\begin{definition}An action $\alpha_{i_0}$ of  $\ZZ^k\times\RR^l$ on $M$ is $C^{k,r,\ell}$ {\em locally
rigid} if any $C^k$ perturbation $\tilde{\alpha}$ which is
sufficiently $C^r$ close to $\alpha_{i_0}$ on a compact generating
set is $C^\ell$ conjugate to a standard perturbation
$\alpha_{i_0}$.
\end{definition}

\begin{remark} It is immediately obvious that if $\alpha_{i_0}$ is locally rigid then the same is true for any time change obtained by an automorphism of $\R^k \times \Z^l$. Hence the notion of local rigidity for $\alpha_{i_0}$  depends only of the subgroup  $A$.
\end{remark}

\subsection{Genuinely higher rank  and almost semisimple actions}We use $\mathcal{G}$ to denote $G$ or $G_\rho$. Let $\alpha_A$ be a translation on $M$. Then $\alpha_A$ lifts to an $A$ action on $\mathcal{G}$ by left translations.
 Suppose $\alpha_A$ satisfies property

\begin{center}
 \hypertarget{star} ($*$) there is no rank one factor of the corresponding lifted algebraic action on $\mathcal{G}$.
\end{center}

\noindent  If $\alpha_A$ satisfies property \starref, then we say $A$ is \emph{genuinely higher rank}.
 \begin{remark}
Condition \starref \, shows that for both spaces the group $G$ has neither compact factors, nor rank one factors. Furthermore, it also implies that the twisted spaces
$\rho$ has no trivial summands.
\end{remark}

\subsection{The Main Theorems}
\begin{definition}\label{de:2}
We say that an antisymmetric bilinear form $\omega$ on $\RR^N\times \RR^N$  \emph{invariant for the representation $\rho$}, if
\begin{align*}
 \omega(\rho(g)w_1,\rho(g)w_2)=\omega(w_1,w_2),\mbox{ for every }\,w_1,\,w_2\in \RR^N\text{ and }g \in G,
\end{align*}
We say that a twisted symmetric space obtained from $G_\rho$ is \emph{non-symplectic} or {\it has no symplectic contributions} if there is no nontrivial, antisymmetric bilinear form on $\mf e$  invariant for $\rho$. Otherwise, we say that the twisted symmetric space is \emph{symplectic} or  has a {\it symplectic contribution}.
\end{definition}
\begin{theorem}
\label{thm:main}\hspace{1cm}Suppose $\alpha_A$ is genuinely higher-rank. Then:
\begin{enumerate}

  \item  $\alpha_A$ is $C^{\infty,1,\Hol}$ locally rigid on symmetric spaces and non-symplectic twisted symmetric spaces. Furthermore, if all roots of
$\Delta_{A_0}$ are detected by $A$ and $0 \not\in \Phi_{A,\rho}$, $\alpha_A$ is $C^{\infty,1,\infty}$-rigid.

  \smallskip
   \item $\alpha_A$ is $C^{\infty,1,\Hol}$ locally rigid on symplectic twisted symmetric spaces if $0\notin \Phi_{A,\rho}$. Furthermore, if all roots of
$\Delta_{A_0}$ are detected by $A$, $\alpha_A$ is $C^{\infty,1,\infty}$-rigid.
  \end{enumerate}
\end{theorem}
\begin{remark}
The proof of smoothness of conjugacy relies on finding elements which act isometrically on each coarse Lyapunov foliation. Detection of all roots guarantees semisimplicity of the action, as well as all of its perturbations (Lemma \ref{lem:isometric}). Once a root or weight becomes undetected, the perturbations may fail act semisimply. The authors plan to address this issue in a subsequent paper, where smoothenss will be investigated in the presence of Jordan blocks.
\end{remark}

Let $\mathcal{G}_0$ denote the universal Lie central extension  of $\mathcal{G}$ (see Definition \ref{def:univ-lie}) and let $\Gamma'$ (resp. $\Gamma_\rho'$) be the lattice in $\mathcal{G}_0$ that covers $\Gamma$ (resp. $\Gamma_\rho$, see Lemma \ref{lem:lattice-lift}). Let $Z$ denote the center of $\mathcal{G}_0$ that lies in the radical of $\mathcal{G}_0$. For each $a \in A$ we can associate an element $\widehat{a} \in \mc G_0$. Then define $\widehat{\alpha}_A(a)$ be the left multiplication map on $\mc G_0$ (see \seref{se:5} and \seref{se:7} of Section \ref{sec:14}).
\begin{definition}\label{def:quasi-nil}
Let $\widetilde{\alpha}_A$ be a $C^\infty$ perturbation of $\alpha_A$ that is $C^1$-sufficiently close to $\alpha_A$. Fix a compact generating set $\mbb A$ of $A$. We say that $\widetilde{\alpha}_A$ admits  a {\it quasi-nil extension} (on $\mbb A$) if $\widetilde{\alpha}_A(a)$ can be lifted to a diffeomorphism $\widehat{\alpha}_A(a)$ on $\mc G_0$ for each $a \in \mbb{A}$ such that:
\begin{enumerate}[(i)]
\item $\widehat{\alpha}_A(a)$ commutes with the (right) $\Gamma'$ (resp. $\Gamma_\rho'$) and $Z$ actions.

\medskip
\item

If $\widehat{W}^s_a$ is the stable foliation for $\widehat{\alpha}_A(a)$, then $\widehat{W}^s_a$ is invariant under $\widehat{\alpha}_A(b)$, $b \in \mbb A$.

\medskip
\item $\widehat{\alpha}_A(a)$ is $C^1$-sufficiently close to $\widehat{\alpha}_A(a)$ for each $a \in \mbb{A}$.
\end{enumerate}
\end{definition}

\begin{remark}
Every homogeneous perturbation admits a quasi-nil extension. The terminology quasi-nil extension comes from the fact that in the homogeneous setting, the group generated by the diffeomorphisms $\widehat{\alpha}_A(a)$ is either isomorphic to $A$, or a 1-step nilpotent extension of $A$.
\end{remark}

\section{History of the  rigidity problem}
\subsection{Rigidity of hyperbolic  actions}
 Differentiable rigidity of higher rank  algebraic Anosov actions including Weyl chamber flows and twisted Weyl chamber flows was proved in the mid-1990s   \cite{ks97}.
The proof consists of two major parts:

 (i)  An  {\em a priori regularity}  argument  that shows  smoothness of the Hirsch-Pugh-Shub orbit equivalence
\cite{shub}. The key part of the argument is the  theory of  non-stationary normal forms developed in \cite{Guysinsky}.
`11
 (ii) {\em Cocycle rigidity} used to ``straighten out'' a time change; it is proved     by a harmonic analysis method in \cite{Spatzier1}.

All these actions satisfy condition $(*'')$:  {\em there are no rank one factors.}

\subsection{Difference between hyperbolic and partially hyperbolic actions}\label{sec:2.2}
The next step is to consider  algebraic {\em partially hyperbolic
actions}. Unlike the hyperbolic case, the a priori regularity method is not directly
applicable here since individual elements  of such actions are not
even structurally stable. In addition to their work on  hyperbolic
rigidity,  A. Katok and R. Spatzier also considered  cocycle
problem  for certain partially hyperbolic actions in
\cite{Spatzier2} and proved essential cocycle rigidity results.

Before  proceeding with the chronological account let us  explain an essential  point that also plays a role in the present work.

In the hyperbolic case smooth orbit rigidity  reduces the local differentiable rigidity problem to rigidity of vector valued cocycles.  For, the expression of the {\em old time}, i.e. that  of the unperturbed action through the {\em new time}, i.e. that  of the perturbed action is a cocycle over the {\em unperturbed action}   with values in the acting group that is  a  vector space or its  subgroup. This is the scheme of \cite{ks97}. In other words,  smooth orbit rigidity  reduces the differentiable conjugacy problem to a time change problem.

In the partially hyperbolic cases instead of smooth orbit rigidity  one may hope at best  to have
{\em smooth rigidity of neutral foliation} when the scheme of  \cite{ks97} is applicable. For that  one needs the following property additional to condition $(*)$:\medskip

\noindent ($\mathfrak B$) {\em The stable directions of various
action elements  or, equivalently,  coarse Lyapunov
distributions,(see
Section~\ref{sec:4}),
together with the orbit direction, generate  the tangent space as a
Lie algebra, i.e. those distributions and  their brackets of all
orders generate the Lie algebra linearly.} \medskip

Still even  after smooth rigidity of the orbit foliation  has been  established,  the problem of differentiable conjugacy   is not reduced to a cocycle problem over the unperturbed action; rather it reduces to the cocycle problem over the
perturbed action; furthermore, the values of the  cocycle may be in a non-abelian group: it is actually the exponential of the  central (neutral) distribution.

\subsection{Previous work on partially hyperbolic rigidity}
\subsubsection{KAM  method }
In Section \ref{sec:2} if we set $H=\R^k\ltimes\RR^m$ then we get the suspension of a toral automorphism action $\Z^k \curvearrowright \T^m$. The case of actions by commuting partially hyperbolic automorphisms of a torus satisfying condition $(*'')$ is considered in earnest in  2004 \cite{DD-AK-ERA} (complete proofs appeared in print).
 There, because the phase space of the
Poincar\'{e} section is abelian, the algebraic features of the stable distributions are
simpler. However, this comes at a cost, as the stable distributions will commute.
In particular, condition $(\mathfrak B)$ fails in this case. To handle this  problem a
 new method was introduced: a KAM type iteration scheme formally similar to that employed by J.Moser \cite{Moser} in the higher-rank version of a conventional  setting for  application of KAM scheme -- diffeomorphisms of the circle. The method was later applied in \cite{Damjanovic3} to obtain weak local rigidity for parabolic actions. The KAM method is powerful when the calculation in the dual spaces can
be carried out in an explicit way. But generally, the dual space is too complex to carry
out corresponding calculation. This drawback restricts the application only to some special models.

\subsubsection{The Computational Geometric method}
So far, the most effective method to treat partially hyperbolic symmetric space examples is the geometric method developed in developed in \cite{dk2005,dk2011}. This method bypasses subtleties of analysis and representation theory altogether.  It builds upon the  already mentioned  observation that  the problem of differentiable conjugacy   reduces to the cocycle problem over the  perturbed action. In this case the condition \starref \, implies an even stronger property ($\mathfrak B'$):

\medskip

\noindent ($\mathfrak B'$) {\em The stable directions of various  action elements   generate  the tangent space as a Lie algebra.}
\smallskip

Solution of coboundary equations  for actions are built  along broken  paths  consisting of pieces of stable foliations for different elements of the action. It was first observed in \cite{dk2005} that consistency of  such a construction follows in certain cases from description of generators and relations in the ambient Lie group. In \cite{dk2011} another key  ingredient  was introduced: under certain circumstances the web of Lyapunov foliations is so robust that the construction of solutions of the  coboundary equation carries out to the {\em perturbed action} thus providing  a solution of the conjugacy problem in the H\"older category. After that smoothness of the conjugacy is established by a priori regularity method as in \cite{ks97}.

In \cite{dk2005,dk2011} a special case $G=SL(n,\KK)$ $(n\geq
 3)$, $\KK=\RR$ or $\CC$  is considered. In this case  necessary algebraic information  can be extracted from the classical work of Steinberg, Matsumoto and Milnor, see \cite{Steinberg, Matsumoto, Milnor}. In these works, the groups $K_2^{\mbox{\tiny alg}}(\R)$ and $K_2^{\mbox{\tiny alg}}(\C)$ appear, so this method is sometimes referred to as the {\it algebraic $K$-theory/Geometric method}.

The approach of \cite{dk2005,dk2011} was further employed
in \cite{damjanovic07}, \cite{zwang-1}, \cite{zwang-2} and \cite{vinhage15} for
extending cocycle rigidity and differentiable
 rigidity from
$SL(n,\RR)/\Gamma$ and $SL(n,\CC)/\Gamma$ to compact homogeneous
spaces obtained from semisimple Lie
groups.  A great strength of this method is that it requires only $C^2$ closeness for
the perturbation, unlike the KAM scheme where  number of derivatives is  large and dependent on the data.
But this comes with a price. All previous works required the following key conditions
in order to carry out this scheme: \smallskip
\begin{enumerate}
  \item [($\mathfrak C$)] detailed information  about specific generators and relations (algebraic $K$-theory) in the ambient group is required, and
  \smallskip
  \item [($\mathfrak *'$)] \hypertarget{starp} coarse Lyapunov  distributions for the partially hyperbolic  restriction should be the same as for the ambient Cartan action.
\end{enumerate}

\newcommand{\starpref}{\hyperlink{starp}{($*'$)}}

This drawback restricts the method can only be used to treat ``generic restrictions'', i.e., actions satisfy condition \starpref, which is more restrictive than condition \starref \, (see \cite{dk2005, dk2011, damjanovic07, zwang-1,vinhage15}). The only exception appears in \cite{zwang-2}, where a very detailed study of the generators and describing is carried out. Because the commutator relations are highly dependant on the subgroup one chooses, there is no uniform computation for  non-generic restrictions.

Furthermore, no nontrivial progress using this method has been made for the partially  hyperbolic twisted space examples. Since until recently both ($\mf C$) and \starpref \, were needed to employ the geometric method,
 the only cases treated were symmetric space examples and twisted cases with all weights detected and not negatively proportional. In \cite{Spatzier1}, twisted symmetric space examples are considered, but only in the Anosov setting: this method fails to treat time changes of partially hyperbolic algebraic actions. There are two main obstacles in extending
the rigidity results for twisted Anosov examples to partially hyperbolic algebraic actions. Geometrically, for partially hyperbolic actions, a weight of the representation may vanish on the acting group, meaning that one needs to study not only the relationships of the coarse Lyapunov foliations in the fibers, but also their interactions with the Lyapunov foliations of the semisimple part.
Algebraically, detailed $K$-theory is not available from literature, unlike the symmetric space models, so we may not rely on any method which would employ condition ($\mf C$).

\subsubsection{The Topological Geometric Method}

We expand the rough scheme used in \cite{dk2011} and \cite{vinhage15} in key and subtle ways. The first feature is our treatment of nongeneric actions. All previous results following the broad scheme used in this paper applied only to restrictions of Weyl chamber flows. These are special cases of homogeneous actions where the homomorphism $i : A \to G$ determining the action has $i(A)$ a subset of a split Cartan subgroup. Furthermore, the genericity assumption \starpref made certain computations possible in special cases. In particular, in the case of $SL(d,\R)$, these calculations coincided with those done by Milnor in his study of $K_2(\R)$. The computations are carried out in split or complex algebraic groups by Steinberg, but no uniform approach can reproduce the calculations in a general higher-rank Lie group.

Later, a topological argument showed that a certain structural aspect of these calculations was sufficint to show local rigidity \cite{vinhage15}. That is, that a certain group extension defined via generating relations among unipotent subgroups was a perfect central extension. While this extended the class of groups to which one can apply the geometric method, it did not improve the genericity assumption on actions for which local rigidity was known. The principal difficulty we must overcome, then, is to show that this structural aspect still holds even when certain relations or generators are missing. This requires a careful look at Steinberg and Deodhar's proofs, as well as new ideas on how to effectively generate and present the ``missing'' generators and relations.


Our techniques also provide a new class of examples: we consider the case of semidirect products, which correspond to the geometry of twisted symmetric spaces. The class of semidirect products we consider is very broad: we do not require the defining representation to be Anosov, nor do we require non-resonance with the roots, key assumptions in many rigidity results. Because we work so broadly, we must take great care in understanding the relationships between the coarse Lyapunov foliations of the semisimple part and the representation part, which, without non-resonance and Anosov assumptions, would be extremely computationally involved when proceeding brute-force. We employ entirely new arguments and analysis, which are developed in Section \ref{sec:16} and Section \ref{sec:new}.

\subsubsection{Difference between genuinely higher rank and generic acions}

 The principal difficulty in studying actions satisfying condition \starref \, but not \starpref \,
is related to the complexity of the restricted root (and/or weight) systems for arbitrary
abelian subgroups, for which little is known in the literature (unlike the case of a split Cartan
subgroup).\footnote{With few exceptions. For the case treated in \cite{zwang-2}
 the restricted root systems of $SL(2n,k)$, $k=\RR$ or $\CC$ behaves like that of $Sp(2n,k)$.} In most cases, the restricted root and weight systems of general abelian groups do
not have the usual structures of Cartan root systems, like a canonical inner product. While a ``Weyl group'' can be constructed, it is not
guaranteed to act transitively on the Weyl chambers (see Section \ref{sec:6}). In
the symmetric spaces it is possible to handle the geometric properties in some cases
by using generating relations, but a complete system of generators and manageable
relations in twisted spaces are too complicated to compute even for the split Cartan
subgroup. In both  symmetric space examples and
twisted symmetric space examples  coarse Lyapunov distributions
of subgroups of split Cartan may be quite different from those of full Cartan
actions.

\subsection{Future Directions}
In current paper the genuinely higher rank condition is necessary for the
geometric method. For non-genuinely higher rank actions, algebraically, the free free product group generated by commutator relations is no longer a central extension of $\mc G$; geometrically, cycles of the same homotopy classes
can't be reduced to each other using allowable substitution of stable cycles.

The authors are working to extend the rigidity results to non-genuinely higher rank actions and the result will appear in a separate paper. This is not surprising
that this is already the case for standard Anosov action \cite{Spatzier1} and \cite{ks97}. If the lattice $\Gamma$ is irreducible, then  any higher rank actions on $M$ have no rank-one factors and the regular representations of the translations on the homogeneous space have a spectral gap, even if $G$ has rank-one factor of non-compact type. This is a deep result of representation theory. 

\section{Geometric Preparatory step I: Lyapunov Structures and Central Extensions}
\label{sec:geom1}
In this part from Section \ref{sec:4} to \ref{sec:11} we introduce general notations and results for partially hyperbolic actions. After making a detailed study of coarse Lyapunov foliations for general abelian homogeneous actions in Section \ref{sec:prelims}, we give necessary and sufficient conditions for simply connected twisted spaces to admit non-trivial central extensions in Section \ref{sec:16}.

\subsection{Coarse Lyapunov foliations}\label{sec:4} Let $\mathcal{B}$ be a compact manifold. A $C^1$ diffeomorphism $f: \mathcal{B}\to \mathcal{B}$ is called
\emph{partially hyperbolic} if  there is a continuous invariant splitting of the tangent bundle $T\mathcal{B}$:
\begin{align*}
 T\mathcal{B} = E^s_a\oplus E^0\oplus E^u_a
\end{align*}
and constants $C>0$, $0<\lambda_-<\widetilde{\lambda}_-<\widetilde{\lambda}_+<\lambda_+$,  $\lambda_-<1<\lambda_+$ such that for $n\geq0$:
\begin{align*}
 \norm{Df^nv}&\leq C\lambda_-^n\norm{v},\qquad v\in E^s_a,\\
 \norm{Df^{-n}v}&\leq C\lambda_+^{-n}\norm{v},\qquad v\in E^u_a,\\
 \norm{Df^nv}&\leq C\widetilde{\lambda}_+^n\norm{v},\qquad v\in E^0_a,\\
 \norm{Df^{-n}v}&\leq C\widetilde{\lambda}_-^{-n}\norm{v},\qquad v\in E^0_a.
\end{align*}
Now assume in addition $E^0$ is a continuous distribution uniquely
integrable to a foliation $\mathcal{N}$ with smooth leaves, then $\alpha(a)$ is \emph{uniformly normally
hyperbolic with respect to $\mathcal{N}$}.

Let $\alpha: A \rightarrow
\text{Diff}(\mathcal{B})$ be an abelian action of $A$ on a compact
manifold $\mathcal{B}$ by diffeomorphisms of $\mathcal{B}$. If there exists $a\in A$ such that $\alpha(a)$ is uniformly normally
hyperbolic with respect to $\mathcal{N}$, we will call such an action $\alpha$ {\it $\mathcal{N}$-normally hyperbolic}.  Elements in $A$ which are uniformly normally
hyperbolic with respect to $\mathcal{N}$ are called regular. Let
$\widetilde{A}$ be the set of regular elements. We call an action a
partially hyperbolic $A$ action if the set $\widetilde{A}$ is dense
in $A$. In particular, if $E^0$ is the tangent distribution to the
orbit foliation of a normally hyperbolic action, then the action is
called {\it Anosov} or {\it hyperbolic}.

Then we have an $\alpha$-invariant continuous splitting of the tangent bundle $T\mathcal{B}$ for partially hyperbolic $A$-action $\alpha$:
\begin{align*}
 T\mathcal{B} = \bigoplus_{1\leq i\leq r} E_i\oplus E^0
\end{align*}
where $E_i$ runs over all maximal nontrivial intersection $\bigcap_{1\leq j\leq m} E^s_{b_j}$ for $b_1,\cdots,b_m\in \widetilde{A}$. We call this decomposition the \emph{coarse Lyapunov decomposition} of $T\mathcal{B}$. Note that $E_i$ is tangent to the foliation $\mathcal{F}_i:=\bigcap_{1\leq j\leq m} \mathcal{W}^s_{b_j}$ with
$C^\infty$ leaves, where $\mathcal{W}^s_{b_j}$
is the stable foliation for
$\alpha(b_j)$. These are the {\it coarse Lyapunov foliations} of $\alpha$. Given a foliation $\mathcal{F}_{i}$ and $x\in \mathcal{B}$ we denote by
$\mathcal{F}_{i}(x)$ the leaf of $\mathcal{F}_{i}$ through $x$.
\begin{remark}
Notations in Section \ref{sec:2} coincide with those in this part for general partially hyperbolic actions.
\end{remark}

\subsection{Paths and cycles for a collection of foliations}\label{sec:5}
 A $\mc F$-{\it Lyapunov path} is a sequence of points $\tau = (x_0,x_1,\dots,x_n)$ such that $x_{i+1} \in \mathcal{F}_{j(i)}(x_i)$ for every $i$. The sequence $(\chi_{j(0)},\dots,\chi_{j(n-1)})$ is called the {\it combinatorial pattern} of $\tau$. The path is called a {\it cycle} if $x_n = x_0$. Note that the space of cycles with a fixed base point carries a canonical product, defined by
 \begin{align*}
   \tau^{(1)} * \tau^{(2)} = (x_0,x_1^{(2)},\dots,x_{n_2-1}^{(2)},x_0,x_1^{(1)},\dots,x_{n_1-1}^{(1)},x_0).
 \end{align*}
 Let $\mc P(\mc F)$ denote the space of paths in the foliation family $\mc F$ (with an understood fixed base point $x_0$), and $\mc C(\mc F)$ the space of cycles.

 \subsection{The Pugh-Shub-Wilkinson H\"older leaf conjugacy}\label{sec:11}
A map $f:M\to M$ is $\theta$ H\"{o}lder if
\begin{align*}
 C=\sup_{x\neq x'}\frac{d(f(x),f(x'))}{d(x,x')}<\infty.
\end{align*}
we will refer to the number $C$ as to the $\theta$-H\"{o}lder norm.

For the translation $\alpha_A$ denote its smooth neutral foliation by $\mathcal{N}$ and the family of coarse Lyapunov foliations by $\mathcal{U}_1,\cdots, \mathcal{U}_r$. The following is an immediate
corollary of the general stability result of Hirsch, Pugh and Shub for diffeomorphisms which are partially hyperbolic with a smooth neutral foliation (see Section $7$ of \cite{shub} and Theorem $A$ of \cite{psw12}).
\begin{theorem}\label{th:1}
Let $\widetilde{\alpha}_A$ be a $C^r$-small $C^\infty$ perturbation of $\alpha_A$, for some $r\geq1$.
then:
\begin{enumerate}
  \item $\widetilde{\alpha}_A$ is a partially hyperbolic action with a dense set of regular elements in $A$
which all have a common neutral foliation $\mc N'$ with $C^r$ leaves.

\smallskip
  \item Coarse Lyapunov foliations $\mathcal{F}_1,\cdots, \mathcal{F}_r$ of $\widetilde{\alpha}_A$ are continuous foliations with
smooth leaves which are leafwise $C^r$-close to the coarse Lyapunov foliations $\mathcal{U}_1,\cdots, \mathcal{U}_r$ of $\alpha_A$.

\smallskip
  \item There is a bi-$\varrho$-H\"older homeomorphism $\mbf h : M \to M$, $C^0$ close to $id_M$, which maps leaves of $\mc N'$ to leaves of $\mc N$: $\mbf h(\mc N') = \mc N$.  Furthermore, the constant $\varrho$ tends to $1$ and the $\varrho$-H\"{o}lder norm of $\mbf h$ is uniformly bounded as $\widetilde{\alpha}_A$ tends to $\alpha_A$ in the $C^r$-topology.
\end{enumerate}
\end{theorem}

We call $\mbf h$ the {\it Pugh-Shub-Wilkinson conjugacy}. Furthermore, $\mbf h$ can be chosen smooth and $C^1$ close to the identity along the leaves of $\mathcal{N}$ although we will not use the latter fact.

\subsection{Course Lyapunov foliations for $\alpha_A$}\label{sec:prelims}
In this section we will make a detailed study of the course Lyapunov foliations for $\alpha_A$ in both spaces. We continue letting $G$ denote a semisimple group without rank one factors and $G_\rho$ denote a semidirect product. We will see that it will be sufficient to analyze the Lyapunov structures for actions which are restrictions of Weyl chamber flows (ie, restrictions of the action of a Cartan subgroup).

\begin{proposition}\label{prop:diagonalize}
For any abelian subgroup $A$ of $G$, there exists a split Cartan subgroup $A_0$ and a homomorphism $p:A\to A_0$ such that the coarse Lyapunov distributions for $\alpha_A$ is the same as for $\alpha_{p(A)}$.
\end{proposition}
\begin{proof}
Under the adjoint representation of $G$ for any $a \in A$,  $\text{Ad}(a)$ with corresponding Jordan-Chevalley normal form decomposition of $\text{Ad}(G)$ can be written as
\begin{align}\label{for:2}
  \text{Ad}(a)= \exp(Z_a)\exp(X_a)\exp(Y_a)
\end{align}
for $3$ commuting elements, where $Z_a$ is compact, $X_a$ is $\R$-semisimple, $Y_a$ is nilpotent (see, eg, \cite[Proposition 2]{goto78}).

We claim that $\exp(X_a)$ and $\exp(X_b)$ also commute for every $a,b \in A$. Indeed, we can consider the decomposition of $\mf g$, the Lie algebra of $G$ into generalized eigenspaces for the action of $\text{Ad}(a)$, and for the action of $\exp(X_a)$. By assumption $\exp(X_a)$ is a scalar multiple of identity on each generalized eigenspace of $\text{Ad}(a)$. Since $b$ commutes with $a$, $\exp(X_b)$ preserves each generalized eigenspace of $\text{Ad}(a)$ and acts by a scalar multiple in each, which implies that
$\exp(X_a)$ and $\exp(X_b)$ commute. Then there exists a split Cartan subgroup $C$ in $\text{Ad}(G)$ containing all $\exp(X_a)$, $a\in A$. Let $A_0$ be the split Cartan subgroup in $G$
 whose image under the adjoint representation is $C$. Then $\Ad$ is an isomorphism restricted to $C$, since $A_0\cap Z(G)=e$. So we may define the map $p(a) = \text{Ad}^{-1}(\exp(X_a))$. Then we get the conclusion immediately since the coarse Lyapunov distributions for $\alpha_A$ is determined by those of the
adjoint representation of $\text{Ad}(A)$ on $\mathfrak{g}$.

The decomposition \eqref{for:2} also implies that we have the corresponding normal form for $a$:
\begin{align}\label{for:3}
 a=s_ak_an_a
\end{align}
where $s_a=\exp(X_a)\in G$ is $\R$-semisimple, $k_a$ is compact and $n_a=\exp(Y_a)\in G$ is nilpotent. It is clear that $s_a$ and $n_a$ commute and $k_as_a=s_ak_az_1$ and $k_an_a=n_ak_az_2$, where $z_1,\,z_2\in Z(G)$. Since $Z(G)$ is finite, there exists $n\in N$ such that $k_as_a^nk_a^{-1}=s_a^n$, which implies that $\text{Ad}(k_a)(X_a)=X_a$. This shows that $k_a$ and $s_a$ commute. Similarly, we get $k_an_a=n_ak_a$.

So we conclude that $p(a)=s_a$ and $\text{Ad}(k_an_a)$ has all eigenvalues of  modulus $1$. So we get the result.

\end{proof}
Using \eqref{for:3} we set  $\varphi(a)=\text{Ad}(k_an_a)$ for any $a\in A$. Then $\varphi$ is a homomorphism since $\varphi(a)$ commutes with $\text{Ad}(p(b))$ for all $a,b \in A$. We get this by observing that $\text{Ad}(p(b))$ acts by homotheties on sums of generalized eigenspaces of $\text{Ad}(a)$. Then we get the next result immediately:

\begin{corollary}\label{cor:1}
For any abelian subgroup $A$ of $G$, let $N$ be the centralizer of $A$ in $G$.  Then there exists a homomorphism $\varphi: A\to GL(\Lie(N))$, such that $\Ad(a)$ on $\Lie(N)$ is the same as $\varphi(a)$ on $\Lie(N)$ for any $a\in A$.
\end{corollary}

We now make similar observations for the twisted spaces $G_\rho$.

\begin{proposition}\label{prop:diagonalize1}
For any abelian subgroup $A$ of $G_\rho$, there exists a split Cartan subgroup $A_0$ of $G$, a homomorphism $p:A\to A_0$ and an element $s\in G_\rho$ such that the coarse Lyapunov distributions for $\alpha_{sAs^{-1}}$ is the same as for $\alpha_{p(A)}$.
\end{proposition}
\begin{proof}
For any $c\in G_\rho$ we have the expression $(g_c, v_c)$, where $g_c\in G$ and $v_c\in\RR^N$. Then the set $g_A=\{g_a, a\in A\}$ is an abelian subgroup of $G$. From
the decomposition \eqref{for:3} in Proposition \ref{prop:diagonalize} we see that for any $a\in A$ the generalized eigenspaces for $\rho(g_a)$ is determined by those of $p(g_A)$ where $p$ is as defined in Proposition \ref{prop:diagonalize}.

Recall notation in Section \ref{sec:3}. Choose a regular element $a\in A$. Then $a=(g_a, v_a)$. Note that $I-\rho(g_a)$ on the subspace $W=\sum_{0\neq\phi\in \Phi_{A,\rho}} \mathfrak{e}_\phi$ is invertible. Denote by $(I-\rho(g_a))^{-1}\mid_W$ the inverse map on $W$. We assume that $\mathfrak{e}_0=\{0\}$ if there is no $0$ weight in $\Phi_{A,\rho}$.

We have a decomposition of $v_a=\sum_{\phi\in \Phi_{A,\rho}} v_{a,\phi}$, where $v_{a,\phi}\in\mathfrak{e}_\phi$. Set $u=\sum_{0\neq\phi\in \Phi_{A,\rho}}v_{a,\phi}$. Then $\rho(g_a)u\in W$. Set $u'=-(I-\rho(g_a))^{-1}|_W(\rho(g_a)u)$ and $s=(g_a,u')$. Then by easy computation we have
\begin{align*}
  sas^{-1}=\big(g_a,\rho(g_a)v_{a,0}\big).
\end{align*}
Note that $\rho(g_a)v_{a,0}\in \mathfrak{e}_0$. Next we will show that $v_{sbs^{-1}}\in \mathfrak{e}_0$ for any $b\in A$. It suffices to show that for any $c\in G_\rho$, if $c$ commute with $\big(g_a,\rho(g_a)v_{a,0}\big)$ then $v_c\in \mathfrak{e}_0$. Since $c$ commute with $\big(g_a,\rho(g_a)v_{a,0}\big)$ then $g_c$ commute with $g_a$ and $v_c$ should satisfy the following identity:
\begin{align*}
 \rho(g_c)^{-1}\big(\rho(g_a)v_{a,0}\big)-\big(\rho(g_a)v_{a,0}\big)=\rho(g_a)^{-1}(v_c)-v_c.
\end{align*}
Writing $v_c=u_1+u_2$ where $u_1\in W$ and $u_2\in \mathfrak{e}_0$ and noting that $\rho(g_c)^{-1}$ preserve $\mathfrak{e}_0$, we have
\begin{align*}
 \rho(g_a)^{-1}(u_1)-u_1=0.
\end{align*}
This follows from that fact that the left side the above equation is in $\mathfrak{e}_0$, then the projection of the right side to $W$ should be $0$. Since
$I-\rho(g_a)^{-1}$ is invertible on $W$ then we see that $u_1=0$. This shows that $v_c\in \mathfrak{e}_0$. Then any element in $sAs^{-1}$ has the decomposition into two commuting terms: \begin{align*}
 (p(s_{g_a}),0)(k_{g_a}n_{g_a}, v_a),\qquad \forall a\in A.
\end{align*}
where $v_a\in \mathfrak{e}_0$ (see notations in \eqref{for:3}). Then we finish the proof.

\end{proof}
From \eqref{for:11} we see that $\rho(g)v=gvg^{-1}$ for any $g\in G$ and $v\in \RR^N$. Then similar to Corollary \ref{cor:1}, we get the next result by letting $\varphi(a)=\text{Ad}(k_an_a, v_a)$:
\begin{corollary}\label{cor:2}
For any abelian subgroup $A$ of $G$, let $s$ be the element as described in Proposition \ref{prop:diagonalize1}. Let $N$ be the centralizer of $sAs^{-1}$ in $G_\rho$.  Then there exists a homomorphism $\varphi: A\to GL(\Lie(N))$, such that $\Ad(sas^{-1})$ on $\Lie(N)$ is the same as $\varphi(a)$ on $\Lie(N)$ for any $a\in A$.
\end{corollary}

\subsection{Central Extensions of Twisted spaces with Symplectic Contributions}\label{sec:16}

 Recall notations in Section \ref{sec:3}.
 For the twisted symmetric space examples, notice that in
$\mathfrak{g}_\rho$
\begin{align}\label{for:9}
[(X_1,0),(X_2,0)]=[X_1,X_2]_\mathfrak{g}\quad\text{ and
}\quad[(X,0),(0,t)]=d\rho(X)t
\end{align}
where $X_1\,,X_2\,,X\in\mathfrak{g}$, $t\in\RR^N$;
$[\cdot,\cdot]_\mathfrak{g}$ denotes the Lie bracket in
$\mathfrak{g}$ and $d\rho$ is the induced Lie algebra representation
on $\RR^N$ from $\rho$. By complete reducibility for representations of
semisimple groups, there is a decomposition
of
\begin{align}\label{for:45}
 \RR^N=\bigoplus_{i\in I}\RR^{N_i},
\end{align}
such that $\rho$ is irreducible over
$\RR$ on each $\RR^{N_i}$ and hence $\rho$ is irreducible over $\RR$
on each $\RR^{N_i}$. Since there is no invariant subspace such that $\Lie(G)$ acts trivially (see assumption about $\Gamma$ in Section \ref{sec:3}) it follows the group $G_\rho$ is perfect.

\begin{definition}\label{de:4}
Let $S$ be an abstract group.
A central extension of $S$ is a pair $(\theta,S')$ where $S'$ is a
group, $\theta$ is a homomorphism of $S'$ onto $S$ and
$\ker\theta\subset Z(S')$.
\end{definition}

We state a result which is very useful for the discussion  (see Section $7$ of \cite{Steinberg2}):
\begin{fact}\label{fact:6}
If $\Psi_1: E\rightarrow F$ and $\Psi_2: F\rightarrow F_1$ are central extensions, then so is $\Psi_1\circ \Psi_2:E\rightarrow F_1$ provided $E$ is perfect.
\end{fact}

Unlike the symmetric space examples, even if $G$ is simply connected, there may exist non-trivial perfect Lie central extensions of the semidirect product  $G\ltimes_\rho\RR^N$. We prove the following:

\begin{proposition}
\label{prop:no-top-ext}
Suppose $G_\rho$ is simply connected. Then it admits a non-trivial perfect Lie central extension if and only if it has a symplectic contribution (see Definition \ref{de:2}).
\end{proposition}

\begin{proof}
First, assume that there is some nontrivial antisymmetric bilinear form $\omega$ on $\RR^N\times \RR^N$ invariant under $G$ (see Definition \ref{de:2}). Set $\mf h=\mf g_\rho \oplus \R Z$ as a vector space. Note that $\mf g_\rho = \mf g \ltimes_{d\rho} \RR^N$. Define a bracket $\set{\cdot,\cdot}$ on $\mf h$ as follows:
\begin{align*}
 \set{X_1 + V_1 + t_1Z,X_2 + V_2 + t_2Z} = [X_1 + V_1,X_2 + V_2]_{\mf g_\rho} + \omega(V_1,V_2)Z.
\end{align*}
where $[,]_{\mf g_\rho}$ is the usual Lie algebra structure on $\mf g_\rho$. Next, we show that $\set{\cdot,\cdot}$ is in fact a Lie bracket on $\mf h$. Antisymmetry of the operator follows from antisymmetry of $[\cdot,\cdot]_{\mf g_\rho}$ and $\omega$. To check that $\{,\}$ satisfies the Jacobi identity  it is sufficient to show that
\begin{align*}
\{\{X, V\}, V'\}+\{\{V, V'\}, X\}+\{\{V', X\}, V\}=0 \mbox{ for all } X\in \mathfrak{g} \mbox{ and } V,\,V'\in \RR^N,
\end{align*}
which is guaranteed by invariance of $\omega$ and the fact that the adjoint representation of $\mathfrak{g}$ is antisymmetric under $\varphi$.

Let $H$ be the simply connected Lie group with Lie algebra $\mf h$. Since  $[\mf h,\mf h]=\mf h$ by construction, $H$ is also perfect. It is clear that the natural projection $j:\mf h\to \mathfrak{g}_\rho$ is central and is a Lie algebra homeomorphism, which induces a homeomorphism $J$ between $H$ and $G_\rho$.
 Let $K=\ker J$ and $K^\circ$ be the connected component of $K$ containing the identity. Then $K^\circ$ is a closed subgroup in the center of $H$ since $j$ is central. Then we have the following sequence:
\begin{align*}
H \xrightarrow{\pi_1} H /K^\circ \xrightarrow{\pi_2} (H /K^\circ)/(K/K^\circ)\cong G_\rho.
\end{align*}
Note that $\pi_1\circ\pi_2=J$ and $\pi_1$ and $\pi_2$ are both central ($\pi_2$ in central since $K / K^\circ$ is discrete in $G_\rho' /K^\circ$). Hence $J$ is central following from Fact \ref{fact:6}. Hence $H$ is a central extension of $G_\rho$.

Now, we prove the converse. Assume that $H$ is a perfect central extension of $G_\rho$ with projection $\Theta$, and $K = \ker \Theta$. Let $H^\circ$ be the connected component of identity in $H$. Since $G_\rho$ is connected $\Theta$ restricted on
$H$ is surjective. This shows that $H=H^\circ\cdot Z$. Then the commutator group $[H,H]=H=[H^\circ,H^\circ]\subset H^\circ$ since $Z$ is in the center. This implies that $H$ is also connected.

Choose a Levi deomposition of $\mf h = \Lie(H)$ as $\mf h = \mf h_{\text{ss}} \ltimes \mf h_{\text{solv}}$.
Then
\begin{align*}
\mf g_\rho = d\Theta(\mf h) = d\Theta(\mf h_{\text{ss}}) \ltimes d\Theta(\mf h_{\text{solv}}).
\end{align*}
Since Levi decompositions are unique up to adjoint conjugation, we may without loss of generality assume that $d\Theta(\mf h_{\text{ss}}) = \mf g$ and $d\Theta(\mf h_{\text{solv}}) = \mf e$. Since $\ker d\Theta \cap \mf g = \set{0}$, we know that $d\Theta|_{\mf h_{\text{ss}}}$ is an isomorphism, and conclude that $\mf h \cong \mf g \ltimes \mf n$, where $\mf n$ is the radical of $\mf h$ which projects onto $\RR^N$. For the adjoint representation of $\mathfrak{g}$ on $\mf n$ we have a decomposition: $\mf n=\sum_i \mf n_i$ such that each $\mf n_i$ is irreducible by complete reducibility of semisimple Lie algebra. It is clear that $\mf n_i\cap \mf z\neq \set{0}$ if and only if $\dim n_i=1$. This shows that $\mf z=\sum_{\dim \mf n_i=1} \mf n_i$ and $d\Theta$ is isomorphic on $\RR^N=\sum_{\dim \mf n_i>1} \mf n_i$. Hence we get $\mf n = \RR^N \oplus \mf z$ as vector space and the adjoint representation of $\mathfrak{g}$ on $\RR^N$ in $\mf h$ is isomorphic to that of $\mathfrak{g}$ on $\RR^N$ in $\mathfrak{g}_\rho$. Finally, perfectness of $\mf h$ implies that $\mf z = [\RR^N,\RR^N]$. In particular, $\mf e$ is nonabelian if $\mf z$ is non-trivial. In $\mf h$ we have
\begin{align}\label{for:50}
 0&=[X, [V_1,V_2]]+[V_1, [X,V_2]]+[[X,V_1],V_2]\notag\\
 &=[V_1,\ad(X)(V_2)]+[\ad(X)(V_1),V_2]
\end{align}
for any $X \in \mathfrak{g}$ and $V_1,\,V_2\in \RR^N$. This shows that the Lie bracket $[\,,\,]_{\mf h} : \RR^N \times \RR^N \to \mf z$ is invariant under the adjoint representation of $\mathfrak{g}$ in $\mathfrak{g}_\rho$. But by projecting $\mf z$ down to a 1-dimensional subspace, we obtain an invariant anti-symmetric $2$-form, and hence a symplectic contribution. Hence we finish the proof.
\end{proof}

Let $\mc I_G \subset \bigwedge^2(\RR^N)^*$ denote the collection of $\rho$-invariant 2-forms on $\R^N$. From the proof of Proposition \ref{prop:no-top-ext}, we see that the following is well-defined:

\begin{definition}
\label{def:univ-lie}
Choose a basis $\omega_1,\dots,\omega_n$ of $\mc I_G$, and define a Lie algebra whose vector space structure is given by $\mf g_\rho' = \mf g \oplus \mf e \oplus\mathfrak{z}$, and whose algebra structure is given by:
\[ \begin{array}{c}
[(X_1,V_1,Z_1),(X_2,V_2,Z_2)] = \hspace{2.5in} \\
\hspace{1in} \big([X_1,X_2]_{\mf g},d\rho(X_1)V_2-d\rho(X_2)V_1,\omega_1(V_1,V_2),\dots,\omega_n(V_1,V_2)\big)
\end{array} \]

Then define the group $G_\rho'$ to be the simply connected Lie group with Lie algebra $\mf g_\rho'$, and let $\Theta : G_\rho' \to G_\rho$ be the canonical projection. We call $G_\rho'$ the {\it universal Lie central extension of $G_\rho$}.
\end{definition}

Observe that $G_\rho'$ is still a semidirect product of $\tilde{G}$, the universal cover of $G$, and a simply connected nilpotent Lie group $E' = \exp(\RR^N \oplus\mathfrak{z})$ (see, eg, \cite{var84}). The multiplication is therefore given as:
\begin{align*}
 (g_1,x_1) \cdot (g_2,x_2) = (g_1g_2,\rho'(g_2^{-1})(x_1)x_2)
\end{align*}
where $\rho'(g_2^{-1})(x_1)=g_2^{-1}x_1g_2$ satisfying $\rho \of \Theta = \Theta \of \rho'$.

\begin{proposition}
\label{prop:universal-property}
Suppose that

\[ 1 \to Z \to H \to G_\rho \to 1 \]

is a central extension of $G_\rho$ by an abelian Lie group $Z$. Then there is a unique homomorphism $\phi : G_\rho' \to H$ equivariant over $G_\rho$. In particular, any perfect Lie central extension of $G_\rho$ is a factor of $G_\rho'$.
\end{proposition}

\begin{proof}
Since $G_\rho'$ is simply connected, it suffices to show the result at the level of Lie algebras. Note that $\mf g_\rho' \cong \mf g \ltimes (\RR^N \oplus \mf z_0)$ and that $\mf h \cong \mf g \ltimes (\RR^N \oplus \mf z)$ (where the direct sums are vector space direct sums). Recall that $[\RR^N,\RR^N]_{\mf g_\rho'} = \mf z_0$ and $[\RR^N,\RR^N]_{\mf h} \subset \mf z$ (this follows from the proof of Proposition \ref{prop:no-top-ext}). Then given $V \in \mf z_0$, we can write it as $V = \sum_{i=1}^m [Y_{1,i},Y_{2,i}]_{\mf g_\rho'}$, with $Y_{1,i},Y_{2,i} \in \RR^N$. So define $d\phi(V) = \sum_{i=1}^m[Y_{1,i},Y_{2,i}]_{\mf h}$. To see that this is well defined, observe that if $\sum_{i=1}^m[Y_{1,i},Y_{2,i}]_{\mf g_\rho'} = \sum_{j=1}^{m'}[Y_{1,j}',Y_{2,j}']_{\mf g_\rho'}$, then $\sum_{i=1}^m \omega(Y_{1,i},Y_{2,i}) = \sum_{j=1}^{m'}\omega(Y_{1,j}',Y_{2,j}')$ for every $\omega \in \mc I_G$. Then if we choose a basis $\mc B = \set{X_1,\dots,X_k}$ of $\mf z$, we see that it determines invariant 2-forms via the formula:

\[ [Y_1,Y_2]_{\mf h} = \sum_{i=1}^k \omega_i(Y_1,Y_2)X_i \]

So since these 2-forms determine the bracket in $\mf h$, we get that the definition of $d\phi$ on $\mf z_0$ is well-defined and unique. It is constructed exactly so that it extends in the obvious way to a Lie algebra homomorphism $d\phi : \mf g_\rho' \to \mf h$.

For the final part of the Proposition, we need to show that if $H$ is perfect, $d\phi$ is surjective onto $\mf z$. Observe that if $H$ is perfect it must be connected (as in the proof of the second part of \ref{prop:no-top-ext}). Then perfectness implies $[\mf h,\mf h] = \mf h$, so the algebra $\mf z = [\RR^N,\RR^N]_{\mf h}$. So we can write any $X \in \mf z$ as $X = [Y_1,Y_2]_{\mf h} = d\phi([Y_1,Y_2]_{\mf g_{\rho}})$. Hence we get that $H$ is a factor of $G_\rho'$.

\end{proof}

\begin{corollary}
The definition of $G_\rho'$ is independent of the choice of basis of $\mc I_G$.
\end{corollary}
Let $\{\mathfrak{g_i}\}_{i\in J}$ be a basis of $\mathfrak{g}$ with rational entries. Then $\mc I_G$ is exactly the subspace of $\bigwedge^2(\RR^N)^*$ satisfying:
\begin{align*}
 \omega(d\rho(\mathfrak{g_i})w_1,w_2)=-\omega(w_1,d\rho(\mathfrak{g_i}) w_2),\qquad \forall\,i\in J,\,\forall \,w_1,\,w_2\in \RR^N.
\end{align*}
Since $d\rho$ is $\Q$-rational, there is a basis $\{\omega_l\}_{l\in L}$ of $\bigwedge^2(\R^N)^*$ with rational coefficients.

Let $G_\rho'$ be the Lie group determined by this basis (see Definition \ref{def:univ-lie}). Let $\Gamma'$ be the lattice in $\widetilde{G}$ covering $\Gamma$ and $\mc Z = \Z^N \subset \mathfrak{z}$. Then since $\widetilde{G}$ sits inside of $G_\rho'$, so does $\Gamma'$. Then define the subgroup $\Gamma_\rho' = \left\langle \Gamma', \exp_{G_\rho'}(\mathcal{Z})\right\rangle$ of $G_\rho'$.

Let $X_1,\dots,X_N$ be the standard basis of $\Z^N$. Then we see that when combined with a rational basis of 2-forms to define $\mf g_\rho'$, it is a basis with rational structure constants. By Malcev's criterion for nilpotent groups, $\exp_{G_\rho'}(\mathcal{Z})$ generates a lattice in $E'$. Then we get the following immediately:

\begin{lemma}
\label{lem:lattice-lift}
$\Gamma_\rho'$ is a lattice in $G_\rho'$ and $\Theta(\Gamma_\rho') = \Gamma_\rho$, where $\Theta : G_\rho' \to G_\rho$ is the projection.
\end{lemma}

Then $G_\rho$ is then isomorphic to the double coset space $\mathcal{R}\backslash (G_\rho'/\Gamma_\rho')$, where $\mathcal{R}=\exp(\mf z)/\mathbb{L}=\TT^{\dim \mf z}$.

\subsection{Useful notations} \label{sec:14}We try as much as possible to develop a unified
systems of notations. We will use notations from this section throughout
subsequent sections. So the reader should consult this section if an unfamiliar symbol appears.
\begin{sect}\label{se:2}
For any group $S$ and its subgroup $S_1$, we use $Z(S)$ to denote it center, $C_S({S_1})$ to denote the centralizer of $S_1$ in $S$. If $H_j$, $j\in J$ are subgroups of $S$, let $\prod_{j\in J} H_j$ (or $\left\langle H_j \right\rangle$ denote the subgroup generated by $H_j$, $j\in J$. For any subset $L\subset S$, we let $F_L$ be the free group with generating set $L$.
\end{sect}

\begin{sect}\label{se:1}
We use $\mathcal{G}$ to denote $G$ or $G_\rho$; and use $\mathcal{G}_0$ to denote the universal cover $\tilde{G}$ of $G$ or a fixed maximal perfect central Lie extension $G_\rho'$ of $G_\rho$. Fix a right invariant metric $d$ on $\mathcal{G}_0$. We denote by $B(h,r)$ the ball in $\tilde{G}$ centered at $h$
with radius $r$.

Let $\Gamma'$ be the corresponding lattice in $\tilde{G}$ covering the lattice $\Gamma \subset G$. Then $G/\Gamma$ is isomeric to $\tilde{G}/\Gamma'$. Then $d$
induces the metrics on $G$ and $G/\Gamma$, which is also denoted by $d$. There exists $c_1>0$ such that, for any $h\in \tilde{G}_0$, $B(h,c_1)$ projects homeomorphically onto its image in $G/\Gamma$.

Write $\text{Lie}(G_\rho')=\mathfrak{g}\oplus\RR^N\oplus \mf z$ under the isomorphism in proof of Proposition \ref{prop:no-top-ext} and $\mf e$ is the center. Let $\Gamma_\rho'$ be the corresponding lattice in $G_\rho'$ covering the lattice $\Gamma_\rho \subset G_\rho$ as described at the end of Section \ref{sec:16}.
We see that $d$ also induces metrics on $G_\rho$ and $G_\rho/\Gamma_\rho$. There exists $c_2>0$ such that, for any $h\in G_\rho'$, the set
\begin{align*}
 \bar{B}(h,c_2)=\{g\exp(v)h:v\in \mathfrak{e},\,g\in \tilde{G}\text{ and }d(g\exp(v),e)<c_2\}
\end{align*}
projects homeomorphically onto its image in $G_\rho$ and $G_\rho/\Gamma_\rho$.
\end{sect}

\begin{sect}\label{se:3}
Let $\alpha_A$ be as described in Section \ref{sec:2}. $p:A\rightarrow A_0$ will denote the homomorphism described in Proposition \ref{prop:diagonalize} and \ref{prop:diagonalize1}. Let $N$ denote the centralizer of $p(A)$. For the twisted space we may assume that $A \subset G\cdot N$ by conjugating by the element $s$ of Proposition \ref{prop:diagonalize1}.
 Since
the exponential maps on $\Lie(A_0)$ are both injective in $\mathcal{G}_0$ and $\mathcal{G}$, we can naturally lift $p$ to map valued in a split Cartan subgroup in $\mathcal{G}_0$, which is still denoted by $p$.

Let $\Lambda_A$ denote $\Delta_{A}$ or $\Delta_{A}\bigcup\Phi^*_{A,\rho}$, where $\Phi^*_{A,\rho}=\Phi_{A,\rho}\backslash 0$. We call $r\in \Delta_{A_0}$ {\it detected} if $r|_{A} \not= 0$, i.e., $r$ descends to a root in $\Delta_A$, otherwise we say that $r$ is {\it undetected}.

For any $\phi\in \Lambda_{A}$ let $U_{[\phi]}$ be the corresponding subgroup of $\mathcal{G}$ with Lie algebra $\mathfrak{g}^{(\phi)}$. Note that later if we write $U_{[\phi]}$, $\phi\in\Delta_{A_0}$, then it denotes the root subgroups for root system $\Delta_{A_0}$. It is clear that $U_{[\phi]}$, $\phi\in \Delta_{A}$ is the product of groups $U_{[r]}$, $r\in \Delta_{A_0}$ such that $r|_{A}=\lambda\phi$ for some $\lambda>0$.

If $\mathcal{G}=G_\rho$, for any $r\in \Delta_A\cup \Phi^*_{A,\rho}$, let $\mathfrak{g}_G^{(r)}=\sum_{\lambda>0}\mathfrak{g}_{\lambda r}$ and $\mathfrak{g}_{\mathfrak{e}}^{(r)}=\sum_{\lambda>0}\mathfrak{e}_{\lambda r}$; and let $U_{G,[r]}$ and $U_{\mathfrak{e},[r]}$ be the corresponding subgroup of $G_\rho$ respectively.
Let $U_{\mathfrak{e},[r],i}=U_{\mathfrak{e},[r]}\bigcap\RR^{N_i}$, $i\in I$ (see Section \ref{sec:16}). Then $U_{[r]}$ is generated by $U_{G,[r]}$ and $U_{\mathfrak{e},[r]}$ and
$U_{\mathfrak{e},[r]}$ is generated by $U_{\mathfrak{e},[r],i}$, $i\in I$.
\end{sect}

\begin{sect}\label{se:5}
Suppose $M=G/\Gamma$. Note that $A\cong\ZZ^k\times\RR^l$ (or its image in $G$). Let
$X_i$, $ 1\leq i\leq l$ be the corresponding elements in $\mathfrak{g}$ such that the flows $\{\exp(t X_i)\}_{t\in\RR}$, $ 1\leq i\leq l$ generate $\RR^l \subset A$. We note that for any $x,\,y\in \mathcal{G}_0$ if they project to an abelian pair in $\mathcal{G}$ then $[x,y]\in Z(\tilde{G})$.  This observation implies that
the subgroup $A_1$ generated by the flows $\{\exp(t X_i)\}_{t\in\RR}$ in $G'$, $ 1\leq i\leq l$ is abelian.

Let $g_i$, $ 1\leq i\leq k$ be a set of corresponding elements in $\tilde{G}$ covering the generators of $\ZZ^k \subset A$.  Note that each $g_i$ is in $C_{\tilde{G}}(A_1)$, $ 1\leq i\leq k$,  since $Z(\tilde{G})$ is discrete. However, the set $\{g_i, 1\leq i\leq k\}$ may be non-abelian and we have
\begin{align*}
 g_ig_j=g_jg_ig_{i,j}, \qquad g_{i,j}\in Z(\tilde{G}).
\end{align*}
In particular, if $A$ is inside the split Cartan subgroup then $A_1$ is abelian.

Let $\tilde{A}$ be the set containing $\{g_i, 1\leq i\leq k\}$ and $A_1$, and for any $a\in A$ denote the corresponding element in $\tilde{A}$ by $\tilde{a}$.
Then we obtain an action of $F_{\tilde{A}}$ on $\mathcal{G}_0$ generated by left translations $\alpha(\tilde{a},\cdot)$, which is denoted by $\alpha_{\tilde{A}}$.
It is clear that translations in $\alpha_{\tilde{A}}$ preserve stable/unstable foliations of each other. Thus we can also define the coarse Lyapunov foliations for $\alpha_{\tilde{A}}$, which are the same as those for $\alpha_{p(A)}$ (see \seref{se:3}).

\end{sect}

\begin{sect}\label{se:7}
Suppose $M=G_\rho/\Gamma_\rho$. For any $c\in G_\rho$ we have the expression $(g_c, v_c)$, where $g_c\in G$ and $v_c\in\RR^N$. The set $A_G=\{g_a,a\in A\}$ is an abelian subgroup in $G$.  Set $a'=(g_a', \exp_{G_\rho'}(\log v_a))$ for any $a\in A$, where $g_a'$ is the
corresponding element in $\tilde{A}_G$ as defined in \seref{se:5}.
\begin{remark}
Note that the set $A'=\{a',a\in A\}$ may no longer be an abelian set again in $G_\rho'$. In fact for any $a',\,b'$ and $x\in G_\rho'$ we have
\begin{align}\label{for:25}
 a'b'=b'a'cz,\qquad \text{ for some }c\in \exp(\mathfrak{z})\text{ and }z\in Z(\tilde{G}).
\end{align}
In particular, $z\in Z(G_\rho')$ since $z$ projects to identity in $G_\rho$.

If there is no symplectic contrubtion of $\rho$, then $c$ is trivial, since $\mathfrak{z}= \set{0}$ in this case. 

Let $\alpha_{A'}$ denote $F_{A'}$ action generated by left translations $\alpha(a',\cdot)$ on $\mathcal{G}_0$ for any $a\in A$.
\eqref{for:25} shows translations in $\alpha_{A'}$ preserve stable/unstable foliations of each other. Thus we can also define the coarse Lyapunov foliations for $\alpha_{A'}$, which are the same as those for $\alpha_{p(A)}$.
\end{remark}

\end{sect}

\section{Geometric Preparatory step II: Lifting from homogeneous space to Lie group}
\label{sec:geom2}

In this part we set up all the ingredients that will be use later to prove cocycle rigidity for perturbations, which is the principal step toward the proof of Theorem \ref{thm:main} and \ref{thm:main:1}. After conjugating the perturbation with the Hirsch-Pugh-Shub conjugacy $h$ , the perturbation has the same neutral foliations as for $\alpha_A$ (see Section \ref{sec:12}). Hence we set up a twisted cocycle (see Definition \ref{def:twisted-coc}) over the conjugated perturbation; then we lift both $\alpha_A$ and the
conjugated perturbation to $\mathcal{G}_0$ (see \seref{se:1}) and make a detailed study of dynamical properties of both quasi-nil extensions (see Definition \ref{def:quasi-nil}). Then we
construct the periodic cycle functional for the conjugated perturbation in Section \ref{sec:PCF}.

\subsection{Actions and cocycles on homogeneous space}\label{sec:12}

 Let $\widetilde{\alpha}_A$ be a $C^\infty$ $A$-action close to $\alpha_A$ in $C^1$
topology.  The neutral foliation for $\alpha_{A}$ is a smooth foliation, we may use
the Pugh-Shub-Wilkinson structural stability theorem to conjugate the small perturbation $\tilde{\alpha}_A$ to an action $\overline{\alpha}_A$ preserving the
neutral foliation $\mathcal{N}$ of $\alpha_A$ via a bi-$\varrho$-H\"older conjugacy $\mbf{h}$, where $\varrho$ is only dependent on $\alpha_A$ and the $C^1$-closeness of $\widetilde{\alpha}_A$. This has been explained
already in Sections \ref{sec:11}.  Thus $\bar{\alpha}_A$ is a small $\varrho$-H\"older perturbation
of $\alpha_A$ along the leaves of the neutral foliation $\mathcal{N}$ whose leaves are $\{N\cdot x : x\in M\}$ (see \seref{se:3} of Section \ref{sec:14}), although it can be chosen to be smooth along the leaves of that foliation.

 Then we have that $\overline{\alpha}_A$ is given by a map $\beta: A\times M\rightarrow N$ by
\begin{align}\label{for:10}
\bar{\alpha}_A(a, x) = \beta(a, x) \cdot \alpha_{A}(a, x)
\end{align}
 for $a\in A$ and $x\in M$. Note that $\beta$ is $C^0$ close to identity. Furthermore, it is a $\kappa$-H\"older map with small H\"older norm, where $\kappa=\varrho^2$.

Since $A$ may fail to commute with $N$, $\beta$ is not in general a cocycle over $\overline{\alpha}_A$. Instead, we need to address the case of {\it twisted cocycles}:

\begin{definition}
\label{def:twisted-coc}
Let $S$ be a Lie group and $\psi : A \to \Aut(S)$ is a homomorphism (we will denote the automorphism $\psi(a)$ by $\psi_a$) and $\alpha$ an action of $A$ on a manifold $\mathcal{B}$, a {\it cocycle over $\alpha$ taking values in $S$ twisted by $\psi$} is a map $\beta : A \times \mathcal{B} \to S$ such that:
\begin{align}\label{for:8}
 \beta(ab,x) = \beta(a,\alpha^b(x))\psi_a(\beta(b,x)).
\end{align}
$\beta$ is said to be cohomologous to a constant (ie, independent of $x \in \mathcal{B}$) twisted cocycle there exists a map $i:A\to S$ satisfying $i(ab) = i(a)\psi_a(i(b))$ and a continuous  transfer map $T : \mathcal{B} \to S$ such that for all $a\in A$
\begin{align*}
  \beta(a,x) = T(\alpha^a(x))i(a)\psi_a(T(x))^{-1}.
\end{align*}
We say that $\psi$ is a {\it slow twist} (or that $\beta$ is {\it slowly twisted}) if there exists a homeomorphism $\iota:A\to S$ such that $\psi_a$ is given by conjugation of $\iota(a)$ with eigenvalues of $\text{Ad}(\iota(a))$ all of modulus $1$ on $\text{Lie}(S)$.
\end{definition}

\begin{lemma}
\label{lem:correction-cocycle}
$\beta$ in \eqref{for:10} is a twisted cocycle over $\overline{\alpha}_A$ with $\psi_a$, $a\in A$ given by the conjugation of $a$.
\end{lemma}
\begin{proof}
Since $A$ is abelian we have
\begin{align*}
 \bar{\alpha}_A(ab, x)&=\bar{\alpha}_A(a, \bar{\alpha}_A(b, x))=\beta(a, \bar{\alpha}_A(b, x)) \cdot \alpha_{A}(a, \bar{\alpha}_A(b, x))\\
 &=\beta(a, \bar{\alpha}_A(b, x)) \cdot a \cdot \beta(b, x) \cdot \alpha_{A}(b, x)
\end{align*}
On the other hand,
\begin{align*}
 \bar{\alpha}_A(ab, x)=\beta(ab, x) \cdot \alpha_{A}(ab, x).
\end{align*}
This implies
\begin{align}\label{for:32}
  \beta(ab, x)&=\beta(a, \bar{\alpha}_A(b, x))\cdot a\beta(b, x)a^{-1}=\beta(a, \bar{\alpha}_A(b, x))\cdot \psi_a(\beta(b, x)).
\end{align}
Then $\beta$ is a twisted cocycle over $\bar{\alpha}_A$.
\end{proof}

\begin{remark}
In fact Corollary \ref{cor:1} and \ref{cor:2} imply that $\beta$ is slowly twisted in both senses. Note that if $N$ is inside $C_{\mc G}(A_0)$, then we are in the case of an untwisted cocycle, since the conjugation action will be trivial. This is the case for restrictions of split Cartan actions.
\end{remark}

\subsection{Lifting of perturbed actions and cocycles}\label{sec:20}
In \seref{se:5} and \seref{se:7} of Section \ref{sec:14}, we lifted $\alpha_A$ to an $F_{\tilde{A}}$-action $\alpha_{\tilde{A}}$ or to an $F_{A'}$-action $\alpha_{A'}$, which may fail to give rise to $A$-actions. Our goal in this subsection will be to study the dynamical properties of liftings of the conjugated perturbed action $\overline{\alpha}_A$ on $\mathcal{G}_0$. Recall Definition \ref{de:2}. We consider two types of actions separately: An action introduced in Section \ref{sec:2} will be called {\it Type II} if it is a twisted symmetric space example which has both a symplectic contribution and zero weight for $A$. Otherwise, we call it {\it Type I}. Note that Type I actions are exactly the actions satisfying the conditions of Theorem \ref{thm:main}.

Fix a compact generating set $\mathbb{A}$ in $A$. First, we show how to lift $\overline{\alpha}_A(a,\cdot)$ for each $a\in \mathbb{A}$ for both types in a uniform way, which is standard. Let $\mf G = \Lie(\mc G)$ and $\mf G_0 = \Lie(\mc G_0)$. First, recall that $\overline{\alpha}_A(a,x) = \beta(a,x) \alpha_A(a,x)$. Since $\beta(a,x) \in N$ and is close to the identity (for elements of $\mbb{A}$), we can write $\beta(a,x) = \exp_{\mf G}( \beta_0(a,x))$ for some $\beta_0 : \mbb A \times M \to \Lie(N)$. Recall that there is a canonical vector space splitting $\mf G \to \mf G_0$ (see Definition \ref{def:univ-lie}), so that we identify $\mf G$ with a subspace of $\mf G_0$ (however, it is only a subalgebra if $\mf G = \mf G_0$). Then define:

\[ \hat{\alpha}_A(a,g) = \exp_{\mf G_0}(\beta_0(a,x))ag \]

on the group $\mc G_0$. Note that $\hat{\alpha}_A(a,g\gamma) = \hat{\alpha}(a,g)\gamma$ for every $\gamma \in \Gamma'$ (resp $\Gamma_\rho'$), so $\alpha$ is indeed a well-defined lift.

Fix a sufficiently small $0<c<c_0$ such that for any $a\in \mathbb{A}$ and $h\in \mathcal{G}_0$ $\hat{\alpha}_{A}(a)$ maps $B(h,\,c)$ inside $B\big(\hat{\alpha}_{A}(a,h),\,c_0\big)$. Let $j$ denote the projection from $\mathcal{G}_0$ to $M$. For Type I actions, for any
$h'\in B(h,\,c)$ if
\begin{align}\label{for:51}
 \overline{\alpha}_A(a,j(h'))=g_{h'}\cdot \overline{\alpha}_A(a,j(h)),
\end{align}
where $g_{h'}\in \mathcal{G}$ is close to identity, then
\begin{align}\label{for:52}
 \hat{\alpha}_A(a,h')=\exp_{\mc G_0}(\log_{\mc G} g_{h'})\cdot \hat{\alpha}_A(a,h).
\end{align}

Notice that Type I actions are exactly those action for which $N$ embeds into $\mc G_0$ as a subgroup (instead of a submanifold) via the exponential mapping. This implies that, for Type I actions,

\begin{align}\label{for:49}
 \hat{\alpha}_{A}(a)\circ \hat{\alpha}_{A}(b)x = c(a,b)(\hat{\alpha}_{A}(b)\circ \hat{\alpha}_{A}(a)x)
\end{align}
for any $a,\,b\in A$ and $c(a,b) = [\tilde{a},\tilde{b}] \in \mc G_0$ (resp. $[a',b']$) is an element of $Z(G) \cap \Gamma$ (we choose fixed, continuously varying representatives $\tilde{a}$ (resp. $a'$) in the group $\mc G_0$ as in \seref{se:5} and \seref{se:7}). Notice that this only occurs for discrete actions, for actions of $\R^k$ since $c$ depends continuously on $a$ and $b$. Notice also that this deficiency already occurs for the homogeneous action. 
\eqref{for:51} and \eqref{for:52} show that the leaves of stable/unstable/netrual foliations for $\overline{\alpha}_A$ can be uniquely lifted to leaves of stable/unstable/netrual foliations for $\hat{\alpha}_A$ correspondingly; and \eqref{for:49} shows that
elements in $\hat{\alpha}_A$ preserve stable/unstable foliations of each other.

For Type II actions, \eqref{for:49} changes to
\begin{align*}
 \hat{\alpha}_{A}(a)\circ \hat{\alpha}_{A}(b)x =c(a,b) \cdot c'(a,b,x) \cdot (\hat{\alpha}_{A}(b)\circ \hat{\alpha}_{A}(a)x)
\end{align*}
where $c'(a,b,x)$ is a continuous function taking values in $\exp(\mathfrak{z})$. Hence Type II actions may fail to preserve stable/unstable foliations of each other. If we suppose the smooth perturbation $\widetilde{\alpha}_A$ admits a quasi-nil extension, then conjugate this extension with $\tilde{H}$, where $\tilde{H}$ is the lifting of $\textbf{h}$ on $\mathcal{G}_0$ obtained by observing that it is close to the identity, so we may write it as left translation by some small group element. Denote the resulted set of homeomorphisms  by $\hat{\alpha}_{A}$.

We can lift the cocycle $\beta$ to $A \times \mc G_0$ by pullback (ie, $\beta(a,\tilde{x})=\beta(a,x)$, where $\tilde{x}\in \mathcal{G}_0$ projects to  $x\in M$). Then even though the lifted set $\widetilde{A}$ or $A'$ may fail to be abelian, because they lift the action on the base and the cocycle is constant on fibers, we have
\begin{align}\label{for:33}
  \beta(ab, \tilde{x})&=\beta(a, \widehat{\alpha}_A(b, \tilde{x}))\cdot \psi_a(\beta(b, \tilde{x})).
\end{align}
So $\beta$ is also called a twisted cocycle over $\widehat{\alpha}_A$.

\subsection{Slowly twisted cocycle and potential function}
\label{sec:PCF}
The contents of this section form generalizations of results from \cite[Section 3]{dk2005} and \cite[Lemma 4.1 \& 4.2]{knt00} to the case of twisted cocycles. We require this adaptation to address the case when $\alpha_A$ is not the restriction of a split Cartan action and its quasi-nil extensions. As we have seen in Proposition \ref{prop:diagonalize} and \ref{prop:diagonalize1} the leaves of the neutral foliation for $\alpha_A$ (or a quasi-nil extension) is $N=C_{\mc G}(p(A))$, not $C_{G_\mc G}(A)$, which results in the difficulty since we need to handle twisted cocycles instead of cocycles. All the properties in our setting are similar to those in previous papers for small untwisted cocycles or cocycles over compact or abelian groups (see \cite{dk2005}, \cite{knt00}, \cite{dktoral}).

In this part we prove general results for $C^0$-small slowly twisted $\kappa$-H\"older cocyles over a partially hyperbolic $A$-action taking values in $N$. We now begin the task of defining the periodic cycle functional. We desire a function family on the foliations $f(x,\cdot) : \mathcal{T}_{i}(x) \to N$ having the properties that $f(x,x) = e$ and:

\begin{equation}
f(\alpha^a(x),\alpha^a(y)) = \beta(a,x)\psi_a(f(x,y)) \beta(a,y)^{-1}
\end{equation}
for all $a \in A$. If this formula holds, then we may apply it to $a^n$ and rearrange to get:

\[f(x,y) = \psi_{a^{-n}}(\beta(a^n,x)^{-1} f(\alpha^{a^n}(x),\alpha^{a^n}(y))\beta(a^n,y))\]

If we choose $a$ which contracts the $\mathcal{T}_{i}(x)$, then the middle term on the right hand side should tend to some $f(z,z) = e$. This suggests defining the potential:

\begin{definition}
We define $N$-valued potential of $\beta$ as
\begin{align}\label{eq:pcf-fund}
p_{\beta,a}(x,y)=\lim_{n\to\delta\infty}\psi_{a^{-n}}(\beta(a^n,x)^{-1}\beta(a^n,y))
 \end{align}
where $\delta\in \{+,\,-\}$, and if $x$ is on a stable (resp. unstable) leaf of $y$ under  $\alpha^{a}$ we take $\delta=+$ (resp. $\delta=-$).
\end{definition}

The next proposition justifies the definition for $p_{\beta,a}(x,y)$ for small slowly twisted $\kappa$-H\"older cocycles. Before the proof we state two technical facts which allow us to control the order of the twist effectively.

\begin{fact}\label{lem:good-norm}
Suppose $S\subset GL(n,\R)$ is a compact abelian set such that for any $s\in S$ all of the eigenvalues $s$ of modulus $1$. Then for every $\ve > 0$ there exists a norm on $\RR^n$ under which the matrix norm $\norm{\cdot}$ of $s$ is less than $\ve+1$ for any $s\in S$.
\end{fact}

We briefly justify this fact. There exists a basis under which $s$ is an upper triangular matrix for any $s\in S$. Since the eigenvalues of $s$ all have modulus $1$, so must the entries of the diagonal of $s$ in this basis for any $s\in S$. Then consider the matrix
\begin{align*}
  \diag(\lambda^{\floor{n/2}},\lambda^{\floor{n/2}-1},\dots,\lambda^{1-\floor{n/2}},\lambda^{-\floor{n/2}}),\qquad \text{for }\lambda < 1.
\end{align*}
This matrix contracts every upper triangular matrix, and all of the entries are multiplied by a power of $\lambda$. Thus, if we choose $\lambda$ sufficiently small, if we scale the chosen basis appropriately, we get the norm of $\RR^n$ under this basis.
\begin{fact}\label{fact:1}
Fix a right invariant metric $d$ on $N$.  Then for any norm $\norm{\cdot}$ on its Lie algebra $\mathfrak{h}$ there is a small neighborhood $V$ of $0\in \mathfrak{h}$ such that $\exp:V\to N$ is a diffeomorphism. Furthermore, there exist constants $c_1,\,c_2$ (dependent on $\norm{\cdot}$ and $d$) such that
\begin{align*}
 c_1\norm{v}\leq d(\exp v,e)\leq c_2\norm{v}, \qquad \text{for }\forall v\in V.
\end{align*}
\end{fact}
For any $a\in \mathbb{A}$,  we denote by $\lambda_{\pm}(a)$ the contraction
and expansion coefficients of $\alpha^a$,  defined by

Set $\epsilon=\min_{a\in \mathbb{A}}\{\lambda_{+}(a)^{-\frac{1}{2}},\lambda_{-}(a)^{\frac{1}{2}}\}$. Suppose $\beta$ is a $\kappa$-H\"older slowly twisted cocycle. Fix a norm $\norm{\cdot}$ on $\text{Lie}(N)$ such that $\sup_{a\in \mathbb{A},\delta=\pm 1}\norm{\text{Ad}(\iota(a))^\delta}<\epsilon^{\kappa/3}$ (see Definition \ref{def:twisted-coc}), under the matrix norm (see Fact \ref{lem:good-norm}).

\begin{proposition}\label{po:1}
  If $\beta$ satisfies:
 \begin{align}\label{for:58}
   \sup_{a\in \mathbb{A},\,x\in \mathcal{B},\,\delta=\pm 1}\norm{\Ad(\beta(a,x))^{\delta}}<\epsilon^{\kappa/3}.
 \end{align}
Then:
\begin{enumerate}
  \item \eqref{eq:pcf-fund} is well defined for any $a\in \mathbb{A}$. Furthermore, $p_{\beta,a}$ depends continuously on $x$ and $y$ along a stable (unstable) leaf for $\alpha(a)$. \label{prop:pot-well-def}

      \smallskip
  \item $p_{\beta,a}(x,y)=p_{\beta,b}(x,y)$ for any $a,\,b\in \mathbb{A}$ if $x$ and $y$ are inside a stable or unstable leaf for $\alpha(a)$ and $\alpha(b)$.  \label{prop:ind-on-a}
\end{enumerate}

\end{proposition}
\begin{proof} Proof of \eqref{prop:pot-well-def}:
We prove the case when $x$ and $y$ are on a stable leaf of $a$. Then $\epsilon^{\kappa/3}< \lambda_{+}(a)^{-\kappa/3}$. For ease of notation, we will denote the action of $\alpha$ by $\alpha^a(x)= a \cdot x$. Let $\gamma_{n}(x,y) = \psi_{a^{-n}}(\beta(a^n,x)^{-1}\beta(a^n,y))$. Iterating \eqref{for:8} we have

\begin{align}\label{for:7}
  \gamma_{n+1}(x,y)\gamma_{n}(x,y)^{-1}=\psi_{b_{n-1}}\big(\beta(a,a^{n}\cdot x)^{-1}\beta(a,a^{n}\cdot y)\big)
\end{align}
where $b_{n-1}=a^{-1}\beta(a,x)^{-1}a^{-1}\beta(a,a\cdot x)^{-1}\cdots a^{-1}\beta(a,a^{n-1}\cdot x)^{-1}$.

Write
\begin{align*}
  \beta(a,a^{n}\cdot x)^{-1}\beta(a,a^{n}\cdot y)=\exp(\mathfrak{c}_n)\quad\text{and}\quad\beta(a,a^{n}\cdot y)\beta(a,a^{n}\cdot x)^{-1}=\exp(\mathfrak{p}_n)
\end{align*}
 where $\mathfrak{c}_n,\,\mathfrak{p}_n\in \text{Lie}(N)$ (they are small and hence must be exponential).

 Since $\exp(\mathfrak{c}_n)=\beta(a,a^{n}\cdot y)^{-1}\exp(\mathfrak{p}_n)\beta(a,a^{n}\cdot y)$, by using Fact \ref{fact:1} and the assumption about $\beta$, for big enough $n$ we have
\begin{align}\label{for:4}
 \norm{\mathfrak{c}_n}&\leq c_1^{-1}d\Big(\exp\big(\text{Ad}(\beta(a,a^{n}\cdot y)^{-1})\mathfrak{p}_n\big),e\Big)\leq c_1^{-1}c_2\norm{\text{Ad}(\beta(a,a^{n}\cdot y)^{-1})\mathfrak{p}_n}\notag\\
 &\leq c_1^{-1}c_2\epsilon^{\kappa/3}\norm{\mathfrak{p}_n}\leq c_1^{-2}c_2\epsilon^{\kappa/3}d\big(\beta(a,a^{n}\cdot x),\beta(a,a^{n}\cdot y)\big)\notag\\
 &\leq c_1^{-2}c_2\epsilon^{\kappa/3}C(\beta,a) \lambda_{+}(a)^{n\kappa}d_{\mathcal{B}}(x,y)^\kappa,
\end{align}
where $C(\beta,a)$ is the H\"older norm of $\beta(a,\cdot)$.

By definition we have $\gamma_{n+1}(x,y)\gamma_{n}(x,y)^{-1}=\exp\big( \text{Ad}(b_{n-1})\mathfrak{c}_n\big)$. Then by using Fact \ref{fact:1} for big enough $n$ we have
\begin{align}\label{for:5}
  &d\big(\gamma_{n+1}(x,y), \gamma_{n}(x,y)\big)=d\big(\gamma_{n+1}(x,y)\gamma_{n}(x,y)^{-1}, e\big)\leq c_2\norm{\text{Ad}(b_{n-1})\mathfrak{c}_n}\notag\\
  &\leq \Pi_{j=0}^{n-1}\big(\norm{\text{Ad}(\iota(a))^{-1}}\cdot\norm{\text{Ad}(\beta(a,a^{j}\cdot x))}\big)\norm{\mathfrak{c}_n}\notag\\
  &\leq \epsilon^{n\kappa/3}\cdot\epsilon^{n\kappa/3}\cdot c_1^{-2}c_2\epsilon^{\kappa/3}C(\beta,a) \lambda_{+}(a)^{n\kappa}d_{\mathcal{B}}(x,y)^\kappa\notag\\
  &\leq c_1^{-2}c_2C(\beta,a) \lambda^{(n-1)\kappa/3}d_{\mathcal{B}}(x,y)^\kappa.
\end{align}
Hence $\gamma_{n}(x,y)$ is a Cauchy sequence which implies that \eqref{eq:pcf-fund} is well defined.

\eqref{for:4} and \eqref{for:5} also hold for any $n$ if $d(x,y)$ is small enough. This shows that
\begin{align*}
 d(\gamma_{n}(x,y), e)&\leq\sum_{j=1}^nd\big(\gamma_{j}(x,y), \gamma_{j-1}(x,y)\big)\notag\\
 &\leq \sum_{j=1}^nc_1^{-2}c_2C(\beta,a) \lambda^{(j-1)\kappa/3}d_{\mathcal{B}}(x,y)^\kappa\leq C(\beta,a)'d_{\mathcal{B}}(x,y)^\kappa.
\end{align*}
Let $n\rightarrow\infty$ we get
\begin{align*}
 d(p_{\beta,a}(x,y), e)&\leq C(\beta,a)'d_{\mathcal{B}}(x,y)^\kappa
\end{align*}
if $d(x,y)$ is small enough. This also implies the continuity in $x$ and $y$. The case when $x$ and $y$ are on a unstable leaf can be proved by using the same method. Hence we finish the proof.

\medskip

\noindent Proof of \eqref{prop:ind-on-a}: We first note that if $\beta$ is a cocycle over $\alpha$ then for any $a,\,b\in \mbb A$ we have
\begin{align*}
 p_{\beta,b}(a\cdot x,a\cdot y)&= \beta(a,x)\psi_b(p_{\beta,b}(x,y))\beta(a,y)^{-1},\qquad\text{and}\\
 p_{\beta,a}(a\cdot x,a\cdot y)&=\beta(a,x)\psi_a(p_{\beta,a}(x,y))\beta(a,y)^{-1}.
\end{align*}
This implies that
\begin{align*}
 p_{\beta,b}(a\cdot x,a\cdot y)p_{\beta,a}(a\cdot x,a\cdot y)^{-1}=\beta(a,x)\psi_a\big(p_{\beta,b}(x,y)p_{\beta,a}(x,y)^{-1}\big)\beta(a,x)^{-1}.
\end{align*}
Set $r_n=p_{\beta,b}(a^n\cdot x,a^n\cdot y)p_{\beta,a}(a^n\cdot x,a^n\cdot y)^{-1}$. Therefore, above equation shows that for each $n \in \Z$:
\begin{align*}
r_0=d_n^{-1}r_nd_n,
\end{align*}
where $d_n=\beta(a,a^{n-1}\cdot x)a\beta(a,a^{n-2}\cdot x)a\cdots \beta(a,x)a$.

By using \eqref{for:7} for big enough  $n$  (resp. for big enough $-n$) we have
\begin{align*}
\norm{\log r_n}&\leq c_1^{-1}d(p_{\beta,b}(a^n\cdot x,a^n\cdot y),e)+c_1^{-1}d(p_{\beta,a}(a^n\cdot x,a^n\cdot y),e)\\
&\leq c(\beta,a,b)\lambda_{+}(a)^{\kappa n}d_{\mathcal{B}}(x,y)^\kappa.
\end{align*}
Furthermore, similar to \eqref{for:5} we have
\begin{align*}
  d(r_0,e)&=d(d_n^{-1}r_nd_n,e)\leq c_2\norm{\log(d_n^{-1}r_nd_n)}=\norm{\text{Ad}(d_{n-1})\log r_n}\\
  &\leq \Pi_{j=0}^{n-1}\big(\norm{\text{Ad}(\iota(a))^{-1}}\cdot\norm{\text{Ad}(\beta(a,a^{j}\cdot x))^{-1}}\big)\norm{\log r_n}\notag\\
  &\leq c(\beta,a,b) \lambda_{+}(a)^{n\kappa/3}d_{\mathcal{B}}(x,y)^\kappa.
\end{align*}
By letting $n\to \infty$ ($n\to -\infty$), the right hand side of the expression above tends to zero. Thus $r_0=e$ and this implies the conclusion.
\end{proof}
As an immediate corollary to above proposition, we get
\begin{corollary}\label{cor:5}
$p_{\beta,a}(x,y)$ depends continuously on $x$ and $y$.
\end{corollary}
\begin{proof}
\eqref{for:5} implies that there is $\delta>0$ such that if $d_{\mathcal{B}}(x,x')+d_{\mathcal{B}}(y,y')<\delta$ for any $x',\,y'\in M$, then
\begin{align*}
  \sum_{n\geq m}d\big(\gamma_{n+1}(x',y'), \gamma_{n}(x',y')\big)\leq C_{x,y}\lambda_{+}(a)^{m\kappa/3}d_{\mathcal{B}}(x,y)^\kappa.
\end{align*}
if $x'$ and $y'$ are on a stable leaf of $a$, where $C_{x,y}$ is a constant dependent only on $x,\,y$. Note that for any $\epsilon>0$
\begin{align*}
 d\big(p_{\beta,a}(x,y), p_{\beta,a}(x',y')\big)&\leq \sum_{0\leq i\leq m}d\big(\gamma_{i}(x',y'), \gamma_{i}(x,y)\big)\\
 &+\sum_{n\geq m+1}d\big(\gamma_{n+1}(x',y'), \gamma_{n}(x',y')\big)\\
 &+\sum_{n\geq m+1}d\big(\gamma_{n+1}(x,y), \gamma_{n}(x,y)\big)\\
 &\leq \sum_{0\leq i\leq m}d\big(\gamma_{i}(x',y'), \gamma_{i}(x,y)\big)+2\epsilon
\end{align*}
 if $m$ is big enough. This shows the continuity if  $x'$ and $y'$ are on a stable leaf of $a$. Similarly, we get the result if $x'$ and $y'$ are on a unstable leaf of $a$. Then we finish the proof.
 \end{proof}

\subsection{Periodic cycle functional and cocycle rigidity}\label{sec:18}

Proposition \ref{po:1} allows us to associate to each Lyapunov path an element of $N$ in the following manner: If $\tau = (x_0,x_1,\dots,x_n)$ is a Lyapunov path, let
\begin{align*}
 F_\beta(\tau) = \prod_{k=n}^1 p_{\beta,a_k}(x_k,x_{k-1})
\end{align*}
where $a_k\in \mathbb{A}$.

For any $a\in A$, $\tau = (\alpha^ax_0,\alpha^ax_1,\dots,\alpha^ax_n)$ is also a Lyapunov path, which is denoted by $\alpha^a\tau$.

The next proposition quite standard and the proof is the same as the untwisted cocycles.  We still give a proof here to make this part complete. We assume notations in Proposition \ref{po:1}.

\begin{proposition}
\label{prop:PCF-nec-suff}
Suppose coarse Lyapunov foliations of $\alpha$ are transitive.  Also suppose $\mathbb{A}(\ZZ)$ is a subgroup of $A$ generated by elements in $\mathbb{A}$ over $\ZZ$. For any $c\in \mathbb{A}(\ZZ)$ $\beta(c,\cdot)$ is cohomologous to a constant twisted cocycle simultaneously if $P_\beta$ vanishes identically on any Lyapunov cycles. Furthermore, if $\mathbb{A}(\ZZ)$ is dense in $A$ then
$\beta $ is cohomologous to a constant twisted cocycle simultaneously.
\end{proposition}

\begin{proof}
Since $P_\beta$ vanishes identically, we may safely define a function $T : M \to N$ by choosing a path $\tau$ from the fixed base point $x_0$ to $x$ and letting $T(x) = F_\beta(\tau_0)$. Then we may build a path from $\alpha^a(x_0)$ to $\alpha^a(x)$ by flowing $\tau_0$ by $a\in \mathbb{A}$ and concatenating this with the original path to $x$. It is easy to check that
\begin{align*}
 F_\beta(\alpha^a(\tau))=\beta(a,x_0) \psi_a(F_\beta(\tau))\beta(a,x)^{-1}.
\end{align*}
Therefore:
\begin{align*}
  T(\alpha^a(x))&= T(\alpha^a(x_0)) F_\beta(\alpha^a(\tau)) \\
 &=T(\alpha^a(x_0)) \beta(a,x_0) \psi_a(F_\beta(\tau))\beta(a,x)^{-1} \\
 &=T(\alpha^a(x_0))\beta(a,x_0) \psi_a(T(x)) \beta(a,x)^{-1}
\end{align*}
Thus if we let $i(a) = T(\alpha^a(x_0))\beta(a,x_0)$ (note that $x_0$ is fixed, so this is a function of $a$ only), then we have shown that
\begin{align*}
 \beta(a,x) = T(\alpha^a(x))^{-1}i(a)\psi_a(T(x)),\qquad \forall a\in \mathbb{A}.
\end{align*}
It is easy to check that above can be extended to
\begin{align*}
 \beta(a,x) = T(\alpha^a(x))^{-1}i(a)\psi_a(T(x)),\qquad \forall a\in \mathbb{A}(\ZZ).
\end{align*}
Finally, the continuity of $\beta$, $T$ and $i$ implies the last statement.
Hence we finish the proof.

\end{proof}
\section{Algebraic Preparatory step I: notations and main technical step}\label{sec:alg1}
\subsection{Central Extensions}\label{sec:15}

In this section, we introduce some principal auxiliary groups. In the case of homogeneous actions, the space of paths and cycles introduced in Section \ref{sec:5} carry a particularly interesting algebraic structure after one quotients by a group with certain dynamical meaning. Because the periodic cycle functional discussed in the previous section has certain continuity, putting a ``good'' topology on these spaces also plays an important role. In this section, we carefully introduce these groups and their topologies. Recall notations in \seref{se:3} of Section \ref{sec:14}.

\begin{definition}\label{de:1}
 For $\mu, \nu\in \Lambda_A$
such that $\mu\neq -k\nu$, for any $k \in \R_+$, define $C(\mu,\nu) = \set{t_1\mu + t_2\nu : t_1,t_2 \in \R_+} \cap \Lambda_A$ to be {\it the cone generated by $\mu$ and $\nu$}.

It is easy to check that
\begin{equation}
\label{eq:comm-eq}
[U_{[\mu]}, U_{[\nu]}]\subset\prod_{\chi \in C(\mu,\nu)}U_{[\chi]}.
\end{equation}

This clearly gives rise to relations between
commutators of above form and elements $R_{\mu,\nu,A}$ which belongs to $\displaystyle\prod_{\chi \in C(\mu,\nu)}U_{[\chi]}$:

\begin{align*}
  [x,y]=R_{\mu,\nu,A}(x,y,\mathfrak{o})\qquad x\in U_{[\mu]} \text{ and }y\in U_{[\nu]},
\end{align*}

where $\mathfrak{o}$ is an order for the roots appearing on the the right side product of \eqref{eq:comm-eq}. If such an ordering is fixed or understood, we often write $R_{\mu,\nu,A}(x,y)$ instead.
Such relations are called {\it commutator relations}. For simplicity, we usually write $R_{\mu,\nu,A}(x,y)$.
\end{definition}

Let $\widetilde{\mathcal{G}}_A=\prod^*_{[r] \in \Lambda_A} U_{[r]}$ be the free product of groups. Now for $\mu, \nu\in \Lambda_A$
such that $\mu\neq -k\nu$, for any $k \in \R_+$, element of the form
\begin{align}\label{for:28}
 [x,y]R_{\mu,\nu,A}(x,y)^{-1}
\end{align}
has a natural meaning in $\widetilde{\mathcal{G}}_A$. Let $\mc S_A$ be the normal subgroup generated by elements of the above forms.

Suppose $G'$ is either a connected Lie group with Lie algebra $\text{Lie}(G)$ (in the case of semisimple groups), or the universal Lie central extension of $G_\rho$. For any $r\in \Lambda_A$ there is a canonical homomorphism $\pi_{G'} : \widetilde{\mathcal{G}}_A \to G'$ which lifts the inclusions of $\mf g^{(r)}$ to $\mf g'=\text{Lie}(G')$. $\pi_{G'}$ is then the homomorphism guaranteed by the universal property of free products.

Define $\mc C_A(G')= \ker \pi_{G'}$ to be the {\it cycle subgroup} (this coincides with the cycle group defined in Section \ref{sec:PCF}). For simplicity, we denote $\mc C_A(\mathcal{G}_0)$ by $\mc C_A$, $\pi_{\mathcal{G}_0} $ by $\pi_1$ (see \eqref{se:1} of Section \ref{sec:14}).
\subsection{Topology on $\widetilde{\mathcal{G}}_A$} \label{sec:7}

The group $\widetilde{\mc G}_A$ carries a canonical topology (called the {\it free product topology}), which is the finest topology making $\widetilde{\mathcal{G}}_A$  a topological group such that the inclusion of $U_{[r]}$ into $\widetilde{\mc G}_A$ is a homeomorphism onto its image for every $r\in \Lambda_A$. This topology was first introduced by Graev in \cite{Gra50}, and it enjoys the following universal property:

\smallskip
\noindent \emph{If $\varphi_{[r]} : U_{[r]} \to H$ are continuous homomorphisms for every $r \in \Lambda_A$, then the extension $\phi : \widetilde{\mc G}_A \to H$ guaranteed by the universal property of free products is continuous.}

\smallskip
The proof of existence for this topology is nontrivial. The first English proof to appear was in \cite{morris}, but the proof was found to have an error. The first proof was Graev's Russian proof \cite{Gra50}, and an English proof can be found in \cite{ordman74}.

In most cases, the topology of a free product is quite pathological. If at least one group of the product is not discrete, the free product is never metrizable, nor even first countable. There are weaker topologies on the free product that yield these nicer properties, but they lack the universal property the free product topology carries.

When the groups appearing in the a free product are $k_\omega$ spaces, the free product topology can be described explicitly. We follow \cite{ordman74-1} to describe the topology of $\widetilde{\mc G}_A$, and reproduce elements of this paper, changing some notations to our specific case.

An element $\mbf r \in\Lambda_A^n$ is called a {\it combinatorial pattern}, and $n$ is called its {\it combinatorial length}. For any $\textbf{r}=(r_1,\cdots,r_n)\in \Lambda_A^n$ let $X_{\textbf{r}}=U_{[r_1]}\times \cdots\times U_{[r_n]}$. Then there is a natural map $\iota_{\mbf r} : X_{\textbf{r}} \to \widetilde{\mathcal{G}}_A$ as follows: $\iota_{\mbf r}(x_1,\cdots,x_n)=x_1\cdot \dots \cdot x_n$. Then we may combine the $\iota_{\mbf r}$ into a single map $\iota : X \to \widetilde{\mc G}_A$, where $X = \displaystyle \bigsqcup_{n \in \N} \bigsqcup_{\mbf{r} \in \Lambda_A^n} X_{\mbf r}$.

For any $\mbf{r}\in \Lambda_A^n$, let $K_{\mbf{r}}=\iota_{\mbf{r}}(X_{\textbf{r}})$. We topologize $K_{\mbf{r}}$ as a quotient of $X_{\mbf{r}}$ and define $B$ to be open in $\widetilde{\mathcal{G}}_A$ whenever $B \cap K_{\mbf r}$ is open in each $K_{\mbf{r}}$.

Note that each $U_{[r]}$ has a countable sequence of compact sets $U_{[r],m}$ whose union is $U_{[r]}$ which determine the topology of $U_{[r]}$. That is, $B$ is open in $U_{[r]}$ if and only if $B \cap U_{[r],m}$ is open in $U_{[r],m}$ for every $m \in \N$. Spaces admitting such a decomposition are called $k_\omega$ spaces. Any manifold is a $k_\omega$ space, as we can take an increasing union of closed balls around a fixed point $x_0$. If $\mbf r = (r_1,\dots,r_n)$, $X_{\mbf r}$ has decomposition as $X_{\mbf r} = \bigcup_{m \in \N} U_{[r_1],m} \times \dots U_{[r_n],m}$ coming from the decomposition of each $U_{[r]}$; and $\iota_{\mbf r}(U_{[r_1],m} \times \dots U_{[r_n],m})$ are the proper
compact sets to determine the topology of $K_{\mbf r}$.  So by Theorem 3.2 of \cite{ordman74-1}, the topology on $\widetilde{\mc G}_A$ inherited from the canonical topology on $X$ coincides with the free product topology. We summarize this is the following:

\begin{proposition}
\label{prop:topology-def}
The free product topology on $\widetilde{\mc G}_A$ is the quotient topology induced by the map $\iota$. In particular, a function $f : \widetilde{\mc G}_A \to Y$ is continuous if and only if $f \of \iota_{\mbf r} : X_{\mbf r} \to Y$ is continuous for every combinatorial pattern $\mbf r$.
\end{proposition}

The free product topology also makes $\widetilde{\mathcal{G}}_A$ a countable $CW$-complexes. For a good description of the $CW$-complex structure we refer the reader to \cite[Example 5.4]{ordman74-1}, which describes explicitly the first few cells and gluings for $S^1 * S^1$.

\subsubsection{The Free Group Topology vs. a Topology on Generators}
The topology on free products is from the 70's, but the connection with the spaces of Lyapunov paths and cycles remained undiscovered until very recently. Papers first applying this method did not topologize the path group, but instead looked only at the generators of the cycle group. The first application of the method was to the case of $SL(d,\R)$ in \cite{dk2011}, where the group $\mc C_A / \mc S_A$ has an explicit description in terms of generators and relations by the classical Matsumoto description of $K_2(\R)$. There, the generators fit into a single combinatorial pattern, and it was shown that $\mc S_A$ is dense in that combinatorial pattern, which implies trivialization of the periodic cycle functional everywhere. The free product topology was first used for cocycle and local rigidity in \cite{vinhage15}.

It should be true that $\mc S_A$ is dense in $\mc C_A$ for all groups, but this is only verified in certain special semisimple groups using ad-hoc techniques in a case-by-case basis (see \cite{damjanovic07}, \cite{zwang-1} and \cite{zwang-2}).

\subsection{Main Technical Step}
The group $\mc S_A$ will play an important role because it is the largest subgroup on which we can directly prove vanishing of the periodic cycle functional (see Section \ref{sec:18}). By construction, the periodic cycle functional takes the concatenation of paths to the product of the values on each path. The structure of $\mc C_A / \mc S_A$ is therefore the most important step in the paper.

An abelian topological group $S$ is called {\it minimally almost periodic} if any homomorphism from $S$ to a Lie group is trivial.

\begin{remark}
Such groups are very exotic, since the Pontryagin duality theorem tells us that no locally compact group can be minimally almost periodic. Such groups are known to exist, but are always infinite dimensional (see Theorem \ref{th:2}). This definition only holds if the group is abelian, otherwise one needs to insist on no homomorphisms to $U(n)$ for every $n$.
\end{remark}

The crucial step in proving the main Theorem \ref{thm:main} is:
\begin{theorem}\label{th:3}
Suppose $\alpha_A$ satisfies condition \starref. Then $\mc C_A/ \mc S_A$ is a minimally almost periodic abelian topological group with the topology inherited from $\widetilde{\mathcal{G}}_A$.
\end{theorem}

To prove Theorem \ref{th:3} the principal ingredient is to show that $\mc C_A/ \mc S_A$ is central in $\widetilde{\mc G}_A / \mc S_A$, which is the main technical result of the algebraic part of the paper. Recall Definition \ref{de:4}. Then:

\begin{theorem}
\label{thm:central}
Suppose $\alpha_A$ satisfies condition \starref. Then $\widetilde{\mathcal{G}}_A/ \mc S_A$ is a perfect central extension of $\mathcal{G}$.
\end{theorem}

\begin{remark}\label{re:3}
From Remark \ref{re:2} and Fact \ref{fact:6} we see that Theorem \ref{thm:central} holds for $\mc G'$ when $\mc G'$ is a perfect central extension or factor of $\mathcal{G}$. That is, to prove Theorem \ref{thm:central}, it suffices to prove that the theorem holds for any Lie group $S$ with $\text{Lie}(S)=\text{Lie}(\mathcal{G})$.

If $\widetilde{\mathcal{G}}_A/ \mc S_A$ is a perfect central extension of $\mathcal{G}_0$, then the kernel is $\mc C_A/ \mc S_A$, which implies that $\mc C_A/ \mc S_A$ is abelian.
\end{remark}

If $A=A_0$ then a similar result to Theorem \ref{thm:central} is proved in \cite{deodhar78} when $G$ is simple; and Theorem \ref{th:3} for the semisimple $G$ is proved in \cite{vinhage15}. This theorem is an extension of these ideas. However, the treatment of undetected roots requires new methods. Furthermore, Theorems \ref{th:3} and \ref{thm:central} have only been proven in semisimple cases previously, extending them to the twisted cases also requires new techniques, see Section \ref{sec:alg3}.

\begin{remark}\label{re:2}
By using Baker-Campbell-Hausdorff formula we see that relations $R_{\alpha,\beta,A}$ also hold in $G'$.
Hence $\pi_{G'}$
descends to $\pi_{G'} :\widetilde{\mathcal{G}}_A /\mc S_A\to G'$.
\end{remark}

We end this part by introducing a notation and a result which will be frequently used later. It holds in the context of both semisimple and twisted cases.

\begin{definition}
\label{def:Aadmissible} A subset $\Phi \subset \Lambda_A$ is called $A$-admissible if:

\begin{enumerate}[(i)]
\item whenever $r,r' \in \Phi$ and $r + r' \in \Lambda_A$, $r + r' \in \Phi$; and

\medskip
\item if $r\in\Phi$ and $-kr\in \Lambda_A$, $k\in \R_+$, then $kr\notin \Phi$.
\end{enumerate}
\end{definition}

Observe that $C(r,r')$ is $A$-admissible for every $r$ and $r'$ which are not negatively proportional over $A$. If $\Phi \subset \Lambda_A$, let $U_\Phi$ denote the subgroup $\prod_{\mu \in \Phi}U_{[\mu]} \subset \widetilde{\mc G}_A / \mc S_A$ (see \seref{se:2} of Section \ref{sec:14}). We have the following lemma whose proof is identical to that of Lemma 1.10 in \cite{deodhar78}:

\begin{lemma}
\label{lem:detected-comm}
If $\Phi$ is $A$-admissible, $\pi_1|_{U_\Phi} :U_\Phi \rightarrow \mc G_0$ is an isomorphism onto its image.
\end{lemma}

\section{Algebraic Preparatory step II: symmetric space examples}\label{sec:alg2}
The subsequent discussion in this section will be devoted to the proof Theorem \ref{thm:central} when $\mathcal{G}=G$. Let $k=\RR$ or $\CC$.
By Remark \ref{re:3} we can further assume that $G$ is a direct product of $G_i$, $1\leq i\leq j$ where each $G_i$ is the $k$-rational points of an almost simple simply connected algebraic group $Q_i$ of $k$-rank $\geq 2$. We use $Q$ to denote $\Pi_{i=1}^jQ_i$. Then $G=Q_k$. For a subgroup $H$ of $Q$, let $H_{k}$
denote the set $H\bigcap Q_k$.  We construct universal central extension for $G$ in  the next two sections.

\subsection{Notations and preliminaries}\label{sec:1} In this part, we
follow nations and quote conclusions without proof fairly close to
\cite{dk2011}. Many of the definitions have already been made, but we will describe them here in the context of real and complex algebraic groups. The definitions will coincide with the Lie group definitions given in Section \ref{sec:9}.

 Let $G^+_{k}$ be the group
generated by the subgroups $U_k$, where $U$ runs through the set of unipotent $k$-split subgroups of $Q$. By assumption about $G$ we see that $G=G^+_{k}$ (see \cite{margulis91}).
 Let $\mathfrak{q}$ be the Lie algebra
of $Q$, $\mathfrak{g}=\mathfrak{q}_k$ be the $k$-points of $\mathfrak{q}$ and $S\subset Q$ be the $k$-split torus. Let $A_0=S_k$. For a subalgebra $\mathfrak{s}$ of $\mathfrak{q}$, let $\mathfrak{s}_{k}$
denote the set $\mathfrak{s}\bigcap \mathfrak{g}$.

Let $\Delta_{S}$ be the $k$-root system of $Q$ with respect to $S$ and let $\Delta_{A_0}$ be the restricted roots on $A_0$. For any $\phi\in \Delta_{A_0}$, denote by $\mathfrak{g}_{\phi}$ the corresponding root space; also let $\mathfrak{g}^{(\phi)}_k=(\sum_{j\in \R_+}\mathfrak{g}_{j\phi})\bigcap \mathfrak{g}$ and
$U_{[\phi]}$ be the corresponding subgroup of $Q_k$.

Set $\Delta^1_{A_0}=\{\alpha\in\Delta_{A_0}|\mu/2\notin \Delta_{A_0}\}$.  For $\mu\in
\Delta^1_{A_0}$, let $G^{\mu}$ be the subgroup
generated by $U_{[\alpha]}$ and $U_{[-\alpha]}$. Let $C(S)$ be the
centralizer of $S$ in $Q$ and $N(S)$ the normalizer of $S$ in $Q$
and $W_0=N(S)/C(S)$ be the Weyl group. Let $W\subset N(S)_k$ be a
complete representatives. $(N(S)\cap G^{\mu})/(C(S)\cap G^{\mu})$ is a group of order two which can be identified with a subgroup of $W_0$; and the corresponding element in $W$ is denoted by $w_{\alpha}$.  We also assume that $w_{\alpha}$ is so
chosen that for any $\mu\in \Delta_{A_0}$, $w_{\mu}\in N(S)\bigcap
G^{\mu}$ and has order $2$ (see Remark \ref{re:1} below). (Note that $w_{\mu}$ is being used to denote both the defined
automorphism and the reflection in the hyperplane orthogonal
to $\mu$).
\begin{remark}\label{re:1}
$w_{\mu}\in N(S)\bigcap
G^{\mu}$ can be obtained as follows: Let $X\in \mathfrak{g}_\mu$ be
a $\R$-rational(nilpotent) element. Let $x=\exp X\in U_{[\mu]}$. The Jacobson-Morosov theorem which asserts the
existence of an element $Y\in \mathfrak{g}_{-\mu}$ such that
$\{X,Y,[X,Y]\}$ span a three-dimension split Lie algebra over $k$. Let $y=$exp$(-Y)$ then $xyx=yxy$. Note that $y\in U_{[-\mu]}$ is uniquely determined. Set $w_\mu=xyx$ then we obtain one representative of $w_\mu\in W_0$. We denote by
$w_{\mu}(x)$ and call it \emph{the detailed expression} of $w_{\mu}$ if we want to emphasize the initial element $x$. Note that $w_{\mu}(x)=w_{-\mu}(y)$. Furthermore, for any $u\in N(S)_k$, $uw_{\mu}(x)u^{-1}=w_{u(\mu)}(uxu^{-1})$.
\end{remark}
If
$\mu\in \Delta^1_{A_0}$ and $x, x_1 \in U_{[\mu]} \setminus \set{e}$, consider
the element $h_\mu(x,x_1)= w_\mu(x)w_\mu(x_1)^{-1}$. Then since $w_\mu(x)$ is an element of order 2 in $W_0$, $h_\mu(x,x_1)\in C(S)\bigcap G_k^+$. Let $H_{\mu}$ be
the subgroup generated by these elements.

\subsection{Central extensions}\label{sec:8}
Let $S$ be an abstract group. A central extension  $(\theta,S')$
of $S$ is said to be universal if for any central extension $(\eta,
E')$ of $S$, there exists a unique homomorphism $\phi:S'\rightarrow
E'$ such that $\eta\circ\phi=\theta$. In the category of groups (not necessarily topological groups), a group admits a universal central extension if and only if it is perfect (for the proof of this and other elementary properties of a
universal central extension, one may refer to \cite[Section 7]{Steinberg2}).

\begin{theorem}[Theorem 1.9, \cite{deodhar78}]
\label{thm:deodhar-centr-ext}
If $G$ is as described in Section \ref{sec:1} and if $A = A_0$ and is almost simple, then $\widetilde{\mathcal{G}}_A/ \mc S_A$  is a universal central extension of $G$.
\end{theorem}

Throughout Sections \ref{sec:alg2} we use the following notations in $\widetilde{\mathcal{G}}_A/ \mc S_A$ for the sake of computation:
for any detected $r\in \Delta_{A_0}$ and any $u\in U_{[r]}$, let $\tilde{u}$ denote the corresponding element in $\widetilde{\mathcal{G}}_A/ \mc S_A$.
Set $\tilde{U}_{[r]}=\{\tilde{u}: u\in U_{[r]}\}$ for detected $r\in \Delta_{A_0}$ and $\tilde{U}_{[s]}=\{\tilde{u}: u\in U_{[s]}\}$ for $s\in \Delta_{A}$. Note that the first one projects to the $r$-root subgroup for root system $\Delta_{A_0}$ and the latter one projects to the $s$-root subgroup for root system $\Delta_{A}$, which is product of subgroups $\tilde{U}_{[\gamma]}$, for $\gamma\in \Delta_{A_0}$ such that $\gamma|_{A}=ks$ for some $k \in \R_+$ (also see \seref{se:3} of Section \ref{sec:14} for comparison).

$\tilde{w}_r$, $\tilde{h}_r(u,u_1)$ for $u,u_1\in
U_{[r]}$ and $\tilde{H}_r$ are obviously defined elements/sets of $\widetilde{\mathcal{G}}_A/ \mc S_A$ for detected $r\in \Delta_{A_0}$.

The next result illustrates the structure of $\ker(\widetilde{\mathcal{G}}_{A_0}/ \mc S_{A_0}\rightarrow G)$, where $G$ is as described in Section \ref{sec:1} and is almost simple. We will not need it in this section, but this technical description will appear in our analysis of twisted spaces:

\begin{theorem}[Corollary 1.13, \cite{deodhar78}] \label{le:1}
Let $\widetilde{W}$ be the subgroup of $\widetilde{\mathcal{G}}_{A_0}/ \mc S_{A_0}$ generated by $\big\{\tilde{w}_\mu(u):
\mu\in\Delta^1_{A_0}, u\in U_{[\mu]} \setminus \set{e}\big\}$. For $\mu\in\Delta^1_{A_0}$,
denote by $\tilde{H}_{\mu}$ the subgroup generated by
$\tilde{h}_{\mu}(v,v_1)$, $v, v_1\in U_{[\mu]} \setminus \set{e}$. Let
$\tilde{H}$ be he subgroup generated by $\{\tilde{H}_{\mu},
\mu\in\Delta^1_{A_0}\}$ and $\Delta$ denote the set of simple roots. Then
\begin{enumerate}[1.]
\item $\tilde{H}_{\mu},\mu\in\Delta^1_{A_0}$, is normal in
$\widetilde{H}$, and $\widetilde{H}$ is normal in $\widetilde{W}$.

\item $\tilde{H}=\prod_{\mu\in \Delta}\tilde{H}_{\mu}$

\item For the projection $\pi_G : \widetilde{\mathcal{G}}_{A_0}/ \mc S_{A_0}\rightarrow G$, $\ker(\pi_G)\subset \tilde{H}$.
\end{enumerate}
\end{theorem}
\begin{remark}
The above  theorem still holds if we get rid of the restriction that $G$ is almost simple.
\end{remark}
We end this part by the stating next result, which is a direct consequence of Theorem \ref{thm:deodhar-centr-ext}:
\begin{corollary}\label{cor:4}
If $G$ is as described in Section \ref{sec:1}. Then $\widetilde{\mathcal{G}}_{A_0}/ \mc S_{A_0}$ is a perfect central extension of $G$.
\end{corollary}

\subsection{Basic properties of Weyl elements in $\widetilde{\mathcal{G}}_{A}/ \mc S_{A}$}
From this part to the end of this section we study $\widetilde{\mathcal{G}}_{A}/ \mc S_{A}$. So unless especially stated, all $\tilde{\cdot}$ denote elements in $\widetilde{\mathcal{G}}_{A}/ \mc S_{A}$.

From Example \ref{sec:6} we see that the Weyl elements associated to the root system $\Delta_A$ are not transitive on $\Delta_A$. Furthermore, the Weyl elements generated from detected roots in $\Delta_{A_0}$ not necessarily preserve the root spaces of $\Delta_A$. To overcome this difficulty we consider the action of Weyl elements generated from detected roots in $\Delta_{A_0}$ on root spaces of $\Delta_{A_0}$ instead, by noting that the root spaces of $\Delta_A$ are products of those of $\Delta_{A_0}$ (see \seref{se:3} of Section \ref{sec:14}). So unless especially stated, we only consider subgroups of $\widetilde{U}_{[r]}$, for detected $r\in \Delta_{A_0}$ throughout this section (note that they are canonically subgroups of $\widetilde{U}_{[r|_A]}$). Since $G$ is a direct product of $G_i$, $1\leq i\leq j$, $\Delta_{A_0}=\bigcup_{i=1}^j \Delta_{A_0,i}$, where $\Delta_{A_0,i}$ is the (standard)root system for $G_i$, $1\leq i\leq j$. Then the restricted root system with respect to $A$ on $G_i$ is $\Delta_{A,i}=\{r\mid_A: r\in \Delta_{A_0,i}\}$, $1\leq i\leq j$.

We will frequently used the words ``proportional'' and ``nonproportional''. We say that $a$ is proportional to $b$, if $a=kb$, $k\in\RR$; otherwise,
$a$ is nonproportional to $b$, i.e., $a\neq kb$, for any $k\in\RR$. Note that since in this context, $a$ and $b$ may be interpreted as linear functionals on $A$ or $A_0$, so we will often emphasize by saying that $a$ and $b$ are proportional {\it over $A$}, or $a|_A$ and $b|_A$ are proportional. If we do not specify, all functionals will be considered functionals on $A_0$ in this section.

We first state a simple but useful fact:

\begin{fact}\label{fact:5}
For any $r\in\Delta_{A_0}$ (detected or undetected) and any $w\in W_0$, $w$ can be written as a product of $w_1w_2$, where $w_2$ is a product of $w_s$, such that $s$ is not proportional to $r$ over $A$ and $w_1$ is a product of $w_{s'}$ such that ${s'}$ is proportional to $r$ over $A$.
\end{fact}

\begin{proof}
Note that $w_{s_1}w_{s_2}=w_{w_{s_1}(s_2)}w_{s_1}$ for any $\gamma_1,\,\gamma_2\in\Delta_{A_0}$. Further, note that if $s_2|_A$ is proportional to $r|_A$, but $s_1|_A$ is not proportional to $r|_A$, then $w_{w_{s_1}(s_2)}$ is not proportional to $r|_A$ since $w_{s_1}(s_2)=ns_1+s_2$ for some $n\in\ZZ$.  Hence we may push all of the roots which restrict to exponents proportional to $r|_A$ to the right hand side of the product.
\end{proof}

As a consequence of Lemma \ref{lem:detected-comm}, we get the following lemma which is important for the sequel:
\begin{lemma}
\label{lem:weyl-conj}
If $r_1,r_2, w_{r_1}(r_2) \in \Delta_{A_0}$ are all detected and $v \in U_{[r_2]}$, then
\begin{align*}
 \widetilde{w}_{r_1}\tilde{v}\widetilde{w}_{r_1}^{-1} =\tilde{v_1}\in \tilde{U}_{[w_{r_1}(r_2)]}.
\end{align*}
\end{lemma}

\begin{proof}
We will first prove the case in which $r_1$ and $r_2$ are nonproportional over $A$. Let $\Phi'= \set{ (ir_1 + jr_2)|_A : i \in \NN, j \in \Z}$. Then one can  verify that $\Phi'$
is $A$-admissible. Thus we can apply Lemma \ref{lem:detected-comm} to $\Phi'$. Now, $\widetilde{w}_{r_1}\tilde{v} \widetilde{w}_{r_1}^{-1} \in\prod_{r \in \Phi'}^* \tilde{U}_{[r]}$ because of the relations $R_{\mu,\nu,A}$ and $\tilde{U}_{w_{r_1}(r_2)}\in \prod_{r \in \Phi'}^* \tilde{U}_{[r]}$ since $w_{r_1}(r_2)\mid_A=(nr_1+r_2)|_A\in \Phi'$, $n\in\ZZ$. Lemma \ref{lem:detected-comm} immediately shows that the result holds in this case.

Now we move to the case when $r_1$ and $r_2$ are propotional over $A$. Note that:
\begin{itemize}
  \item[($\dagger$)] \label{for:40} for any $s\in \Delta_{A_0}$, if $s$ is proportional to $r_2$ over $A$, then $w_s(r_2)$ is proportional to $r_2$ over $A$;

  \smallskip
  \item[($\ddagger$)] \label{for:44} if $s$ is not proportional to $r_2$ over $A$, but $w_s(r_2)$ is proportional to $r_2$ over $A$, then $w_s(r_2)|_A = r_2|_A$.
\end{itemize}
Suppose $r_2\in \Delta_{A_0,i}$. We wish to find detected $r,\,r\in \Delta_{A_0}$ with $r$, $r'$ and $r_2$ pairwise nonproportional over $A$ such that
$r_2 = w_r(r')$.  Assume that for each $r \in \Delta_{A_0}$ $w_r(r_2)$ is proportional to $r_2$ over $A$. Then using Fact \ref{fact:5} we see that any element of $W_0$ acts on $r_2$ first by elements non-proportional to $r_2$ over $A$, which must fix it by ($\ddagger$); then we may act by those proportional. This shows that the orbit of $r_2$ in $W$ contains only roots proportional to $r_2$ over $A$ by ($\dagger$). Since the orbit of any root must contain a basis of $A^*$, this shows that all roots in $\Delta_{A,i}$ are proportional to each other, which contradicts the fact that $A$ is genuinely higher rank. So we conclude that we may write $r_2 = w_r(r')$ with $r$ and $r'$ detected and linearly independent over $A$.

Thus if $\widetilde{v} \in U_{[r_2]}$, we can write $\tilde{v }= \widetilde w_r\tilde{v_1}\widetilde w_r^{-1}$ for a suitable $\tilde{v_1} \in \tilde{U}_{[r']}$ by earlier part of the proof. Then:
\begin{align}\label{for:34}
 \widetilde{w}_{r_1}\tilde{v}\widetilde{w}_{r_1}^{-1} & =\widetilde{w}_{r_1}(\widetilde w_r \tilde{v_1} \widetilde w_r^{-1} )\widetilde{w}_{r_1}^{-1}
  = \left(\widetilde w_{r_1} \widetilde w_r \widetilde w_{r_1}^{-1}\right)  \widetilde w_{r_1} \tilde{v_1} \widetilde w_{r_1}^{-1} \left(\widetilde w_{r_1}\widetilde w_r \widetilde w_{r_1}^{-1}\right)^{-1}.
 \end{align}
Since $r$ and $r_2$ nonproportional over $A$, by earlier part of the proof and Remark \ref{re:1}, it follows that
\begin{align*}
 \widetilde w_{r_1}\widetilde w_r(x)\widetilde w_{r_1}^{-1}=\widetilde w_{w_{r_1}(r)}(w_{r_1}xw_{r_1}^{-1}).
\end{align*}
Here $w_r(x)$ is the detailed expression of $w_r$ as described in Remark \ref{re:1} (this is a special case of the subsequent Corollary).  Also,

\begin{align*}
 \widetilde w_{r_1} \tilde{v}_1 \widetilde w_{r_1}^{-1} = \tilde{v}_2 \in \tilde{U}_{[w_{r_1}(r')]}.
\end{align*}

We wish to see that $w_{r_1}(r)$ and $w_{r_1}(r')$ are not proportional over $A$. Note that
\begin{align*}
 w_{r_1}(r')=w_{r_1}(w_{r}(r_2))=w_{w_{r_1}(r)}w_{r_1}(r_2).
\end{align*}
Furthermore, we see that $w_{r_1}(r_2)\mid_A$ is proportional to $r_2\mid_A$ while $w_{r_1}(r)|_A = r|_A + nr_1|_A$ is not proportional to $r_2|_A$ since $r_1|_A$ is proportional but $r|_A$ is not by assumption. Then $w_{r_1}(r')|_A=w_{w_{r_1}(r)}w_{r_1}(r_2)|_A$ is not proportional to $w_{r_1}(r)|_A$.

Applying above discussion for \eqref{for:34} we see that
\begin{align*}
 \widetilde{w}_{r_1} \tilde{v} \widetilde{w}_{r_1}^{-1}=\widetilde w(\widetilde w_{r_1}\tilde{v_1} \widetilde w_{r_1}^{-1})
 \widetilde w^{-1}\in \tilde{U}_{[w_{r_1}(r_2)]}
 \end{align*}
 where $\widetilde w=\widetilde w_{w_{r_1}(r)}(w_{r_1}xw_{r_1}^{-1})$, by noting that
 \begin{align*}
  w_{w_{r_1}(r)}(w_{r_1}(r'))= (w_{r_1} w_r w_{r_1}^{-1}) w_{r_1}(r') = w_{r_1}(w_r(r')) = w_{r_1}(r_2).
 \end{align*}
\end{proof}
As an immediate corollary of the above lemma and Remark \ref{re:1} we have:

\begin{corollary}
\label{cor:weyl-conj}
If $r_1,r_2, w_{r_1}(r_2) \in \Delta_{A_0}$ are all detected then
\begin{align*}
\widetilde{w}_{r_1}\widetilde{w}_{r_2}(x)\widetilde{w}_{r_1}^{-1} =\widetilde{w}_{r_2}(w_{r_1}xw_{r_1}^{-1}),\qquad e\not=x\in U_{[r_2]}.
\end{align*}
\end{corollary}

Lemma \ref{lem:weyl-conj} and Corollary \ref{cor:weyl-conj} drive the remaining arguments of the section, as they allow us to use the Weyl group to conjugate root spaces when they take an inconvenient form. Before proceeding further with establishing technical results for recovering undetected root subgroups, we prove
the following results which are very useful for the discussion.

\begin{lemma}
\label{lem:det-gen}
If $A$ satisfies the genuinely higher rank condition then:
\begin{enumerate}
  \item \label{for:16} for any $r\in \Delta_{A_0}$ there exists a detected root $r_1\in \Delta_{A_0}$ such that $r_1$ and $r$ are not proportional over $A$ and $w_{r_1}(r)$ is detected;

      \medskip
  \item \label{for:35} for any $r\in \Delta_{A_0}$ if $r$ is undetected there exist at least two detected roots $r_1,\,r_2\in \Delta_{A_0}$ such that $r_1$ and $r_2$ are not proportional over $A$ and $w_{r_i}(r)$, $i=1,\,2$, are both detected.
\end{enumerate}

\end{lemma}

\begin{proof}
Proof of \eqref{for:16}: Note that if $s \in \Delta_{A_0}$ is detected and $s$ is not proportional to $r$ over $A$, then either $w_{s}(r) = r$ or $w_{s}(r)$ is detected. Suppose $r\in \Delta_{A_0,i}$.  Using Fact \ref{fact:5}, if for every detected $s$, $w_{s}(r) = r$, then the Weyl orbit of $r$ contains roots only proportional to $r$ over $A$. This shows that all roots in $\Delta_{A,i}$ are proportional to $r$, since the Weyl orbit of $r$ contains a basis of $A^*$. This is a contradiction to genuinely higher-rank property of $A$. Then we get the result.

\smallskip
Proof of \eqref{for:35}: Previous result shows the existence of one root. Now we show the existence of the second root. Suppose that $r_1$ has the property that $w_{r_1}(r)$ is detected. If any other $r_2$ with the same property is proportional to $r_1$, then the Weyl orbit of $r$ must contain roots only proportional to $r_1$ by Fact \ref{fact:5} (act first by roots not proportional to $r_1$ over $A$, which must fix $r$ by assumption). This shows all roots in $\Delta_{A,i}$ are proportional to each other (since $r$ restricts to 0 over $A$), which contradicts the fact that $A$ is higher rank.
\end{proof}

\begin{corollary}
\label{cor:det-gen}
For any $r\in \Delta_{A_0}$ the subgroups $U_{[s]}$, $s\in\Delta_{A_0}$ with $s$ not proportional to $r$ over $A$ generate the group $G$.
\end{corollary}

\begin{proof}
Note that if we had chosen the subgroups to vary over all of the roots $\Delta_{A_0}$ the result would follow classically. So we just need to show that  $U_{[r]}$ is generated by elements in $U_{[s]}$, with $s$ not proportional to $r$ over $A$ (note that $U_{[r]}$ may be undetected). By \eqref{for:16} of Lemma \ref{lem:det-gen}, we can find detected $r'\in\Delta_{A_0}$ such that $w_{r'}(r)$ is also detected and $r'$ is not proportional to $r$ over $A$. Since $w_{r'}U_{[w_{r'}(r)]}w_{r'}^{-1} = U_{[r]}$ $w_{r'}$ is constructed from elements of $U_{[r']} \cup U_{[-r']}$ (see Remark \ref{re:1}), we get the conclusion.
\end{proof}

We are now in a position to construct the elements inside undetected root subgroups for $\Delta_{A_0}$. Next result shows that not only are we able to construct undetected root subgroups by using detected ones, but most importantly there are at least two equivalent ways to do so, whose role will be quite clear in next part.

\begin{lemma}
\label{lem:isom-gen-ind}
Suppose $r_i,\,s_i\in \Delta_{A_0}$, $i=1,\,2$ are detected, and $r\in \Delta_{A_0}$ is undetected. Also suppose $w_{r_1}(r_2) = w_{s_1}(s_2) = r$.  For any $u \in U_r$, denote by $w_{r_1}^{-1}uw_{r_1}=u_1$ and $w_{s_1}^{-1}uw_{s_1}=u_2$, then $\widetilde w_{r_1}\tilde{u}_1\widetilde{w}_{r_1}^{-1} = \widetilde w_{s_1}\tilde{u}_2 \widetilde{w}_{s_1}^{-1}$.
\end{lemma}

\begin{proof}
At first we note that $r_1$ and $r_2$ must be proportional over $A$, as are $s_1$ and $s_2$.

First we consider the case when  $r_1$ and $s_1$ are nonproportional over $A$. Then $r_2$ and $s_1$ are also nonproportional over $A$. This implies that both $w_{s_1}(r_1)$ and $w_{s_1}(r_2)=w_{s_1}w_{r_1}(r)$ are detected. Denote by $\tilde{v}_1=\widetilde w_{s_1}^{-1}\tilde{u}_1\widetilde w_{s_1}$, so that:
\begin{equation}\begin{array}{rcl}
\label{eqar:1}
  \widetilde w_{r_1}\tilde{u_1}\widetilde{w}_{r_1}^{-1} & = & \widetilde w_{r_1}\widetilde w_{s_1}\tilde{v_1}\widetilde w_{s_1}^{-1}\widetilde{w}_{r_1}^{-1}\\
  &\overset{(1)}{=} & \widetilde w_{s_1} \left( \widetilde w_{s_1}^{-1} \widetilde w_{r_1}(x)
   \widetilde w_{s_1}\right)\tilde{v_1}\left( \widetilde w_{s_1}^{-1} \widetilde{w}_{r_1}(x)\widetilde w_{s_1}\right)^{-1} \widetilde w_{s_1}^{-1}\\
   &\overset{(2)}{=} & \widetilde w_{s_1}\widetilde w_{w_{s_1}(r_1)}(x')\tilde{v}_1\widetilde w_{w_{s_1}(r_1)}(x')^{-1} \widetilde w_{s_1}^{-1},
\end{array}
\end{equation}

Here in $(1)$ we insert an element and its inverse on the left and right, and add the detailed expression of $w_{r_1}$ as described in Remark \ref{re:1} for emphasis; in $(2)$ we used Corollary \ref{cor:weyl-conj}, with $x' = w_{s_1}^{-1}xw_{s_1}$.

Since $\tilde{v}_1 \in \tilde{U}_{[w_{s_1}w_{r_1}(r)]}=\tilde{U}_{[w_{s_1}(r_2)]}$ and

\begin{align*}
 w_{w_{s_1}(r_1)}\big(w_{s_1}(r_2)\big)=w_{s_1}w_{r_1}w_{s_1}^{-1}w_{s_1}(r_2)=w_{s_1}w_{r_1}(r_2)=w_{s_1}(r)=s_2,
\end{align*}

which implies that $w_{w_{s_1}(r_1)}\big(w_{s_1}(r_2)\big)$ is detected and hence $w_{s_1}(r_1)$ and $w_{s_1}(r_2)$ are nonproportional over $A$ (since if they were proportional, they would both be proportional to $s_2$ over $A$). Then Lemma \ref{lem:weyl-conj} shows that
\begin{align*}
 \widetilde w_{w_{s_1}(r_1)}(x')\tilde{v}_1\widetilde w_{w_{s_1}(r_1)}(x')^{-1}\in \tilde{U}_{[s_2]}.
\end{align*}
 Noting that
\begin{align*}
   w_{w_{s_1}(r_1)}(x')v_1w_{w_{s_1}(r_1)}(x')^{-1}=u_2,
\end{align*}
we get
\begin{align*}
 \widetilde w_{w_{s_1}(r_1)}(x')\tilde{v}_1\widetilde w_{w_{s_1}(r_1)}(x')^{-1}=\tilde{u}_2
\end{align*}
by Lemma \ref{lem:detected-comm}. So we conclude from this and \eqref{eqar:1} that $\widetilde{w}_{s_1} \tilde{u}_2 \widetilde{w}_{s_1}^{-1} = \widetilde w_{r_1} \tilde{u}_1 \widetilde{w}_{r_1}^{-1}$, which was our goal.

If $r_1\mid_A$ and $s_1\mid_A$ are proportional, by using Lemma \ref{lem:det-gen} we choose detected $\gamma\in \Delta_{A_0}$ with $\gamma\mid_A$ not proportional to $r_1\mid_A$ such that
$w_{\gamma}(r)=\gamma_1$ is detected. Set $w_{\gamma_1}^{-1}uw_{\gamma_1}=v$. Then:
\begin{align*}
\widetilde w_{r_1}\tilde{u_1}\widetilde{w}_{r_1}^{-1}=\widetilde w_{\gamma_1}\tilde{v}\widetilde{w}_{\gamma_1}^{-1}
\qquad\text{and}\qquad \widetilde w_{s_1}\tilde{u_2} \widetilde{w}_{s_1}^{-1}=\widetilde w_{\gamma_1}\tilde{v}\widetilde{w}_{\gamma_1}^{-1}
\end{align*}
by previous discussion. Then we get the result immediately. Hence we finish the proof.

\end{proof}

 \subsection{An Equivariant Isomorphism of $\widetilde{\mathcal{G}}_{A_0}/ \mc S_{A_0}$  and $\widetilde{\mathcal{G}}_{A}/ \mc S_{A}$}
 Next we prove the following proposition which plays a key role in the proof of Theorem \ref{thm:central}. The subsequent discussion in this section will be devoted to the proof of this result.
 \begin{proposition}
\label{prop:restricted-factor}
There is a surjective homomorphism $\Psi : \widetilde{\mathcal{G}}_{A_0}/ \mc S_{A_0} \to \widetilde{\mathcal{G}}_{A}/ \mc S_{A}$ such that the following diagram commutes:

\[
\begin{tikzcd}
\widetilde{\mathcal{G}}_{A_0}/ \mc S_{A_0} \arrow{rr}{\Psi} \arrow{dr} & & \widetilde{\mathcal{G}}_{A}/ \mc S_{A} \arrow{dl} \\
 & G &
\end{tikzcd}
\]
\end{proposition}

For any $r\in \Delta^1_{A_0}$ (see Section \ref{sec:1}), we define a map $\psi_r:U_{[r]}\rightarrow \widetilde{\mathcal{G}}_{A}/ \mc S_{A}$ as follows: if $r$ is detected then $\psi_r(u)=\tilde{u}$ for any $u\in U_{[r]}$ (note that $\widetilde{u}$ represents an element of $\widetilde{U}_{r|_A}$). If $r$ is undetected, choose detected $s \in \Delta_{A_0}$ such that $w_s(r)$ is detected, (see Lemma \ref{lem:det-gen}) and then set $\psi_r(u)=\widetilde w_{s}\tilde{u_1}\widetilde{w}_{s}^{-1}$ where $u_1=w_{s}^{-1}u w_{s}\in U_{[w_s(r)]}$. Lemma \ref{lem:isom-gen-ind}
shows that $\psi_r(u)$ is well-defined and is independent of the choice of $\gamma$. That we may choose a uniform conjugation implies that this is a homomorphism. Then $\{\psi_r,r\in \Delta^1_{A_0}\}$ induces a map $\Psi :\widetilde{\mathcal{G}}_{A_0}\rightarrow \mc G_A / \mc S_A$ by the universal property of free products.

To prove Proposition \ref{prop:restricted-factor}, it suffices to show that $\Psi(\mc S_{A_0})=e\in \widetilde{\mathcal{G}}_{A}/ \mc S_{A}$, so that $\Psi$ descends to a surjective homomorphism $\Psi : \widetilde{\mathcal{G}}_{A_0}/ \mc S_{A_0} \to \widetilde{\mathcal{G}}_{A}/ \mc S_{A}$. Recall notation in Section \ref{sec:15}.  Denoted by $\tilde{R}_{\alpha,\beta,A_0}(u,v)$ the corresponding element of $R_{\alpha,\beta,A_0}(u,v)$  in $\widetilde{\mathcal{G}}_{A_0}$ or just $\tilde{R}(u,v)$ for simplicity.

Note that even though $[u,v]$ may not be inside a single $U_{[r]}$ for $r \in \Delta_A$, we may choose a canonical presentation using \eqref{eq:comm-eq}. Then for any $r_1,r_2 \in \Delta_{A_0}$ which are nonproportional over $A$, the cone $C(r_1,r_2)$ is $A$-admissible (see Definition \ref{de:1}). So we fix an ordering on the roots appearing in $C(r_1,r_2)$ and present them this way.

We prove a technical Lemma about the maps $\psi_\mu$ and their composite $\Psi$:

\begin{lemma} \label{fact:8}For any $\mu_1,\,\mu_2\in \Delta_{A_0}$, if $\mu_2$ is detected and $\mu_1$ is undetected, then for every $x \in U_{[\mu_1]}$:
\begin{align}\label{for:38}
 \widetilde{w}_{\mu_2}\psi_{\mu_1}(x)\widetilde{w}_{\mu_2}^{-1}=\psi_{w_{\mu_2}(\mu_1)}(w_{\mu_2}xw_{\mu_2}^{-1}).
\end{align}
\end{lemma}

\begin{proof}
If $w_{\mu_2}(\mu_1)$ is detected, then $\psi_{w_{\mu_2}(\mu_1)}(w_{\mu_2}xw_{\mu_2}^{-1})\in \tilde{U}_{[w_{\mu_2}(\mu_1)]}$. We then get \eqref{for:38} almost immediately, since to define $\psi_{\mu_1}$, we choose any root $\mu' \in \Delta_{A_0}$ such that $w_{\mu'}(\mu_1)$ is detected. So we may choose $\mu' = \mu_2$, and \eqref{for:38} is exactly the definition.

Suppose $w_{\mu_2}(\mu_1)$ is undetected. In this case, since $w_{\mu_2}(\mu_1) = \mu_1 + n\mu_2$ for some $n$, we conclude that $w_{\mu_2}(\mu_1) = \mu_1$. Then we choose some way to present $\psi_{\mu_1}(x)$, writing $\mu_1 = w_\theta(\theta')$. We may choose $\theta$ (and $\theta'$) to be nonproportional to $\mu_2$ over $A$. Then

\begin{align*}
\psi_{\mu_1}(x) = \widetilde w_\theta \widetilde x_1 \widetilde w_\theta^{-1},
\end{align*}
where $x_1 = w_\theta^{-1}xw_\theta$.

Noting that $w_\theta w_{\mu_2}(\mu_1)= w_\theta (\mu_1)= \theta'$ is detected, we set
\begin{align*}
 \psi_{w_{\mu_2}(\mu_1)}(w_{\mu_2}xw_{\mu_2}^{-1})=\tilde{w}_\theta\tilde{x_2}\tilde{w}_\theta^{-1},\quad \text{where }x_2=w_\theta^{-1} w_{\mu_2}xw_{\mu_2}^{-1}w_\theta
 \in \tilde{U}_{[w_\theta w_{\mu_2}(\mu_1)]}.
\end{align*}
This shows that
\begin{align*}
\tilde{w}_{\mu_2} \psi_{w_{\mu_2}(\mu_1)}(x)\tilde{w}_{\mu_2}^{-1}=\tilde{w}_{\mu_2} \tilde{w}_\theta\tilde{x_1}\tilde{w}_\theta^{-1}\tilde{w}_{\mu_2}^{-1}=
\tilde{w}_\theta\tilde{w}_{w_\theta(\mu_2)}(z)\tilde{x_1}\tilde{w}_{w_\theta(\mu_2)}(z)^{-1}\tilde{w}_\theta^{-1}
\end{align*}
where $z=w_\theta^{-1} tw_\theta$ if $w_{\mu_2}(t)$ is the detailed expression of $w_{\mu_2}$.

So we wish to show that we can apply the inner conjugation in the extension $\widetilde{\mc G_A} / \mc S_A$. Note that $w_{w_\theta(\mu_2)}(\theta')=w_\theta w_{\mu_2}w_\theta^{-1}(\theta')=w_\theta w_{\mu_2}(\mu_1)= w_\theta(\mu_1) = \theta'$ is detected, then we have
\begin{align*}
 \tilde{w}_{w_\theta(\mu_2)}(z)\tilde{x_1}\tilde{w}_{w_\theta(\mu_2)}(z)^{-1}\in
\tilde{U}_{[w_\theta(\mu_2)(\theta')]}.
\end{align*}
Since $\tilde{w}_{w_\theta(\mu_2)}(z)\tilde{x_1}\tilde{w}_{w_\theta(\mu_2)}(z)^{-1}$ projects to $x_2$, then by Lemma \ref{lem:detected-comm} we see that
$\tilde{w}_{w_\theta(\mu_2)}(z)\tilde{x_1}\tilde{w}_{w_\theta(\mu_2)}(z)^{-1}=\tilde{x_2}$, which implies that \eqref{for:38} holds again. This proves the Lemma.
\end{proof}

\begin{proposition}
\label{prop:undetected-comm}
For any $u \in U_{[r]}$ and $v \in U_{[r']}$ where $r,\,r'\in \Delta^1_{A_0}$ and $r \not= -r'$ (over $A_0$), let $\tilde{u}$ and $\tilde{v}$ denote the corresponding elements in $\widetilde{\mathcal{G}}_{A_0}$. Then
\begin{align}\label{for:39}
\Psi([\tilde{u},\tilde{v}])=\Psi(\tilde{R}(u,v)).
\end{align}
\end{proposition}

Note that $r\not= -r$ is equivalent to the condition that $r$ and $r'$ are not negatively proportional over $A_0$, since we have assumed that they are in $\Delta^1_{A_0}$. Since such relations and their conjugates generate $\mc S_{A_0}$ we immediately get that $\Psi(\mc S_{A_0})=e$.

\begin{proof}
We divide the proof into four cases:

\begin{enumerate}[{\bf Case }($a$)]
\item $r$ and $r'$ are both detected, and are not negatively proportional over $A$

\smallskip
\item $r$ and $r'$ are both detected and are negatively proportional over $A$ (but not over $A_0$)

\smallskip
\item $r$ is detected, and $r'$ is undetected

\smallskip
\item Both $r$ and $r'$ are undetected
\end{enumerate}
First we consider \textbf{case $(a)$}: if $r$ and $r'$ are both detected and $r$ is not negatively proportional to $r'$ over $A$, then the conclusion follows directly from Lemma \ref{lem:detected-comm}.

\smallskip

\noindent Next we consider  \textbf{case $(b)$}: $r$ and $r'$ are both detected and $r|_A=\lambda r'|_A$ for some $\lambda<0$.
We claim that there exists $s \in \Delta_{A_0}$ such that $s|_A$ is not proportional to $r|_A$ and $\inner{s}{r} \not= \lambda \inner{s}{r'}$ (with the canonical inner product on the root system $\Delta_{A_0}$. Set
$E=\{\eta\in \Delta_{A_0} : \eta|_A = kr|_A,k\in\RR\}$. Note that $E$ contains all undetected roots.
Indeed, Lemma \ref{lem:det-gen} implies that every root in $\Delta_{A_0}$ is a linear combination of roots in $\Delta_{A_0}\backslash E$. Hence $\Delta_{A_0}\backslash E$ contains a basis of $A_0^*$. So if $\inner{s}{r} = \lambda\inner{s}{r'}$ for every $s \in \Delta_{A_0}\backslash E$, we would conclude that $r= \lambda r'$, which is a contradiction to the assumption that $r$ and $r'$ are not negatively proportional over $A_0$.

So we may choose such an $s$. Then if $k w_s(r)|_A=w_s(r')|_A$ for some $k\in\RR$, we get
\begin{align*}
 kr|_A-2k\frac{\inner{s}{r}}{\inner{s}{s}}s|_A=kw_s(r)|_A&=w_s(r')|_A=r'|_A-2\frac{\inner{s}{r'}}{\inner{s}{s}}s|_A,\text{ then}\\
 (k\lambda-1)r'|_A&=\left(\frac{2k\inner{s}{r}-2\inner{s}{r'}}{\inner{s}{s}}\right)s|_A.
\end{align*}
By assumption $r'|_A$ is not proportional to $s|_A$, so both sides must be zero. Then we have
\begin{align*}
  k\lambda=1\text{ and }2k\inner{s}{r}=2\inner{s}{r'},
\end{align*}
which implies $\inner{s}{r} = \lambda \inner{s}{r'}$. This is a contradicts the choice of $s$. Furthermore, since $s$ is not proportional to $r$ over $A$, then both $w_s(r')$ and $w_s(r)$ are detected. This shows that $w_s(r)$ and $w_s(r')$ are nonproportional over $A$.

We compute:
\begin{align*}
 \tilde{w}_s\cdot \Psi([\tilde{u},\tilde{v}])\cdot \tilde{w}_s^{-1}&\overset{(1)}{=}\tilde{w}_s[\tilde{u},\tilde{v}]\tilde{w}_s^{-1}=[\widetilde w_s\tilde{u}\widetilde w_s^{-1}, \widetilde w_s \tilde{v}\widetilde w_s^{-1}].
\end{align*}
It is clear that $\tilde{u}$ and $\tilde{v}$ represent elements in $\widetilde{\mathcal{G}}_{A_0}$ on the left side of $(1)$; and represent elements in $\widetilde{\mathcal{G}}_{A}/ \mc S_{A}$ on the right side.

We note that $w_s(r)|_A$ and $w_s(r')|_A$ are nonproportional by previous discussion. Then $w_s(C(r,r')) = C(w_s(r),w_s(r'))$ is $A$-admissible by linearity of reflections and projections, and the independence of $w_s(r)$ and $w_s(r')$. By using Lemma \ref{lem:weyl-conj} we get

\begin{align}\label{for:36}
 \tilde{w}_s \Psi([\tilde{u},\tilde{v}])\tilde{w}_s^{-1}\in [\tilde{U}_{[w_s(r)]},\tilde{U}_{[w_s(r')]}]\in \prod_{\eta\in w_s(C(r,r'))}\tilde{U}_{[\eta]},
\end{align}
If each $\omega\in C(r,r')$ is detected, then by using Lemma \ref{lem:weyl-conj} again we obtain:
\begin{align*}
 \tilde{w}_s\Psi(\tilde{R}(u,v))\tilde{w}_s^{-1}&\in \tilde{w}_s \left(\prod_{\omega\in C(r,r')}\tilde{U}_{[\omega]}\right)\tilde{w}_s^{-1}=\prod_{\eta\in w_s(C(r,r'))}\tilde{U}_{[\eta]}.
\end{align*}
Observe that since $w_s(r)$ and $w_s(r')$ are linearly independent, every
\begin{align*}
\mu \in C(w_s(r),w_s(r')) = w_s(C(r,r'))
\end{align*}
is detected. Hence if $\mu \in C(r,r')$ is undetected, $w_s(\mu)$ is detected.

If there exists undetected $\mu\in C(r,r')$, we write $\mu=w_s(w_s(\mu))$. Note that since $w_s(\mu)$ is detected, for any $x\in U_{[\mu]}$, we may set $\psi_\mu(x)=\widetilde w_{s}\tilde{x_1}\widetilde{w}_{s}^{-1}$ where $x_1=w_{s}^{-1}xw_{s}$ (see definition below Proposition \ref{prop:restricted-factor}).

Then we have $\Psi(\tilde{R}(u,v))\in \prod_{\omega\in C(r,r')}\tilde{V}_{[\omega]}$,
where $\tilde{V}_{[\omega]}=\tilde{U}_{[\omega]}$ if $\omega$ is detected and $\tilde{V}_{[\omega]}=\tilde{w}_s^{-1}\tilde{U}_{[w_s(\omega)]}\tilde{w}_s$ if $\omega$ is undetected. This implies that

\begin{align}\label{for:37}
  \tilde{w}_s \Psi(\tilde{R}_r(u,v))\tilde{w}_s^{-1}\in \prod_{\eta\in w_s(C(r,r'))} \tilde{V}_{[\eta]}.
\end{align}
\eqref{for:36} and \eqref{for:37} show that both $\tilde{w}_s \Psi(\tilde{R}(u,v))\tilde{w}_s^{-1}$ and $\tilde{w}_s \Psi([\tilde{u},\tilde{v}])\tilde{w}_s^{-1}$ are inside $\prod_{\eta\in w_s(C(r,r'))}\tilde{U}_{[\eta]}$. Furthermore, they project to
$R(w_s(u),w_s(v))$ and $[w_s(u),w_s(v)]$ respectively in $G$, then by Lemma \ref{lem:detected-comm}  we get:
\begin{align}\label{for:14}
 \tilde{w}_s \Psi(\tilde{R}(u,v))\tilde{w}_s^{-1}=\tilde{w}_s \Psi([\tilde{u},\tilde{v}])\tilde{w}_s^{-1}.
\end{align}
This implies the result.

\smallskip

\noindent Now we consider \textbf{case $(c)$}: $r$ is detected and $r'$ is undetected. By Lemma \ref{lem:det-gen}, we can write $r'=w_s(s')$, where $s,\,s'$ in
$\Delta_{A_0}$ are both detected and $s$ and $r$ are not proportional over $A$. Then we set $\psi_{r'}(v)=\tilde{w}_s \tilde{v_1} \tilde{w}_s^{-1}$, where $v=w_s(v_1)w_s^{-1}$, $v_1\in U_{[s']}$. Note that $w_s(r)$ is detected since $s$ and $r$ a not proportional over $A$. Then we have

\begin{align*}
  \Psi([\tilde{u},\tilde{v}])&=[\tilde{u}, \widetilde w_{s}\tilde{v_1} \widetilde w_{s}^{-1}]=\widetilde w_{s}[\widetilde w_{s}^{-1}\tilde{u}\widetilde w_{s},  \tilde{v_1} ]\widetilde w_{s}^{-1}\\
  &\in \widetilde w_{s} [\tilde{U}_{[w_s(r)]},  \tilde{U}_{[s']} ] \widetilde w_{s}^{-1} \overset{(\dagger)}{\subset} \widetilde w_s \left( \prod_{\eta\in C(w_s(r),s')}\tilde{U}_{[\eta]}\right) \widetilde w_s^{-1}.
\end{align*}
Note that ($\dagger$) follows from the fact that $C(w_s(r),s')\subset \{ir\mid_A+js'\mid_A:i\geq 1\}\bigcap \Lambda_A$ where the latter is $A$-admissible.
Finally, we may distribute the conjugation into the product to arrive at the desired equality since we also have
\begin{align*}
 w_s(C(w_s(r),s'))\subset \{ir\mid_A+js'\mid_A:i\geq 1\}\bigcap \Lambda_A.
\end{align*}
Thus, we finish this case.

\smallskip

\noindent Finally we consider \textbf{case $(d)$}: Suppose that $r$ and $r'$ are undetected. By Lemma \ref{lem:isom-gen-ind} we can write $r = w_s(r_1)$ such that $s$ and $r_1$ are detected. Write $u=w_su_1w_s^{-1}$ and $v=w_sv_1w_s^{-1}$. By using Lemma \ref{fact:8} we get
 \begin{align*}
 \tilde{w}_s^{-1}\Psi([\tilde{u},\tilde{v}])\tilde{w}_s&=\Psi([\tilde{u_1},\tilde{v_1}])\qquad\qquad\text{and}\\
 \tilde{w}_s^{-1}\Psi\big(\tilde{R}(u,v)\big)\tilde{w}_s&=\Psi\big(\tilde{R}(u_1,v_1))\big).
 \end{align*}
Since $u_1\in U_{[r_1]}$ and $r_1$ is detected,
\begin{align*}
 \Psi([\tilde{u_1},\tilde{v_1}])=\Psi\big(\tilde{R}(u_1,v_1)\big)
\end{align*}
by either case (b) or (d) depending on whether $w_s(r')$ is detected or not. This implies the result again. Then we finish the proof of this case. Since this exhausts all cases, we conclude the proof.
\end{proof}

\subsection{Proof of Theorem \ref{thm:central} when $\mathcal{G}=G$}\label{cor:simple-central}
Now we are ready to prove Theorem \ref{thm:central} when $\mathcal{G}=G$. By Remark \ref{re:3} we can assume $G$ is as described in Section \ref{sec:1}.
By using Corollary \ref{cor:4}, $\pi: \widetilde{\mathcal{G}}_{A_0}/ \mc S_{A_0}\to G$ is a perfect central extension. Note that since $\widetilde{\mathcal{G}}_{A}/ \mc S_{A}$ is a factor of the perfect group $\widetilde{\mathcal{G}}_{A_0}/ \mc S_{A_0}$, it is again perfect.

For $\pi_{G}:\widetilde{\mathcal{G}}_{A}/ \mc S_{A}\to G$ (see Section \ref{sec:15}), we note that $\pi_{G}\circ \Psi=\pi$, where $\Psi$ is as defined in Proposition \ref{prop:undetected-comm}. If $\pi_G(\sigma) = e \in G$ for some $\sigma \in \widetilde{\mathcal{G}}_{A}/ \mc S_{A}$, then surjectivity and equivariance implies that $\sigma = \Psi(\sigma')$ for some $\sigma' \in \ker(\pi)$. But since $\widetilde{\mathcal{G}}_{A_0}/ \mc S_{A_0}$ is a central extension, $\sigma'$ commutes with everything in $\widetilde{\mathcal{G}}_{A_0}/ \mc S_{A_0}$. Again, since $\Psi$ is surjective, this implies that $\sigma = \Psi(\sigma')$ is central in $\widetilde{\mathcal{G}}_{A}/ \mc S_{A}$.

\subsection{Kernel of $\pi_{G}$}
\label{sec:kernel-sec}

At the end of this section we state a result which will be used later. For any undetected $r\in\Delta^1_{A_0}$, choose detected $r_1,\,r_2\in \Delta_{A_0}$, such that $w_{r_2}(r_1)=r$. For $x,\,x_1\in
U_{[r]}\setminus\set{e}$ define
\begin{align}\label{for:43}
 \textbf{h}_r(x,x_1)=w_{r_2}\cdot \left(w_{r_1}(w_{r_2}^{-1}xw_{r_2})w_{r_1}(w_{r_2}^{-1}x_1w_{r_2})\right)\cdot w_{r_2}^{-1}.
\end{align}
Let $\mbf{H}_{r}$ be the subgroup generated by $\mbf{h}_r(x,x_1)$. Note the similarities with the elements $\widetilde{h}_\mu(x,x_1)$ appearing in Theorem \ref{le:1} (they are exactly the images of $\widetilde{h}_r(x,x_1)$ via $\Psi$). By using Theorem \ref{le:1} and Proposition \ref{prop:restricted-factor} we get the following result immediately.

\begin{corollary}\label{cor:3}
Let
$\widetilde{H}$ be he subgroup generated by $\widetilde{H}_{r},
r\in\Delta^1_{A_0}$ for detected $r$, and $\widetilde{\mbf{H}}_{r}$, for undetected $r\in \Delta^1_{A_0}$.
Then each $\widetilde{H}_r$ or $\widetilde{\mbf{H}}_{r}$  is normal inside $\tilde{H}$. Furthermore, if $S$ is a connected Lie group with $\text{Lie}(S)=\text{Lie}(G)$, then
$\ker\big(\widetilde{\mathcal{G}}_{A}/ \mc S_{A}\to S\big)\subset \tilde{H}$.

\end{corollary}

\section{Algebraic Preparatory step III: twisted symmetric space examples}
\label{sec:alg3}

The aim of this section is to prove  Theorem \ref{thm:central} when $\mathcal{G}=G_\rho$. As with the semisimple case, the principle difficulty will come from the negatively proportional weights and the zero weights of the representation, if they are present. Furthermore, the existence of non-trivial radical in the universal Lie central extension will create extra complications. In Section \ref{sec:alg2} the key technical ingredient in the proof is the action of the Weyl group, which acts transitively on all roots of a fixed length for a split Cartan subgroup inside each simple factor. This fails for most weight systems. So we introduce new techniques using different conjugations to act on the weight spaces.

Recall notations in \seref{se:3} of Section \ref{sec:14}. For any $\mu \in \Lambda_A^*$, set $\tilde{U}_{G,[\mu]}=\{\tilde{u}: u\in U_{G,[\mu]}\}$ for $\mu \in \Delta_{A}$ and $\tilde{U}_{\mathfrak{e},[\mu]}=\{\tilde{u}: u\in U_{\mathfrak{e},[\mu]}\}$ for $\mu \in \Phi_{A,\rho}^*$. Also set $\tilde{U}_{\mathfrak{e},[\mu],i}=\{\tilde{u}: u\in U_{\mathfrak{e},[\mu],i}\}$, $i\in I$. Note that for every $\mu \in \Lambda_A^*$, we then get from Lemma \ref{lem:detected-comm}:

\begin{equation}
\label{eq:weightroot-decomp}
 \tilde{U}_{[r]} = \tilde{U}_{G,[r]} \cdot \prod_{i \in I} \tilde U_{\mf e,[r],i}
\end{equation}
Similar to notations below Theorem \ref{thm:deodhar-centr-ext}, we set $\tilde{U}_{[r]}=\{\tilde{u}: u\in U_{G,[r]}\}$ for detected $r\in \Delta_{A_0}$.
Combining this with \eqref{eq:weightroot-decomp}, we get that if $\mu \in \Delta_A$:

\[ \tilde U_{[\mu]} = \prod_{\substack{s\in \Delta_{A_0} \\ s|_A = \mu}} \tilde{U}_{[s]} \cdot \prod_{i \in I} \tilde U_{\mf e,[\mu],i} \]

If we write $r\in \Delta_{A_0}$ then $\mathfrak{g}_r$ (see Section \ref{sec:3}) comes from the decomposition of $\mathfrak{g}$ with respect to $A_0$, if we write $\mu \in \Delta_{A}$ then $\mathfrak{g}_\mu$ comes from the decomposition of $\mathfrak{g}$ with respect to $A$.

\subsection{Basic properties of weight system}\label{sec:new}
In this part we show the technique to handle negatively proportional weights and roots, which is the first step toward the proof of Theorem \ref{thm:central}. We begin by proving technical results which will be useful for subsequent discussions, the first of which constructs a basis of each detected weight space using nonproportional ones:

\begin{lemma}
\label{lem:gen-zero-weights}
Assume $A$ satisfies the genuinely higher rank condition. For any $\mu\in \Phi_{\rho,A}$, if $U_{\mathfrak{e},[\mu],i}\not=\set{0}$, there exist fixed $\mu_j\in \Phi_{\rho,A}$ and $r_j\in \Delta_{A_0}$, together with distinguished vectors $X_j\in \mathfrak{g}_{r_j}$ and $y_j\in U_{\mathfrak{e},[\mu_j],i} \setminus \set{0}$, $j\in J$ such that $\mu_j$ and $r_j|_A$ are not proportional and the set $\{d\rho(X_j)y_j : j \in J\}$ spans $U_{\mathfrak{e},[\mu],i}$.
\end{lemma}
\begin{proof}
Choose $\nu\in \Phi_{\rho,A}$ that is not proportional to $\mu$ and $U_{\mathfrak{e},[\nu],i}\not=\set{0}$. Fix $y\in U_{\mathfrak{e},\nu,i}\setminus \set{0}$. For any $r\in \Delta_{A_0}$, let
\begin{align*}
L_r=\{X_{r,i}\}  \text{ be a basis of }\mathfrak{g}_r\quad \text{and} \quad L_0=\{Y_i\} \text{ is a basis of } \mf g_0 = \text{Lie}(N).
\end{align*}
Note that $\mathfrak{g}_r$ comes from the decomposition of $\text{Lie}(G_\rho)$ with respect to $A_0$.

For any $\mu \in \Phi_{\rho,A}$, Let $P_\mu^1$ be the union of the $L_s$ such that $s|_A$ is proportional to $\mu$ and $P_\mu^2$ be the union of $L_s$ where $s|_A$ are not proportional to $\mu$. Since $\{P_r^2, P_r^1, L_0\}$ is a ordered basis of $\mathfrak{g}$, by Poincar\'{e}-Birkhoff-Witt theorem and irreducibility of $\RR^{N_i}$ for each $i\in I$, monomials $X_{r_1,1}^{k_1}X_{r_2,2}^{k_2}\cdots X_{r_m,m}^{k_m}Y_{1}^{l_1}Y_2^{l_2}\cdots Y_n^{l_n}$ withe respect to the order form a basis of universal enveloping algebra of $\mathfrak{g}$. Then there exist monomials $X_{r_{j_1},1}^{k_1}X_{r_{j_2},2}^{k_2}\cdots X_{r_{j_m},m}^{k_m}Y_{1}^{l_1}Y_2^{l_2}\cdots Y_n^{l_n}$, $j\in J$ such that
\begin{align*}
 d\rho(X_{r_{j_1},1})^{k_1} d\rho(X_{r_{j_2},2})^{k_2}\cdots  d\rho(X_{r_{j_m},m})^{k_m} d\rho(Y_{1})^{l_1} d\rho(Y_2)^{l_2}\cdots  d\rho(Y_n)^{l_n}y,\quad j\in J
\end{align*}
span a basis of $U_{\mathfrak{e},[\mu],i}$.

By assumption $r_{j_1}\in P_\mu^2$ for each $j\in J$. For any $j\in J$, set
\begin{align*}
y_j=d\rho(X_{r_{j_1},1})^{k_1-1} d\rho(X_{r_{j_2},2})^{k_2}\cdots  d\rho(X_{r_{j_m},m})^{k_m} d\rho(Y_{1})^{l_1} d\rho(Y_2)^{l_2}\cdots  d\rho(Y_n)^{l_n}y
\end{align*}
and $X_j=X_{r_{j_1},1}$.

Then $X_j$ and $y_j$, $j\in J$ are what we need.
\end{proof}

\begin{lemma}\label{fact:10}
For any $r \in \Delta_{A_0}$ and $g\in U_{G, [r]}$, there exists $S\subset \Delta_{A_0}$ such that $g\in \prod_{\theta\in S} U_{G, [\theta]}$. Furthermore, $S$ is inside an $A$-admissible set and any $\theta\in S$ is not proportional to $r$ over $A$.
\end{lemma}
\begin{proof}
First, we prove the case $g=\exp(u)$ with $u\in \mathfrak{g}_r$. For any detected $r\in \Delta_{A_0}$ \eqref{for:16} of  Lemma \ref{lem:det-gen} shows that we can write $r=w_{s_1}(s_2)$ where the $s_i\in \Delta_{A_0}$ are not proportional to $r$ over $A$. Suppose $w_{s_1}=yxy^{-1}$ is a detailed expression of $w_{s_1}$, where $y\in U_{G,[s_1]}$  and $x \in U_{G,[-s_1]}$. Note that  $w_{s_1}(s_2)=ns_1+s_2$, $0\neq n\in\ZZ$. Then for any $0\neq u\in \mathfrak{g}_{r}$:
 \begin{enumerate}
   \item if $n>0$ , there exists a suitable $v\in \mathfrak{g}_{(n-1)s_1+s_2}$ such that $u=[\log y, v]$. Then $\exp(u)\in \prod_{\theta\in S} U_{G, [\theta]}$ where $S=\{r\neq is_1+js_2, ks_1: j,\,k\geq 1, i\geq0\}\bigcap \Delta_{A_0}$;

   \smallskip
   \item if $n<0$, then exists a suitable $v\in \mathfrak{g}_{(n+1)s_1+s_2}$ such that $u=[\log x, v]$. Then $\exp(u)\in \prod_{\theta\in S} U_{G, [\theta]}$ where $S=\{r\neq is_1+js_2, ks_1: j\geq 1, \,k\leq-1,\,i\leq0\}\bigcap \Delta_{A_0}$.
 \end{enumerate}
If $is_1+js_2 = mr$, then $m$ is $1$ or $2$. So when $2r\notin \Delta_{A_0}$, from above discussion we get the result. If $2r\in \Delta_{A_0}$, by noting that  $2r=w_{s_1}(2s_2)$,  then by repeating above process to $2r$ in both cases, we get:
\begin{enumerate}
   \item if $n>0$ then $\exp(u)\in \prod_{\theta\in S} U_{G, [\theta]}$ where $S=\{mr\neq is_1+js_2, ks_1: j,\,k\geq 1, i\geq0,\,m\geq1\}\bigcap \Delta_{A_0}$;

   \smallskip
   \item if $n<0$, then $\exp(u)\in \prod_{\theta\in S} U_{G, [\theta]}$ where $S=\{mr\neq is_1+js_2, ks_1: j\geq 1, \,k\leq-1,\,i\leq0,\,m\geq1\}\bigcap \Delta_{A_0}$.
 \end{enumerate}
 Then we get the result if $g=\exp(u)$ with $u\in \mathfrak{g}_r$. Generally, for any $g\in U_{G, [r]}$ it can be written as a product of $g_1\cdots g_\ell$ with $\log g_i\in \mathfrak{g}_{\lambda_i r}$, $\lambda_i=\frac{1}{2},\,1,\,2$ for each $1\leq i\leq\ell$. Also note that $\lambda_i r=w_{s_1}(\lambda_i s_2)$ for each $1\leq i\leq\ell$. Then by the earlier argument, we get the result.

\end{proof}

Next, we state a simple fact which will be useful in subsequent discussions.
\begin{fact}\label{fact:4}
Suppose $s\in \Phi_{A,\rho}$, $r\in \Delta_{A_0}$ with $r\mid_A\neq0$, $x\in U_{\mathfrak{e},s,i}$, $i\in I$, and $X_r\in \mathfrak{g}_r$.  Also suppose $x_1=d\rho(X_r)x=[X_r,x]\neq 0$ (see Section \ref{sec:16}). Then
\begin{align*}
 \rho(\exp(X_{r})x)=x+\sum_{1\leq j\leq m} x_j
\end{align*}
where $0\neq x_j\in U_{\mathfrak{e}, [s+jr|_A],i}$ $2\leq j\leq m$.

Therefore:
\begin{align*}
 x_1=[\exp(X_{r}),x]-\sum_{2\leq j\leq m} x_j.
\end{align*}
The above computation shows:

\begin{enumerate}
  \item \label{for:48} $x_1$ is generated by elements in the subgroups $U_{G,[r]}$ and $U_{\mathfrak{e},[s+jr\mid_A],i}$, $j=0,\,2,\cdots, m$;

      \smallskip
  \item \label{for:54}if $s\neq0$ and $s+j_0r=0$ for some $1\leq j_0\leq m$, then
\begin{align*}
 x_{j_0}=[\exp(X_{r}),x]-\sum_{1\leq j\neq j_0\leq m} x_j.
\end{align*}
Then $x_{j_0}$ is generated by elements in the subgroups $U_{G,[r]}$ and $U_{\mathfrak{e},[s+jr\mid_A],i}$, $j\neq j_0$.
  \end{enumerate}
  \end{fact}

As a consequence of the above lemmas, we have
\begin{lemma}\label{le:6}
For any $\mu \in \Phi_{A,\rho}$ and any $x\in U_{\mathfrak{e},[\mu]}$, we can write $x = x_1 \cdot \dots \cdot x_n$
such that $x_i \in \left( \displaystyle\prod_{\theta\in S_i} U_{[\theta_i]}\right)\bigcap U_{\mathfrak{e},[\mu]}$  where $S_i\subset \Lambda_A$. Furthermore, each $S_i$ contains no element proportional to $\mu$, and is contained in an $A$-admissible subset containing $\mu$.
\end{lemma}

\begin{proof}
We assume that $x\neq e$ otherwise the conclusion is obvious. We can write $x=y_1\cdots y_m$, where $y_i\in U_{\mathfrak{e},[\mu],i}$, $i\in I$. By Lemma \ref{lem:gen-zero-weights} we can write
\begin{align*}
 y_i=d\rho(X_1)(y_{i,1}) + \dots + d\rho(X_l)(y_{i,l})
\end{align*}
where $X_j\in \mathfrak{g}_{r_j}$, $r_j\in \Delta_{A_0}$ and $y_{i,j}\in U_{\mathfrak{e},[\mu_j],i}$ such that $r_j|_A$ and $\mu_j$ are not proportional to $\mu$ for any $ 1\leq j\leq l$.

By using \eqref{for:48} of Fact \ref{fact:4} we
see that $d\rho(X_j)(y_{i,j})\in \prod_{\theta_{i,j}\in S_{i,j}}U_{[\theta_{i,j}]}$, where $S_{i,j}=\{\mu_j+kr_j,\,\lambda r_j: \lambda>0,\,k=0,2,\cdots\}\bigcap\Lambda_A$. It is clear that any $\theta \in S_{i,j}$ is not proportional to $\mu$ and that
$S_{i,j}$ is contained in an $A$-admissible set.
\end{proof}

\begin{lemma}\label{le:2}
Suppose $r\in\Delta_A\cup \Phi^*_{A,\rho}$ and $\widetilde{p}$ is inside $\left\langle \tilde{U}_{[r]},\tilde{U}_{[-r]}\right\rangle$. Then for any $\mu\in \Delta_A\cup \Phi^*_{A,\rho}$ with $\mu$ not proportional to $r$ and $\tilde{y}\in \tilde{U}_{[\mu]}$, we have: $[\widetilde{p},\widetilde{y}]\in \prod_{\theta\in S}\tilde{U}_{[\theta]}$,
where $S=\{kr+j\mu:k\in\R,j\geq1\}\bigcap\Lambda_A$.
\end{lemma}

\begin{proof}
For any $\tilde{x}$ in $\tilde{U}_{[r]}$ or in $\tilde{U}_{[-r]}$, it is clear that

\[ \text{(a)}\;[\tilde{y},\tilde{x}]\in \prod_{\theta\in S}\tilde{U}_{[\theta]}.\]

Also, for any $\tilde{u}\in \prod_{\theta\in S}\tilde{U}_{[\theta]}$ we also have

\[ \text{(b)}\;[\tilde{u},\tilde{x}]\in\prod_{\theta\in S}\tilde{U}_{[\theta]},\]
We now proceed with the general case as stated in the Lemma. Write $\widetilde{p}=\widetilde{p}_1\cdots \widetilde{p}_m$, where $\widetilde{p}_i\in \tilde{U}_{[(-1)^ir]}$, $1\leq i\leq m$. By applying above discussion repeatedly we have
\begin{align*}
  \widetilde{p}_1\cdots \widetilde{p}_m\tilde{y}=\tilde{z}_{m,1}\widetilde{p}_1\cdots \widetilde{p}_{m-1}\tilde{y}\widetilde{p}_m,
\end{align*}
where $\tilde{z}_{m,1}\in \prod_{\theta\in S}\tilde{U}_{[\theta]}$. At each step, we apply the equality $(a)$ to get $z_{m,k}'\tilde y \tilde p_k = \tilde p_k \tilde y$, then move $\tilde z_{m_k}'$ to the left using $(b)$.

Inductively we can show that
\begin{align*}
  \widetilde{p}_1\cdots \widetilde{p}_m\tilde{y}=\tilde{z}_{m,m}\tilde{y}\widetilde{p}_1\cdots \widetilde{p}_{m-1}\widetilde{p}_m,
\end{align*}
where $\tilde{z}_{m,m}\in \prod_{\theta\in S}\tilde{U}_{[\theta]}$. Then we get the result.
\end{proof}

Next, we are in a position to obtain some information on negatively propositional weights or roots.

\begin{lemma}

\label{lem:1dim-centrality}
Suppose $r\in\Delta_A\bigcup\Phi^*_{A,\rho}$ and $\pi_{G_\rho}(\widetilde{p})=e$ where $\widetilde{p}$ is inside $\left\langle \tilde{U}_{[r]},\tilde{U}_{[-r]} \right\rangle$.
Then $\widetilde{p} \in Z(\widetilde{\mathcal{G}}_A/ \mc S_A)$.
\end{lemma}

\begin{proof}
To prove the result it suffices to show that $\widetilde{p}$ in the centralizer of $\tilde{U}_{[s]}$, for any $s\in \Delta_A\cup \Phi^*_{A,\rho}$.

First, we consider the case $s$ is nonproportional to $r$. For any $\tilde{y}\in \tilde{U}_{[s]}$, by using Lemma \ref{le:2} we see that $[\tilde{y},\tilde{p}]\in \Pi_{\theta\in S}\tilde{U}_{[\theta]}$, where $S=\{ir+js:j\geq1\}\bigcap \Lambda_A$. Since $S$ is $A$-admissible and $[\tilde{y},\tilde{p}]$ projects to $e\in \mathcal{G}$, Lemma \ref{lem:detected-comm} implies that $[\tilde{y},\tilde{p}]=e$. This shows that $\widetilde{p}$ is in the centralizer of $\tilde{U}_{[s]}$ with $s$ nonproportional to $r$.

Now we consider the case $s$ is proportional to $r$. Note that $\tilde{U}_{[s]}$ is generated by $\tilde{U}_{G,[s]}$ and $\tilde{U}_{\mathfrak{e},[s]}$.
Also note that $\tilde{U}_{G,[s]}$, $s\in \Delta_A$ is generated by $\tilde{U}_{G,[\mu]}$, $\mu\in \Delta_{A_0}$ with $\mu|_A$ positively proportional to $s|_A$. 

For any $\tilde{y}\in \tilde{U}_{G,[\mu]}$, $\mu\in \Delta_{A_0}$ with $\mu|_A$ positively proportional to $s|_A$,  \eqref{for:16} of  Lemma \ref{lem:det-gen} shows that
we can write $\tilde{y}=\widetilde w_{s_1}\tilde{y_1}\widetilde w_{s_1}^{-1}$, with $s_1,\,s_2\in \Delta_{A_0}$ $\tilde{y_1}\in \tilde{U}_{G,[s_2]}$, $w_{s_1}(s_2) =s$ and $s_1\mid_A$, $s_2\mid_A$ and $s$ are pairwise nonproportional. Thus, previous argument shows that $\tilde{y}$ is a product of elements which commute with $\widetilde{p}$. Hence we get $[\tilde{y},\tilde{p}]=e$. This shows that $\widetilde{p}$ is in the centralizer of $\tilde{U}_{G,[s]}$ with $s$ proportional to $r$.

For any $\tilde{y}\in \tilde{U}_{\mathfrak{e},[s]}$  by Lemma \ref{le:6} and previous argument we see that $\tilde{y}$ is a product of elements which commute with $\widetilde{p}$. This also shows that $\widetilde{p}$ is in the centralizer of $\tilde{U}_{\mathfrak{e},[s]}$ with $s$ proportional to $r$.
Then we finish the proof.
\end{proof}

As an immediate corollary of the above proposition we have

\begin{corollary}
\label{cor:weight-comm}
If $\mu_1,\mu_2 \in \Phi^*_{\rho,A}$ and $\widetilde{x_i} \in\tilde{U}_{\mathfrak{e},[\mu_i]}$, $i=1,\,2$, then if $\mu_1$ is not negatively proportional to $\mu_2$, $[\widetilde{x}_1,\widetilde{x}_2]=e$. If $\mu_1$ is negatively proportional to $\mu_2$, then $[\widetilde{x}_1,\widetilde{x}_2]\in Z(\widetilde{\mathcal{G}}_A/ \mc S_A)$ and projects to identity in $G_\rho$.
\end{corollary}

Corollary \ref{cor:weight-comm} will allow us to freely rearrange elements of the subgroup $\Pi_{s\in \Phi^*_{\rho,A}}\tilde{U}_{\mathfrak{e},[s]}$ (at the cost of elements in the center).

\subsection{Recovering zero weight} \label{fact:2}
Next we will show how to recover zero weight space in $G_\rho$, if it exists.

\begin{definition}
\label{def:0-weight-gen}
For any $r\in\Phi^*_{A,\rho}$, let $\tilde{U}_{r,[0],i}$, $i\in I$ be the group generated by elements in $\widetilde{\mathcal{G}}_A/ \mc S_A$ that project to $U_{\mathfrak{e},[0],i}$ and can be written in the form:

\[ [\widetilde{g},\widetilde{x}]\widetilde{x}_1\cdots\widetilde{x}_{m},\]

where $\tilde{g}\in \tilde{U}_{G,[\gamma]}$, $\gamma\in \Delta_{A}$ with $\gamma$ negatively proportional to $r$, $\tilde{x}\in \tilde{U}_{\mathfrak{e},[r],i}$ and $\tilde{x_j}\in \tilde{U}_{\mathfrak{e},[s],i}$ with $s\in \Phi^*_{A,\rho}$ proportional to $r$, $1\leq j\leq m$.
\end{definition}

Lemma \ref{fact:4} and \eqref{for:54} if Fact \ref{fact:4} show that the $\pi_{G_\rho}(\tilde U_{r,[0],i})$ generate $U_{\mf e,[0],i}$ as they vary over $r$.
We now make a detailed study of how the $\tilde{U}_{r,[0],i}$ sit inside $\widetilde{\mc G}_A / \mc S_A$.

\begin{lemma}\label{le:3}
Suppose $\mu_i \in\Delta_A$ and $r_i\in\Phi^*_{A,\rho}$, $i=1,\,2$, with $\mu_i$ negatively proportional to $r_i$ and $\mu_1$ nonproportional to $\mu_2$.
Then for any $g_i\in \tilde{U}_{G,[\mu_i]}$ and $x_i\in\tilde{U}_{\mathfrak{e},[r_i]}$, $i=1,\,2$:
\begin{enumerate}
  \item \label{for:21}$\left[\widetilde g_1 \widetilde x_1 {\widetilde g_1}^{-1},\widetilde x_2\right] = e$ and

  \medskip
  \item \label{for:22}$\left[\widetilde{g_1}\widetilde{x_1}\widetilde{g_1}^{-1},\widetilde{g_2}\widetilde{x_2}\widetilde{g_2}^{-1}\right]\in Z(\widetilde{\mathcal{G}}_A/ \mc S_A)$ and projects to identity in $G_\rho$.
\end{enumerate}
\end{lemma}
\begin{proof}
From Lemma \ref{le:2} and the fact that $[g_1x_1g_1^{-1},x_2]=e$, we see that $[\widetilde{g_1}\widetilde{x_1}\widetilde{g_1}^{-1},\widetilde{x_2}]=e$. Then we proved \eqref{for:21}.

Using Lemma \ref{le:2} and the fact that $[g_1x_1g_1^{-1},g_2]\in \RR^N$ we see that
\begin{align*}
  \widetilde{g_1}\widetilde{x_1}\widetilde{g_1}^{-1}\widetilde{g_2}=\tilde{u}\widetilde{g_2}\widetilde{g_1}\widetilde{x_1}\widetilde{g_1}^{-1},
\end{align*}
where $\tilde{u}\in \{\tilde{U}_{\mathfrak{e},[\theta]}:\theta\in S\}$ with $S=\{t\mu_1+j\mu_2:t\in\R,j\geq1\}$ projects to $\left[g_1{x_1}{g_1}^{-1},g_2\right]$. Note that $S$ is $A$-admissible. In the following computations, we use parentheses to emphasize where in the expression we apply commutation relations. Then we have
\begin{align*}
\left(\widetilde{g}_1\widetilde{x}_1\widetilde{g}_1^{-1}\right) \left(\widetilde g_2\right) \widetilde x_2 {\widetilde g_2}^{-1} & \overset{(a)}{=}
  \tilde{u}\widetilde g_2 \left(\widetilde g_1 \widetilde x_1 {\widetilde g_1} ^{-1}\right)\left(\widetilde x_2 \right){\widetilde g_2} ^{-1} \\
  &\overset{(b)}{=}\widetilde{u}\widetilde g_2 \widetilde x_2 \left(\widetilde g_1 \widetilde x_1 {\widetilde g_1} ^{-1}\right)\left({\widetilde g_2}^{-1}\right)\\
  &\overset{(c)}{=}\tilde{u}\widetilde g_2 \left(\widetilde x_2 \right)\left(\widetilde u_1 \right){\widetilde g_2}^{-1}\widetilde g_1 \widetilde x_1 \widetilde g_1 ^{-1}\\
  &\overset{(d)}{=}[\widetilde x_2 ,\widetilde u_1 ]\widetilde{u}\left(\widetilde g_2 \right)\left(\widetilde{u}_1\right)\widetilde x_2 {\widetilde g_2}^{-1}\widetilde g_1
  \widetilde x_1 {\widetilde g_1}^{-1}\\
  &\overset{(e)}{=}[\widetilde{x}_2,\widetilde{u}_1]\left(\widetilde{u}[\widetilde g_2 ,\widetilde u_1 ]\widetilde u_1 \right)
  \widetilde g_2 \widetilde x_2 {\widetilde g_2}^{-1}\widetilde g_1 \widetilde x_1{ \widetilde g_1}^{-1}\\
  &\overset{(f)}{=}[\widetilde x_2 ,\widetilde u_1 ]\widetilde g_2 \widetilde x_2 {\widetilde g_2}^{-1}\widetilde g_1 \widetilde x_1 {\widetilde g_1}^{-1}
\end{align*}
where $\widetilde u_1 =[\widetilde g_1 \widetilde x_1 {\widetilde g_1}^{-1},{\widetilde g_2}^{-1}]\in \{\tilde{U}_{\mathfrak{e},[\theta]}:\theta\in S\}$. $(a)$ follows from the definition of $\widetilde{u}$. (b) follows from part (1) of the Lemma. We get $(c)$ by using Lemma \ref{le:2} and the fact that $[g_1x_1g_1^{-1},g_2^{-1}]\in \RR^N$. In $(d)$ we used $[\widetilde{x_2},\widetilde{u_1}]\in Z(\widetilde{\mathcal{G}}_A/ \mc S_A)$ (see Corollary \ref{cor:weight-comm}). In $(e)$ note that $[\widetilde{g_2},\widetilde{u_1}]\in \{\tilde{U}_{\mathfrak{e},[\theta]}:\theta\in S\}$.
In $(f)$ we see that $\tilde{u}[\widetilde{g_2},\widetilde{u_1}]\widetilde{u_1}\in \{\tilde{U}_{\mathfrak{e},[\theta]}:\theta\in S\}$ and projects to identity in $G_\rho$, so by
Lemma \ref{lem:detected-comm} $\tilde{u}[\widetilde{g_2},\widetilde{u_1}]\widetilde{u_1}=e$. Hence we get $(f)$ and thus finish the proof.
\end{proof}
\begin{corollary}\label{prop:zero-weight-group-struct}
\begin{enumerate}
  \item \label{for:23} For any $\widetilde{y}\in \tilde{U}_{\mu,[0],i}$, $i\in I$, $\mu\in\Phi^*_{A,\rho}$ any $\tilde{x}\in \tilde{U}_{\mathfrak{e},[\nu]}$, $r\in \Phi^*_{A,\rho}$, $\tilde{y}\tilde{x}=\tilde{x}\tilde{y}$;

      \medskip

  \item \label{for:41} For any $\widetilde{y}\in \widetilde{U}_{s,[0],i}$, $i\in I$, $s\in\Phi^*_{A,\rho}$ and $\tilde{g}\in \tilde{U}_{G,[r]}$ with detected $r\in\Delta_{A_0}$, then $[\tilde{y},\tilde{g}]\in \tilde{U}_{\mathfrak{e},[r]}$.

      \medskip

  \item \label{for:55} For any $\widetilde{y}\in \widetilde{U}_{s,[0],i}$, $i\in I$, $s\in\Phi^*_{A,\rho}$ and $\tilde{g}\in \tilde{U}_{G,[r]}$, $r\in\Delta_{A}$, then $[\tilde{y},\tilde{g}]\in \tilde{U}_{\mathfrak{e},[r]}$.

  \medskip
  \item \label{for:24} Suppose $\tilde{y}=\widetilde y_1 \widetilde y_2 \cdots \widetilde y_n$, where $\widetilde y_i \in \tilde{U}_{s_i,[0],\ell_i}$ is a generator for $1\leq i\leq n$, $\ell_i\in I$ and $\tilde{y}$ projects to
$e$ in $G_\rho$. Then $\widetilde{y}\in Z(\widetilde{\mathcal{G}}_A/ \mc S_A)$. In particular, $[\widetilde y_1, \widetilde y_2] \in Z(\widetilde{\mc G}_A / \mc S_A)$.
\end{enumerate}

\end{corollary}
\begin{proof}
Proof of \eqref{for:23}: If $\mu$ is nonproportional to $\nu$, Corollary \ref{cor:weight-comm} and \eqref{for:21} of Lemma \ref{le:3} imply the result immediately.  If $\mu$ is proportional to $\nu$, then by Lemma \ref{le:6} and previous argument we see that $\tilde{x}$ is a product of elements which commute with $\widetilde{y}$. Hence we finish the proof.

\smallskip
Proof of \eqref{for:41}:  If $s$ is nonproportional to $r$, then $\widetilde{y}\widetilde{g}=\tilde{u}\widetilde{g}\widetilde{y}$, where $\widetilde{u}\in \Pi_{\theta\in S}\tilde{U}_{[\theta]}$, $S=\{is+jr:j\geq1\}$ is $A$-admissible (see Lemma \ref{le:2}). Since $[\widetilde{y},\widetilde{g}]$ projects to an element in $U_{\mathfrak{e},[r]}$, by Lemma \ref{lem:detected-comm} we see that $[\widetilde{y},\widetilde{g}]\in \tilde{U}_{\mathfrak{e},[r]}$.

If $s$ is proportional to $r$, by using Lemma \ref{fact:10} $\tilde{g}\in\prod_{\theta\in S}\tilde{U}_{G,[\theta]}$, where $S$ is inside an $A$-admissible and any $\theta\in S$ is not proportional to $r$ over $A$, then by previous argument $[\tilde{y},\tilde{g}]\in \prod_{\theta\in S}\tilde{U}_{\mathfrak{e},[\theta]}$.
Since $S$ is inside an $A$-admissible set and $[\tilde{y},\tilde{g}]$ projects to an element inside $U_{\mathfrak{e},[r]}$, then we still get $[\widetilde{y},\widetilde{g}]\in \tilde{U}_{\mathfrak{e},[r]}$.

\smallskip
Proof of \eqref{for:55}: We can write $\tilde{g}=\tilde{g_1}\cdots \tilde{g_m}$, where $\tilde{g_i}\in \tilde{U}_{G,[r_i]}$, $r_i\in \Delta_{A_0}$ such that $r_i\mid_A=\lambda_i r$, $\lambda_i>0$. Then by \eqref{for:41}, $[\widetilde{y},\widetilde{g_i}]\in \tilde{U}_{\mathfrak{e},[r]}$ for each $i$. Then we get the conclusion.

\smallskip
Proof of \eqref{for:24}: From \eqref{for:23} we see that for any $r\in\Phi^*_{A,\rho}$, any $\widetilde{x}\in \widetilde{U}_{\mathfrak{e},[r]}$, $\widetilde{x}\widetilde{y}=\widetilde{y}\widetilde{x}$.
Next, we want to show for any $r\in\Delta_A$, any $\widetilde{g}\in \widetilde{U}_{G,[r]}$ and $\widetilde{g}\widetilde{y}=\widetilde{y}\widetilde{g}$.  From \eqref{for:41} we see that
\begin{align*}
  \widetilde{y_1}\widetilde{y_2}\cdots \widetilde{y_n}\widetilde{g}&=\widetilde{y_1}\widetilde{y_2}\cdots\widetilde{y_{n-1}} [\widetilde{y_n},\widetilde{g}]\widetilde{y_n}=[\widetilde{y_n},\widetilde{g}]\widetilde{y_1}\widetilde{y_2}\cdots\widetilde{y_{n-1}} \widetilde{g}\widetilde{y_n}.
\end{align*}
We get the last equality because $[\widetilde{y_n},\widetilde{g}]\in \tilde{U}_{\mathfrak{e},[r]}$ which commutes with each $\widetilde{y}_i$, $1\leq i\leq n-1$ by \eqref{for:23}. Then inductively we see that
\begin{align*}
  \widetilde{x_1}\widetilde{x_2}\cdots \widetilde{x_n}\widetilde{g}&=\widetilde{u}\widetilde{g}\widetilde{x}_1\widetilde x_2 \cdots \widetilde x_n
\end{align*}
where $\widetilde{u}\in \tilde{U}_{\mathfrak{e},[r]}$, Since $\widetilde{u}$ projects to $e$ in $G_\rho$ then Lemma \ref{lem:detected-comm} implies that $\widetilde{u}=e$. This implies that $\widetilde{g}\widetilde{x}=\widetilde{x}\widetilde{g}$. This concludes the proof.
\end{proof}

\begin{proposition}
\label{prop:semidirect-path-struct}
Any $\widetilde{p} \in \widetilde{\mathcal{G}}_A/ \mc S_A$ can be written as $\widetilde{g} \widetilde{x}\widetilde{y}\widetilde{z}$ satisfying:
\begin{enumerate}
  \item $\widetilde{g}\in \Pi_{r\in \Delta_A}\tilde{U}_{G,[r]}$;

  \smallskip
  \item $\widetilde{x} = \widetilde x_{r_1}\widetilde x_{r_2}\dots\widetilde x_{r_n}$, where $\widetilde x_{r_i}\in \tilde{U}_{\mathfrak{e},[r_i]}$, $r_i\in \Phi^*_{\rho,A}$, and the product if given in a fixed, perscribed order.

      \smallskip

  \item $\widetilde{y}$ is a product of elements in $\tilde{U}_{r,[0],i}$, $r\in \Phi^*_{A,\rho}$, $i\in I$;

  \smallskip
  \item $\widetilde{z} \in Z(\widetilde{\mathcal{G}}_A/ \mc S_A)$ with $\widetilde{z}$ projected to $e$ in $G_\rho$.
\end{enumerate}
\end{proposition}
\begin{proof}
We first wish to push elements of each $\tilde{U}_{G,[r]}$ to the left, and do so by computing commutators.  For any $\widetilde{g}\in \tilde{U}_{G,[r]}$, $r\in \Delta_A$ and $\widetilde{x}\in \tilde{U}_{\mathfrak{e},[s],i}$, $s\in \Phi^*_{\rho,A}$, $i\in I$, if $s$ and $r$ not negatively nonproportional, then $[\widetilde{g},\widetilde{x}]\in \Pi_{\gamma=is+jr, i,j>0}\widetilde{U}_{\mathfrak{e},[\gamma]}$.

If $s$ and $r$ are negatively proportional, we divide into two subcases. If $s + ir \not= 0$ for all $i \ge 1$, we see that
$[g,x]=x_1\in \langle U_{\mf e,[s]},U_{\mf e,[-s]}\rangle$. Hence $[\widetilde{g},\widetilde{x}]=\widetilde{x}_1\widetilde{z}$ where $\widetilde{z} \in Z(\widetilde{\mathcal{G}}_A/ \mc S_A) \cap \ker \pi_{G_\rho}$ and $\widetilde{x}_1 \in \langle \tilde{U}_{\mf e,[s]},\tilde{U}_{\mf e,[-s]}\rangle$ by Lemma \ref{lem:1dim-centrality}.

In this second case, there $s + ir = 0$ for some $i$. Then $[g,x]+\sum_{1\leq j\neq i\leq m} x_j=x_{i}$, where $x_j\in U_{\mathfrak{e},[jr+s]}$. So we produce an element of the zero weight space using $[\widetilde{g},\widetilde{x}]\widetilde{x}_1\cdots \widetilde{x}_{i-1}\widetilde{x}_{i+1}\cdots \widetilde{x}_m\in \tilde{U}_{s,[0],i}$ (see Definition \ref{def:0-weight-gen}). Thus, $[\widetilde{g},\widetilde{x}]$ is a product of elements in $\tilde{U}_{s,[0],i}$, $r \in \Phi^*_{\rho,A}$ and $\tilde{U}_{\mathfrak{e},[\mu]}$, $\mu\in \Phi^*_{\rho,A}$.

Thus we conclude that
$\widetilde{p}$ can be written as $\widetilde{g} \widetilde{u}\widetilde{z}_0$, where $\widetilde{g}\in \prod_{r\in \Delta_A}\tilde{U}_{G,[r]}$, $\widetilde{u}$ is inside the subgroup generated by
$\tilde{U}_{\mathfrak{e},[s]}$, and $\tilde{U}_{r,[0],i}$, $r,s\in \Phi^*_{A,\rho}$, $i\in I$, and $\widetilde{z}_0 \in Z(\widetilde{\mathcal{G}}_A/ \mc S_A) \cap \ker \pi_{G_\rho}$.
Furthermore, \eqref{for:23} of Corollary \ref{prop:zero-weight-group-struct} show that $\widetilde{u}=\widetilde{x}\widetilde{y}\widetilde{z_1}$, where $\widetilde{x}$ and $\widetilde{y}$ are as described in the statement of the Proposition and $\widetilde{z_1} \in Z(\widetilde{\mathcal{G}}_A/ \mc S_A) \cap \ker \pi_{G_\rho}$. We thus arrive at the canonical form described.
\end{proof}

\begin{corollary}
\label{cor:twisted-generate}
For any $r\in \Delta_{A}\bigcup \Phi^*_{\rho,A}$,  the subgroups $U_{[s]}$, $s\in\Delta_{A}\bigcup \Phi^*_{\rho,A}$ with $s$ not proportional to $r$ generate the group $G_\rho$.
\end{corollary}

\begin{proof}
By Lemma \ref{le:6} we see that any weight subspace with its weight proportional to $r$ can be generated by
elements inside the subgroups $U_{[s]}$, with $s$ not proportional to $r$. Lemma \ref{lem:det-gen} also shows that $U_{G,[\pm r]}$ can also be generated by elements inside the subgroups $U_{[s]}$, with $s$ not proportional to $r$. Hence we finish the proof.
\end{proof}

\subsection{Proof of Theorem \ref{thm:central} when $\mathcal{G}=G_\rho$}
Now we are ready to prove Theorem \ref{thm:central} when $\mathcal{G}=G_\rho$. Suppose $\widetilde{p} \in \widetilde{\mathcal{G}}_A/ \mc S_A$ projects to $e\in G_\rho$. By Proposition \ref{prop:semidirect-path-struct}, we can write $\widetilde{p}=\widetilde{g} \widetilde{x}\widetilde{y}\widetilde{z}$ as decried in Proposition \ref{prop:semidirect-path-struct}. It is clear that $\widetilde{g}$, $\widetilde{x}$, $\widetilde{y}$ and $\widetilde{z}$ all project to $e\in G_\rho$. By \eqref{for:24}
of Corollary \ref{prop:zero-weight-group-struct} we see that $\widetilde{y}\in Z(\widetilde{\mathcal{G}}_A/ \mc S_A)$. We also see that $\widetilde{x}=e$ since each $\widetilde x_{r_i}$ projects to $e\in G_\rho$ and hence $\widetilde x_{r_i}=e$ for each $i$ by Lemma \ref{lem:detected-comm}.
Next, we need to show that $\widetilde{g}\in Z(\widetilde{\mathcal{G}}_A/ \mc S_A)$. We first prove that ``good'' elements of $\prod_{r \in \Delta_A} \tilde{U}_{G,[r]}$ preserve the groups making up the product $\prod_{\mu \in \Phi_{A,\rho}^*} \tilde{U}_{\mf e,[r]}$:

\begin{lemma}\label{fact:9}
For any $\widetilde{h}\in \widetilde{H}_\mu$ (resp. $\widetilde{h}\in \widetilde{\textbf{H}}_\mu$) for some $\mu \in \Delta_{A_0}$, then for any $\widetilde{u}\in \widetilde{U}_{\mathfrak{e},[s]}$, $s\in \Phi^*_{A,\rho}$, we have $\widetilde{h}\widetilde{u}\widetilde{h}^{-1}\in  \widetilde{U}_{\mathfrak{e},[s]}$.
\end{lemma}

\begin{proof}
We divide the proof into three cases:

\textbf{Case $(a)$}: If $\mu$ is detected and $\mu'=\mu|_A$ is not proportional to $s$, then $\widetilde{h}\widetilde{u}\widetilde{h}^{-1}\in \prod_{j\in \ZZ}\tilde{U}_{\mathfrak{e},[j\mu'+s]}$.
Note that set $S=\{j\mu'+s:j\in\ZZ\}$ is $A$-admissible and $\widetilde{h}\widetilde{u}\widetilde{h}^{-1}$ projects to an element inside $\tilde{U}_{\mathfrak{e},[s]}$. Then by Lemma \ref{lem:detected-comm} we get the result for this case.

\textbf{Case $(b)$}: Suppose $\mu$ is detected and $\mu'=\mu|_A$ is proportional to $s$. We can write $\widetilde{u}=\prod_{i\in I}\widetilde{u_i}$, where $\widetilde{u_i}\in \widetilde{U}_{\mathfrak{e},[s],i}$. By Lemma \ref{le:6} each $\widetilde{u_i}$ is inside $\prod_{\theta\in S_i}\widetilde{U}_{[\theta]}$, where $S_i$ is contained in an $A$-admissible set and any $\theta\in S_i$ is not proportional to $s$.
Then by previous argument, $\widetilde{h}\widetilde{u_i}\widetilde{h}^{-1}\in  \prod_{\theta\in S_i}\widetilde{U}_{[\theta]}$, $i\in I$. Note that each $S_i$ is contained in an $A$-admissible set and $\widetilde{h}\widetilde{u_i}\widetilde{h}^{-1}$ projects to an element inside
$\widetilde{U}_{\mathfrak{e},[s],i}$. Then by Lemma \ref{lem:detected-comm} $\widetilde{h}\widetilde{u_i}\widetilde{h}^{-1}\in \widetilde{U}_{\mathfrak{e},[s],i}$ for each $i$. This shows that $\widetilde{h}\widetilde{u}\widetilde{h}^{-1}\in \widetilde{U}_{\mathfrak{e},[s]}$.

\textbf{Case $(c)$}: Suppose $\mu$ is undetected. Then there exists detected $r_1,\,r_2\in \Delta_{A_0}$, such that $w_{r_2}(r_1)=\mu$ and $r_i|_A$ is linearly independent with $s$ (see Lemma \ref{lem:det-gen}). From the construction of \eqref{for:43} we see because $r_1|_A$ is nonproportional to $s$, we reduce this case to Case $(a)$.
\end{proof}

\begin{proposition}
\label{cor:semisimple-centrality}
If $\widetilde{g}\in \Pi_{r\in \Delta_A}\widetilde{U}_{G,[r]}$ and projects to $e$ in $G_\rho$
then $\widetilde g \in Z(\widetilde{\mathcal{G}}_A/ \mc S_A)$.

\end{proposition}

\begin{proof}
By Corollary \ref{cor:3}, we can write $\widetilde{g}=\prod_{\mu\in \Delta}\widetilde{h}_{\mu}$, where $\widetilde{h}_{\mu}\in \widetilde{H}_\mu$ (resp. $\widetilde{h}_{\mu}\in \widetilde{\textbf{H}}_\mu$). By iteratively appling Lemma \ref{fact:9}, we get that for any $\widetilde{u}\in \widetilde{U}_{\mathfrak{e},[s]}$, $s\in \Phi^*_{A,\rho}$, $\widetilde{g}\widetilde{u}\widetilde{g}^{-1}\in \widetilde{U}_{\mathfrak{e},[s]}$. Since $\widetilde{g}\widetilde{u}\widetilde{g}^{-1}$ projects to $e\in G_\rho$, $\widetilde{g}\widetilde{u}\widetilde{g}^{-1}=\widetilde{u}$. So $\widetilde{g}$ is in the centralizer of $\prod_{s\in \Phi^*_{A,\rho}}\widetilde{U}_{\mathfrak{e},[s]}$. Since $\widetilde{g}$ is also in the centralizer of $\Pi_{r\in \Delta_{A}}\widetilde{U}_{G,[r]}$ (see Section \ref{cor:simple-central}) we get that $\widetilde{g}\in Z(\widetilde{\mathcal{G}}_A/ \mc S_A)$.
\end{proof}

Combining Proposition \ref{cor:semisimple-centrality} and \ref{prop:semidirect-path-struct}, we immediately get that $\widetilde{\mathcal{G}}_A/ \mc S_A$ is a central extension of $G_\rho$. So to prove Theorem \ref{thm:central}, we only need to show that $\widetilde{\mc G}_A / \mc S_A$ is perfect:

\begin{proof}
We note that any $\widetilde{g}\in \prod_{r\in \Delta_A}\widetilde{U}_{G,[r]}$ is contained in the commutator group of $\widetilde{\mathcal{G}}_A/ \mc S_A$ (see Section \ref{cor:simple-central}). Then we just need to show that any element inside $\widetilde{U}_{\mathfrak{e},[s],i}$, $i\in I$ and $s\in \Phi^*_{\rho,A}$ is contained in the commutator group of $\widetilde{\mathcal{G}}_A/ \mc S_A$.

Suppose that $x \in U_{\mf e, [s],i}$. Then since $[s] \not= [0]$, there exists some $a \in p(A)$ such that $\lambda = e^{s(a)} \not= 1$. Let $x' = x/(\lambda-1)$ (where we write $x$ additively. Since $\mbf{H}$ projects to the centralizer of $A_0$ in $G$ (see Proposition $1.8$ \cite{deodhar78}), we can choose some element $\widetilde{h} \in \mbf{H}$ (see Section \ref{sec:kernel-sec}) such that $\pi_{\mc G}(\widetilde{h}) = a$. Then by Lemma \ref{fact:9}:

\[ [\widetilde{h},\widetilde{x}'] = \widetilde{\lambda x'}\cdot \widetilde{-x'} = \widetilde{(\lambda-1)x'} = \widetilde{x}. \]

\end{proof}

\section{Topological Properties of $\widetilde{\mc G}_A$ and its subgroups and factors}\label{sec:17}

To prove Theorem \ref{th:3}, we first need to study certain topological properties of the groups $\widetilde{\mc G}_A$ and $\mc C_A$ with the topology described in Section \ref{sec:7}. We first prove a few technical Lemmas:

\begin{lemma}
\label{lem:extension-diagram}
Assume that $1 \to \Upsilon \to X_2 \xrightarrow{p} X_1 \to 1$ is a topological extension of $X_1$. Also, assume that the $X_i$ have a common topological extension $1 \to Z_i \to X \xrightarrow{\pi_i} X_i \to 1$ for $i = 1,2$ satisfying $\pi_1 = p \of \pi_2$. Then there is a topological extension:

\[ 1 \to Z_2 \to Z_1 \to \Upsilon \to 1 \]
\end{lemma}

\begin{proof}
First, observe that if $z \in Z_2$, then $\pi_1(z) = p(\pi_2(z)) = p(e) = e$, so $Z_2 \subset Z_1 = \ker \pi_1$. Furthermore, $Z_2$ is normal in $Z_1$ since it is normal in $X$, which contains $Z_1$. We claim that $\pi_2(Z_1) = \Upsilon$. Indeed, if $z \in Z_1$, then $p(\pi_2(z)) = \pi_1(z) = e$, so $\pi_2(z) \in \ker p = \Upsilon$. It is surjective since if $u \in \Upsilon$, we may choose an element $x \in X$ such that $\pi_2(x) = u$. Then $\pi_1(x) = p(\pi_2(x)) = p(u) = e$, so $x \in Z_1$. To see that $Z_2$ is the kernel, observe that if $z \in Z_1$, then $\pi_2(z) = e$ if and only if $z \in \ker \pi_2 = Z_2$. Since $Z_2 \subset Z_1$, we conclude the result.
\end{proof}

\begin{lemma}
\label{lem:discrete-ext}
Suppose that $1 \to Z \to Y \xrightarrow{p} X \to 1$ is a topological extension of groups with $X$ discrete. Then $Z$ is locally path-connected if and only if $Y$ is.
\end{lemma}

\begin{proof}
Note that $\set{e} \subset X$ is open in $X$ since $X$ is discrete, so $p^{-1}(\set{e}) = \ker p = Z$ is open in $Y$. That is $Z$ is an open subgroup, so $e$ has an open neighborhood in $Z$ if and only if it has an open neighborhood in $Y$.
\end{proof}

\begin{proposition}\label{le:5}
With the free group topology, $\widetilde{\mathcal{G}}_A$ and $\mc C_{A}$ are are locally path-connected.
\end{proposition}

\begin{proof}
From Section \ref{sec:7}, we know that the topology on $\widetilde{\mc G}_A$ is the quotient topology after identifying elements of each combinatorial pattern. Since we know that the quotient of a locally path-connected space is locally path-connected, we conclude that $\widetilde{\mc G}_A$ is locally path-connected.

To consider the groups $\mc C_A$, we first assume that $\mc G$ is algebraic and $A$ is an algebraic subgroup of $\mc G$ (and generalize to nonalgebraic groups afterwards). In this case, each $U_{[r]}$ is also an algebraic subgroup, so the natural maps $\iota_{\mbf r} : X_{\mbf r} \to \mc G$ (see Section \ref{sec:15} and \ref{sec:7}) are algebraic morphisms. Thus the preimage of the identity must be an subvariety of $U_{\mbf r}$. Local path-connectedness of $\mc C_A$ then follow from local path-connectedness of algebraic varieties.

Now we extend to the case of non-algebraic groups. Since $\mc G$ is perfect and has discrete center in both the semisimple and twisted cases, it always has a center-free factor $\underline{\mc G}$ with the same Lie algebra (via the adjoint representation). So $\mc G$ fits into a short exact sequence $1 \to \Upsilon \to \mc G \to \underline{\mc G} \to 1$.

Note that $\mc C_{A,\underline{\mc G}}$ is locally-path connected by our previous argument in the case of algebraic groups. Furthermore, $\widetilde{\mc G}_A$ covers both $\mc G$ and $\underline{\mc G}$. So we are in the setup of Lemma \ref{lem:extension-diagram}, and conclude that $1 \to \mc C_{A,\mc G} \to \mc C_{A,\underline{\mc G}} \to \Upsilon \to 1$. So by Lemma \ref{lem:discrete-ext}, we conclude the result.
\end{proof}
Recall that $\mc G_0$ is either a simply connected, semisimple Lie group, or a maximal Lie central extension of a semidirect product $G_\rho$. Then:

\begin{proposition}\label{prop:central-splits}
If $1 \to Z \to H \xrightarrow{\eta}\mc G_0 \to 1$ is a central topological extension of $\mc G_0$ by a Lie group $Z$ with a local section (in the sense that $H$ a principal bundle over $\mc G_0$). Then $H$ is a Lie group and $H \cong \mc G_0 \times Z$.
\end{proposition}

\begin{proof}
We first show that $H$ is locally compact. Suppose $\sigma : L \to H$ is a local section at $g_0\in \mc G_0$ on a small compact neighborhood $L$ of $g_0$. Then $\eta(\sigma(g))=g$ for any $g\in L$.  Set $h=\sigma(g_0)$. Choose a compact neighborhood $K \subset Z$ of identity. We can form the map $\eta : K \times L \to H$ defined by $(k,l) \mapsto k\sigma(l)$. Then $\eta$ is a homeomorphism by definition. We claim it contains a neighborhood of $h\in H$. Indeed, it suffices to show that for any net $x_\alpha \to h$, $x_\alpha \in \eta(K\times L)$ for sufficiently large $\alpha$. Indeed, we may first project $x_\alpha$ to $y_\alpha = \eta(x_\alpha)$. Then since $x_\alpha \to h \in H$, $y_\alpha \to g\in \mc G_0$ by continuity of $\eta$. Hence there exists $\alpha_0$ such that if $\alpha \ge \alpha_0$, $y_\alpha \in L$. Now, let $z_\alpha =  x_\alpha \cdot \sigma(y_\alpha)^{-1}$. Then $\eta(z_\alpha) = \eta(x_\alpha\cdot \sigma(y_\alpha)^{-1}) = e$, so $z_\alpha \in Z$, which implies that $y_\alpha \to g_0$ and $x_\alpha \to h$, $z_\alpha \to e$ by continuity of the section and group multiplication. Then there exists $\alpha_1$ such that if $\alpha \ge \alpha_1$, $z_\alpha \in K$ by closeness of $K$. By construction, $x_\alpha = \sigma(y_\alpha)z_\alpha = \eta(z_\alpha,y_\alpha)$, so if $\alpha \ge \max\{\alpha_0,\,\alpha_1\}$, we conclude that $x_\alpha \in \eta(K \times L)$.

Thus we conclude that $H$ is locally compact and locally path-connected (since the compact neighborhood constructed above can also obviously be made path-connected). This implies it is a Lie group by the classical theorem of Gleason-Montgomery-Zippin.

Let $[H,H]$ be the commutator group of $H$. Then $H=[H,H]\ker \eta=[H,H]Z$ and $[H,H]$ is perfect (see pg $75$ of \cite{Steinberg2}). Let the connected component of $[H,H]$ be $H'$. Then $\eta(H')=\mc G_0$ since $\mc G_0$ is connected. Therefor, $H=H'Z$. If $\mc G_0$ is simply connected, then it is clear that $H'\bigcap Z=\{e\}$; if $\mc G_0$ has non-trivial radical, then Proposition \ref{prop:universal-property} also implies that $H'\bigcap \ker\eta=\{e\}$. This further shows that $H'$ is isomorphic to $L$. Hence we prove the result.
\end{proof}

We document a result of Gleason and Palais which will be used in the main proof of this section:

\begin{theorem}[Corollary 7.4,\cite{gleason57}]\label{th:2}
If $S$ is a locally path-connected topological group in which some neighborhood of the identity admits a continuous one-to-one map into a finite dimensional metric space, then $S$ is a Lie group.
\end{theorem}

\subsection{Proof of Theorem \ref{th:3}}
Suppose $f:\mc C_A/\mc S_A\rightarrow S$ is a homomorphism  and $S$ is a Lie group. We wish to show that $f$ is trivial. Then $f$ induces an injection from $P=\left. (\mc C_A/\mc S_A) \big/ \ker(f) \right.$. By Lemma \ref{le:5} $P$ is locally path connected, then Theorem \ref{th:2} implies that $P$ is a Lie group. Note that since $\ker f \subset \mc C_A / \mc S_A$ and $\mc C_A / \mc S_A$ is central in $\widetilde{\mc G}_A / \mc S_A$, $\ker f$ is a normal subgroup of $\widetilde{\mc G}_A / \mc S_A$. So we may form a group $\overline{\mc G}_A = \left(\widetilde{\mc G}_A / \mc S_A\right) \big/ \ker f$, and we get a short exact sequence:

\[ 1 \to P \to \overline{\mc G}_A \to \mc G_0 \to 1 \]
\begin{lemma}\label{le:7}
The projection $\overline{\mc G}_A \to \mc G_0$ has a local section at some point $g_0\in \mc G_0$.
\end{lemma}

\begin{proof}
From Lemma \ref{cor:det-gen} and Corollary \ref{cor:twisted-generate} implies that there is a combinatorial pattern $\mbf{r} = (r_1,\dots,r_n)$ such that the map $\iota_{\mbf r} : X_{\mbf r} \to \mc G_0$ is onto. Then there exists a $g_0 \in \mc G_0$ such that $g_0$ is a regular value of the projection (by Sard's Theorem). This implies that there is a (smooth) section $i_{g_0} : B(g_0) \to X_{\mbf r}$, where $B(g_0)$ is a small neighbourhood of $g_0$ in $\mc G_0$, which induces a local section at $g_0$.
\end{proof}


Now by Proposition \ref{prop:central-splits}, $\overline{\mc G}_A \cong \mc G_0 \times P$. By noting that $\overline{\mc G}_A$ is a factor of a perfect group $\widetilde{\mc G}_A / \mc S_A$ and is hence perfect itself and $P$ is central, we get immediately that $P$ is central $P = \set{e}$, which implies that $\ker(f)=\mc C_A/\mc S_A$.  Hence we finish the proof.

\section{Trivialization of small lattice homomorphisms}
\label{sec:cocyle-rigid}

Recall the notations of \seref{se:1} and \seref{se:3} of Section \ref{sec:14}, so that $\Gamma$ is a lattice in $G$, $\Gamma_\rho$ is a lattice in $G_\rho$, $\Gamma'$ (or $\Gamma_\rho'$) is a lattice in the group $\mc G_0$ and $\widetilde{\Gamma}$ is the product of the lattice and center $\exp(\mathfrak{z})$ in the case of groups with symplectic contributions.

The following lemma treats cases which occur in the proof of main theorems
when one studies the Lyapunov-cycles of different homotopy classes.
\begin{lemma}
\label{cor:no-lattice-morphisms}
\hspace{1cm}

\begin{enumerate}
\smallskip
  \item \label{for:29}For any homomorphism $f : \Gamma' \to N$, $f(\Gamma')$ is finite and the order is bounded by a universal number only dependent on $G$. Especially, if $f$ is sufficiently small then $f$ is trivial.

  \smallskip
  \item \label{for:30.5} If $f : \Gamma_\rho' \to N$ is a sufficiently small homomorphism, then $f$ is trivial.

  \smallskip
  \item \label{for:31} If $f : \widetilde{\Gamma} \to N$ is a sufficiently small homomorphism, then $f$ is trivial.
\end{enumerate}

\end{lemma}

\begin{proof}
Proof of \eqref{for:29}: First we consider the case $N\subset G$.  We can assume
that $N\subset GL(\ell,\CC)$ is a linear algebraic group since the center of $N$ is finite. Since $A$ is genuinely higher rank, $0\neq \Lie(N)\bigcap \mathfrak{g}_i\neq \mathfrak{g}_i$, for each simple factor $\mathfrak{g}_i$ of $\Lie (\mathcal{G}_0)$. By \cite[3'Theorem]{margulis91}, Zariski closure
$\overline{f(\Gamma')}^z$ of $f(\Gamma')$ is a semisimple $\QQ$-algebraic group. We first
show that $f(\Gamma')$ is finite. Suppose it is not finite. Since
$\overline{f(\Gamma')}^z/(\overline{f(\Gamma')}^z)^0$ is finite, we can assume
$\overline{f(\Gamma')}^z$ is connected inside $N$. Compose $f$ with a Galois
automorphism $\tau$ of $\CC$ over $\QQ$ to matrix coefficients of
elements from $f(\Gamma')$, then $f(\Gamma')$ is a non-compact subgroup
of $\overline{f(\Gamma')}^z=\sigma\overline{f(\Gamma')}^z$. We can assume $\tau f$ is from $\Gamma'$ to
$\overline{f(\Gamma')}^z/Z(\overline{f(\Gamma')}^z)$ since $Z(\overline{f(\Gamma')}^z)$ is finite.
By Margulis lattice
superrigidity Theorem \cite[2' Theorem]{margulis91}, $\tau f$ can
be extended to a continuous homomorphism $\widetilde{f}$ from
$\tilde{G}_0$ to $\overline{f(\Gamma')^z}/Z(\overline{f(\Gamma')}^z)$.
Then $\ker(\widetilde{f})$ restricted to each simple factor $\widetilde{G}_i$ is either $\widetilde{G}_i$ or inside $Z(\widetilde{G}_i)$. Since the complexification of each
$\mathfrak{g}_i$ is identical (see $4.1$ of Chapter IX, \cite{margulis91}) and $\text{Lie}(N)\bigcap \mathfrak{g}_i$ is strictly inside $\mathfrak{g}_i$ for each $i$,
the above discussion shows that $\ker(\widetilde{f})|_{\widetilde{G}_i}=\widetilde{G}_i$ for each $i$. 
Hence we proved that $f(\Gamma')$ is finite.

Next, we show this order is bounded by a universal number only dependent on $G$.
By Jordan's theorem which claims that any finite group inside $GL(\ell,\CC)$ contains a normal abelian subgroup whose
index is at most $j(\ell)$. We let the biggest normal abelian
subgroup in $f(\Gamma')$ be $Q$. Consider the restriction of $f$ from
$f^{-1}(Q)$ to $Q$. The index of $[\widetilde{\Gamma}:f^{-1}(Q)]$ is bounded by
$j(\ell)$. There are only finitely many sublattices $L'$ in $\Gamma'$ with
the index smaller than $j(\ell)$ since
$\Gamma'$ is finitely generated (as a group)\cite{margulis91}. Hence we show that $f(\Gamma')$ is finite and the order is bounded by a universal number.

Now suppose $N\subset G_\rho$. By previous argument $f: \Gamma' \to N \to N \cap G$ is finite and the order is bounded by a universal number (note that $N \cap G$ is a factor of $N$). The trivialization of $f: \Gamma' \cap G \to N\bigcap \RR^N$ follows directly from by the Margulis Normal Subgroup Theorem \cite[4' Theorem]{margulis91}. Hence we also get the result for
$N\subset G_\rho$. Hence we finish the proof.

\smallskip
\noindent Proof of \eqref{for:30.5}: By previous argument, $f$ is trivial on $\Gamma'$ which implies that $[\gamma, \widetilde{n}]\in \ker(f)$ for any $\gamma \in \Gamma_\rho'$ and $\widetilde{n}\in \mathcal{G}_0$ of the form $\widetilde{n}$ for $\widetilde{n} \in \mc Z$ (see the end of Section \ref{sec:16}). Our goal will be to show that the group generated by $[\gamma,\widetilde{n}]$ generates a finite index subgroup of $\Gamma_\rho' \cap E'$. Suppose that we show it for $\Gamma_\rho \cap E$. Then since multiplication on $\Gamma_\rho'$ covers multiplication on $\Gamma_\rho$ and $\exp(Z)$ generates $\Gamma_\rho' \cap E'$ by construction, we conclude that it suffices to show that $\set{[\gamma,n] : n \in \Z^n}$ generates a finite index subgroup of $\Z^n$.


Suppose $\RR^N=\bigoplus_{i\in I} \RR^{N_i}$ such that $\rho$ is irreducible on each $\RR^{N_i}$. For any $0\neq n\in \ZZ^{N_i}$, the set $\rho(\Gamma)(n)$ is Zariski dense in $\RR^{N_i}$ since $\rho$ is rational and $\Gamma$ is Zariski dense in $G$ by the Borel density theorem (see 3.2.5 of \cite{zimmer}). This shows that $\rho(\Gamma)(n)$ is not inside any subspace of $\RR^{N_i}$.
Fix $\ell=[\gamma,n]\neq 0$, $\gamma\in \Gamma$. Then we can choose finite $\gamma_{j,i}\in\Gamma$, $j\in J$ such that $\rho(\gamma_{j,i})(\ell)$ span a basis of $\RR^{N_i}$. Then $[\gamma_{j,i},\ell]$, $j\in J$ and $\ell$ also a span a basis for $\RR^{N_i}$. This shows that $[\Gamma,n]$ span a basis of each $\RR^{N_i}$ for any $0\neq n\in\ZZ^{N_i}$. Hence we see that the set
\begin{align*}
  \{[\gamma,n]:\gamma\in \widetilde{\Gamma}, \, n\in\ZZ^N\}
\end{align*}
span a sublattice with finite index in $\Gamma_\rho'$. This shows that $f(\Gamma_\rho')$ is finite. Since $f$ is small then it is trivial.

\smallskip
\noindent Proof of \eqref{for:31}: By previous argument $f$ is trivial on $\Gamma_\rho'$. Since $\Gamma_\rho'$ contains a lattice of $\exp(\mathfrak{z})$ and $f$ is small, $f$ induces a small homomorphism from a torus. There are no small homomorphisms from tori to a fixed Lie group, so we get trivialization on $\widetilde{\Gamma}$ as well.
\end{proof}

\section{Lyapunov Structures of Perturbations}
\label{sec:perturbed}
\subsection{Coarse Lyapunov foliations for perturbations}\label{sec:19}
Now we consider the family of Lyapunov foliations for perturbations which correspond
to the unipotent foliations $\mathcal{U}_\mu$, $\mu\in \Lambda_A$ (see \eqref{se:3} of Section \ref{sec:14}) for $\alpha_A$.
Recall notations in \ref{sec:12}. We denote the family of coarse Lyapunov foliations for $\widetilde{\alpha}_A$ by $\widetilde{\mathcal{F}}$. Since
$\overline{\alpha}_A$ and $\widetilde{\alpha}_A$ are conjugate via a H\"older map, the stable and unstable, and
thus the corresponding coarse Lyapunov foliations $\overline{\mathcal{F}}$, are dynamically defined for
the action $\overline{\alpha}_A$.  Notice that we do not know whether the leaves of those foliations
are smooth manifolds. However, every foliation in $\widetilde{\mathcal{F}}$ is integrable with
$\mathcal{N}$ and the resulting center-Lyapunov foliation does have smooth leaves
since it coincides with the center-Lyapunov foliation for $\alpha_A$.

More precisely, let $a_i\in I$
be different elements within the Weyl chambers and
away from the Weyl chamber walls such that:
\begin{enumerate}
  \item for any $\mu\in \Lambda_A$, there is subset $I_\mu\subset I$ such that $\mathcal{U}_\mu =
\bigcap_{a_i\in I_\mu} \mathcal{W}^s_{a_i}$ (i.e., such that $\mu(a_i)<0$, where $\mathcal{W}^s_{a_i}$
is the stable foliation for
$\alpha_A(a_i,\cdot)$);

  \medskip
  \item  if $\mu$ and $\nu$ are not negatively propositional in $\Lambda_A$, then the foliations $\mathcal{U}_\mu$ and $\mathcal{U}_\nu$ are included in the stable foliations of $\alpha_A(a,\cdot)$ for some $a\in I_\mu\bigcup I_\nu$.
\end{enumerate}
We then
denote the corresponding intersections of stable foliation for $\widetilde{\alpha}_A$ by $\widetilde{\mathcal{F}}_\mu$
i.e., $\widetilde{\mathcal{F}}_\mu=\bigcap_{a_i\in I_\mu} \widetilde{\mathcal{W}}^s_{a_i}$
are the corresponding coarse Lyapunov foliations for the
action $\widetilde{\alpha}_A$. We also denote coarse Lyapunov foliations $h^{-1}\widetilde{\mathcal{F}}_\mu$
for $\overline{\alpha}_A$ by $\overline{\mathcal{F}}_\mu$.

\subsection{Canonical projections between $\overline{\mathcal{F}}$ and $\mathcal{U}$}

In this section, we link the coarse Lyapunov foliations of $\alpha_{\tilde{A}}$ (resp. $\alpha_{A'}$) (see \seref{se:5} and \seref{se:7} of Section \ref{sec:14}) with those of $\hat{\alpha}_{A}$, which are also denoted by $\mathcal{U}$ and $\overline{\mathcal{F}}$ respectively.

The foliations $\mathcal{N}$ and $\mathcal{U}_\mu$ on $\mathcal{G}_0$ integrate to an invariant foliation $\mathcal{W}_\mu$ for $\alpha_{\tilde{A}}$ (resp. $\alpha_{A'}$) with the product
structure. This foliation is also invariant for $\hat{\alpha}_{A}$. Moreover, every leaf of $\overline{\mathcal{F}}_\mu$ inside this foliation intersects every leaf of $\mathcal{N}$ at a unique point since $\overline{\mathcal{F}}_\mu$ is $\kappa$-H\"older close to $\mathcal{U}_\mu$. More precisely, for any $x'\in \mathcal{U}_\mu(x)$ and $y\in \mathcal{N}(x)$ we can define a map $P^{x,y}_\mu : x'\rightarrow \overline{\mathcal{F}}_\mu(y)\bigcap\mathcal{N}(x')$, which is a well defined and continuous map taking $\mathcal{U}_\mu(x)$ to $\overline{\mathcal{F}}_\mu(y)$ (see \cite{dk2011}). Also we can define a continuous map: $\widehat{P}^{x,y}_\mu:x'\rightarrow \mathcal{U}_\mu(y)\bigcap\mathcal{N}(x')$, for any $x'\in \overline{\mathcal{F}}_\mu(x)$ and $y\in \mathcal{N}(x)$.

We then extend the projections to paths in the obvious way:

\begin{definition}
Let $\mathfrak{c}:x_1,\cdots,x_n=x_1$ be a $\mathcal{U}$-path with initial
point $x_1$ and $x_{k+1}\in \mathcal{U}_{\mu_i}$, $1\leq k\leq n-1$. Then the canonical projection of this path at a point $y\in \mathcal{N}(x_1)$ is an $\overline{\mc F}$-path $P^{x_1,y}(\mathfrak{c}): y=y_1,y_2\cdots,y_n$ such that for
$y_{k+1}:=P^{x_k,y_k}_{\mu_k}$, $1\leq k\leq n-1$.

Similarly we define the reverse projection $\widehat{P}^{x_1,y}(\mathfrak{c})$ of an $\overline{\mathcal{F}}$-path $\mathfrak{c}$
with an initial point $x_1$ to a $\mathcal{U}$-path starting at $y$,
using canonical projections $\widehat{P}^{x_1,y}(\mathfrak{c})$.
\end{definition}

\begin{remark}\label{re:5}
Canonical projections are mutually inverse:
\begin{align*}
  \widehat{P}^{y,x}_\mu\big(P^{y,x}_\mu(x')\big)=x',\qquad P^{y,x}_\mu\big(\widehat{P}^{y,x}_\mu(x')\big)=x'.
\end{align*}
\end{remark}

For any $h\in \mathcal{G}_0$ and $\widetilde x= \widetilde u_1\widetilde u_{2}\cdots \widetilde u_{n}\in \widetilde{\mathcal{G}}_A$ (see Section \ref{sec:7}), $\widetilde u_i\in \widetilde U_{[\mu_i]}$, $\mu_i\in \Lambda_A$, let $h_i=u_{n-i+1}\cdots u_nh$. We define a corresponding $\mathcal{U}$-path $\widetilde{x}_h$ on $\mathcal{G}_0$ based at $h$: $h=h_1, h_2,\cdots,h_n$. We define $\mathcal{P}(\widetilde{x},h)$ to be the endpoint of the $\overline{\mathcal{F}}$-path $P^{h,h}(\widetilde{x}_h)$ (ie, $\mc P(\widetilde x,h) = \pi_{\mc F}(P^{h,h}(\widetilde x_h)$). Note that $\beta$ is the lifted cocycle for $\hat{\alpha}_A$ (see Section \ref{sec:20}).
\begin{remark}\label{re:6}
For any $\widetilde{x}\in \mathcal{C}_A$ and $h\in \mathcal{G}_0$, $\widetilde{x}_h$ is a cycle with initial point $h$ in the sense defined in Section \ref{sec:5}. Also, if $\widetilde x, \widetilde y \in \mathcal{C}_A$, $(\widetilde{x}\widetilde{y})_h=\widetilde{x}_{h}*\widetilde{y}_h$. Since $\widetilde{x}\rightarrow \widetilde{x}_h$ is one-to-one, hence we can identify $\mathcal{C}_A$ with the $\mathcal{U}$-cycle group at $h$. Then we also call $\mathcal{C}_A$ the cycle group in $\widetilde{\mathcal{G}}_A$. We call the subgroup $\mathcal{S}_A$ the \emph{stable cycle group}. The reason is for any $h\in \mathcal{G}_0$ and any relation defined by $R_{\mu,\nu,A}$ (see Section \ref{sec:15}) in $\mathcal{G}_0$, we can choose $a_i\in I$ (see Section \ref{sec:19}) such that the $\mathcal{U}$-cycle defined by $R_{\mu,\nu,A}$ with initial point at $h$ is inside a stable leaf for $\alpha_{\widetilde{A}}(a_i,\cdot)$ (resp. $\alpha_{A'}(a_i,\cdot)$).
\end{remark}

\begin{proposition}\label{po:3}
For any $h\in \mathcal{G}_0$ $\mathcal{P}(\cdot,h)$ is continuous on $\widetilde{\mathcal{G}}_A$ and descends to a continuous map on $\widetilde{\mathcal{G}}_A/\mathcal{S}_A$.
\end{proposition}

\begin{proof}
To see continuity, from Section \ref{sec:7} it suffices to show that $\mc P(\cdot,h)$ is continuous in each combinatorial pattern. This follows from continuity of the canonical projection of each leg (as the projection of paths is defined inductively based on the length of the path).

Next we show that $\mathcal{P}(\sigma ,h)=h$ for any $\sigma \in\mathcal{S}_A$. Suppose that $\textbf{r}$ is defined by a relation of the form $R_{\mu,\nu,A}$ in $\widetilde{\mathcal{G}}_A$. By Remark \ref{re:6} we see that $\mathcal{P}(\textbf{r} ,h)=h$. We also have $\mathcal{P}(\widetilde{y}\sigma\widetilde{y}^{-1} ,h)=h$ for any $\widetilde{y}\in\widetilde{\mathcal{G}}_A$.
Note that products of elements having the form $\widetilde{y}\sigma\widetilde{y}^{-1}$ generate $\mathcal{S}_A$. Hence we get the result.
This proves that $\mathcal{P}$ can descend to $\widetilde{\mathcal{G}}_A/\mathcal{S}_A$. Hence we finish the proof.
\end{proof}

\subsection{Isomorphism between $\overline{\mathcal{F}}$ and $\mathcal{U}$ cycle groups}
Let $\mathcal{C}^{\mathcal{T}}_x$ denote the group of $\mathcal{T}$-cycles with with initial point
$x\in \mathcal{G}_0$ (see Section \ref{sec:5}), where $\mathcal{T}$ stands for $\mathcal{U}$ or $\mathcal{\overline{F}}$). The next theorem is crucial to the proof of main Theorem:

\begin{theorem}
\label{thm:perturbed-proj}
$P^{h,h} : \mathcal{C}^{\mathcal{U}}_h \to \mathcal{C}^{\mathcal{\overline{F}}}_h$ is a topological isomorphism for $h\in \mathcal{G}_0$ which takes the stable cycles for $\alpha_A$ to the stable cycles for $\overline{\alpha}_A$.
\end{theorem}

\begin{proof}
We proceed in two steps. First we will show  {\bf (a)} $P^{h,h}\left(\mc C^{\mc U}_h,h\right) \subset \mc C^{\overline{\mc F}}_h$ and second, we will show that {\bf (b)} $\widehat{P}^{h,h}\left(\mc C^{\overline{\mc F}}_h\right) \subset \mc C^{\mc U}_h$.

\textbf{Step (a)}: In this step we show that for any $\widetilde{x}\in \mathcal{C}_A$, $\mathcal{P}(\textbf{x}, h)=h$. By Proposition \ref{po:3}
$\mathcal{P}(\cdot, h)$ is a continuous map from $\mathcal{C}_A/\mathcal{S}_A$ to $\mathcal{G}_0$. Set
\begin{align*}
 \textbf{S}_h=\{\widetilde{x}\in \mathcal{C}_A/\mathcal{S}_A: \mathcal{P}(\textbf{x}, h)=h\}.
\end{align*}
Then $\textbf{S}_h$ is a subgroup of $\mathcal{C}_A/\mathcal{S}_A$. Note that $\mathcal{C}_A/\mathcal{S}_A$ is abelian (see Theorem \ref{thm:central} and Remark \ref{re:3}). Hence $\mathcal{P}(\cdot, h)$ descends to a continuous and injective map on the quotient group $\left. (\mathcal{C}_A/\mathcal{S}_A ) \big/ \textbf{S}_h \right.$. Note that $\left. (\mathcal{C}_A/\mathcal{S}_A ) \big/ \textbf{S}_h \right.$ is locally path connected (see Lemma \ref{le:5}). Then by Theorem \ref{th:2}
$\left. (\mathcal{C}_A/\mathcal{S}_A ) \big/ \textbf{S}_h \right.$ is a Lie group. Then the projection from $\mathcal{C}_A/\mathcal{S}_A $ to $\left. (\mathcal{C}_A/\mathcal{S}_A ) \big/ \textbf{S}_h \right.$ is trivial by Theorem \ref{th:3}. This implies that $\textbf{S}_h=\mathcal{C}_A/\mathcal{S}_A$. Then we prove the the result. This shows that $P^{h,h}$ is a injective from $\mathcal{C}^{\mathcal{U}}_h$ to $\mathcal{C}^{\mathcal{\overline{F}}}_h$.

\medskip
\textbf{Step (b)}: In this part we show that $\widehat{P}^{h,h}(\mathcal{C}^{\mathcal{\overline{F}}}_h)\subset \mathcal{C}^{\mathcal{U}}_h$. In fact, if we get {\bf (a)}, then
the argument proving the opposite direction is standard (see \cite[Subsection 6.4] {dk2011} or \cite[Proposition 6.7]{vinhage15})). To make the proof complete, we briefly outline it here. For any $\mathfrak{c}\in \mathcal{C}^{\mathcal{\overline{F}}}_h$, the end point of $\overline{P}^{h,h}(\mathfrak{c})$ has the expression $nh$, $n\in N$.
Let $\phi(\mathfrak{c})=n$. This define a map $\phi:\mathcal{C}^{\mathcal{\overline{F}}}_h\to N$. Then $\phi(\mathcal{C}^{\mathcal{\overline{F}}}_h)$ is a subgroup of $N$  by noting that the end point of $\widehat{P}^{h,nh}(\mathfrak{c})$ is $n\phi(\mathfrak{c})$ for any $n\in N$ and $\mathfrak{c}\in \mathcal{C}^{\mathcal{\overline{F}}}_h$. The subgroup must be discrete. Otherwise there is $\mathfrak{c}\in \mathcal{C}^{\mathcal{\overline{F}}}_h$ such that $d(\phi(\mathfrak{c}),e)=1$. Write $\phi(\mathfrak{c})=u_1\cdots u_{n_0}$ with $u_{n_i}\in U_{\textbf{r}_i}$ and $d(u_{n_i},e)<\epsilon_1$. We can connect $\phi(\mathfrak{c})$ with $e$ with a uniformly bounded $\mathcal{U}$-path $\mathfrak{r}$. Then $\mathfrak{r}*\widehat{P}^{h,h}(\mathfrak{c}) $ is a $\mathcal{U}$-cycle at $h$ hence projects back to a $\mathcal{\overline{F}}$-cycle at $h$ by {\bf (a)}. Since $\widehat{P}^{h,h}(\mathfrak{c})$ projects back to the $\mathcal{\overline{F}}$-cycle $\mathfrak{c}$ (canonical projections are mutually inverse, see Remark \ref{re:5}), $\mathfrak{r}$ projects to a $\mathcal{\overline{F}}$-cycle at $h$. But by insisting the  perturbation is sufficiently small, $\mathfrak{r}$ cannot be projected to a $\mathcal{\overline{F}}$-cycle. This is a contradiction.

To pass from discrete to trivial, one uses a homotopy argument. For any $\mathfrak{c}\in \mathcal{C}^{\mathcal{\overline{F}}}_h$, since $\mathcal{G}_0$ is simply connected, we may create a homotopy contracting the perturbed $\mathcal{C}$ to the point $h$. Since $\mathcal{\overline{F}}$-foliations are locally transitive,  we can construct a finite sequence of $\mathcal{\overline{F}}$-cycles: $h=\mathfrak{c}_0, \mathfrak{c}_1,\cdots,\mathfrak{c}_n=\mathfrak{c}$ such that the $C^0$ distance between $\mathfrak{c}_k$ and $\mathfrak{c}_{k+1}$ are sufficiently small such that $\phi(\mathfrak{c}_k)=\phi(\mathfrak{c}_{k+1})$, $0\leq k\leq n-1$. Since $\phi(\mathfrak{c}_0)=e$, we see that $\phi(\mathfrak{c})=e$, which implies the result.

Since $P^{h,h}\circ \widehat{P}^{h,h}=I$ and $\widehat{P}^{h,h}\circ P^{h,h}=I$, then the result follows immediately from conclusions in {\bf (a)} and {\bf (b)}.
\end{proof}

\subsection{Stability of (local) transitivity}
In \cite{pesion}  Brin and Pesin show that
property of local transitivity of stable and unstable foliations of a partially
hyperbolic diffeomorphism persists under $C^2$ small perturbations. Because of this, most previous results using the geometric method required control on the second derivative. In this part by using algebraic structure we show (local) transitivity for  $C^1$ small perturbations of $\alpha_A$, which is sufficient.
\begin{definition}
Foliations $\mathcal{T}_1, \dots, \mathcal{T}_r$ are locally
transitive on $\mathcal{G}_0$ if there exists $N\in\NN$ such that for any $\epsilon
> 0$ there exists $\delta> 0$ such that for every $x\in \mathcal{G}_0$ and for every
$x\in \mathcal{G}_0$ with $d(x,y)<\epsilon$, there is a $\mathcal{T}$-path $\mathfrak{c} : x =
x_1, x_2, \dots, x_{N-1}, x_N = y$ such that $x_{k+1}\in \mathcal{T}_{i_k}(x_k)$ and
$d_{\mathcal{T}_{i_k}(x_k)}(x_{k+1}, x_k) < 2\epsilon$, where $d_{\mathcal{T}_{i_k}(x_k)}$ is the induced metric on the foliation $\mathcal{T}_{i_k}$, $1\leq i_k\leq r$.
\end{definition}
\begin{proposition}
\label{prop:transitive-persist}
If $\widetilde{\alpha}_A$ is a sufficiently $C^1$-small perturbation of $\alpha_0$ then the foliations $\overline{\mc F}$ for $\hat{\alpha}_{A}$ are transitive on $\mathcal{G}_0$.
\end{proposition}
\begin{proof}
Lemma \ref{cor:det-gen} and Corollary \ref{cor:twisted-generate} implies that $\mathcal{U}$-foliations are locally transitive (see Section 4 of \cite{kk96}).
For any $\mathcal{U}$-path $\mathfrak{c}$ with an initial point $x_1$, let $\pi^{x_1}_{\mathcal{U}}$ denote the endpoint of $P^{x_1,x_1}(\mathfrak{c})$. For any $y\in \mathcal{G}_0$, choose a $\mathcal{U}$-path $\mathfrak{c}$ with an initial point $x_1$ connecting $y$ and $x_1$ and define $\tau_{x_1}(y)=\pi^{x_1}_{\mathcal{U}}(\mathfrak{c})$. Theorem \ref{thm:perturbed-proj} shows that $\tau_{x_1}$ is well-defined; and the local transitivity of $\mathcal{U}$-foliations implies that on a small neighborhood $B(x_1)$ of $x_1$, $\tau_{x_1}$ is close to identity. Then $\tau_{x_1}(B(x_1))$ contains a small neighborhood $x_1$ by degree theory. This shows that the $\overline{\mc F}$-foliations are locally transitive, which implies that $\overline{\mc F}$-foliations are also transitive by connectedness of $\mathcal{G}_0$.

\end{proof}
\subsection{Periodic cycle functional and cocycle rigidity for perturbations}
From \eqref{for:33} we see that even though elements in $\hat{\alpha}_A$ are non-aeblian, but they preserve stable/unstable of each other, the setting for the general abelian partially hyperbolic actions in Section \ref{sec:PCF} can be adapted to that for $\hat{\alpha}_A$. Then all the results in Section \ref{sec:PCF} hold for $\hat{\alpha}_A$ over lifted $\beta$ defined accordingly. We prove the following (recall notations in Proposition \ref{po:1}):
\begin{lemma}\label{le:4}
Fix a small number $\delta>0$. For any $a\in\mathbb{A}$ and $x,\,y$ in a stable (resp. unstable) leaf of $\hat{\alpha}_A(a,\cdot)$ with $d(x,y)<\delta$, $d(p_{\beta,a}(x,y),e)$ can be made arbitrarily small if $\widetilde{\alpha}_A$ is sufficiently $C^1$-close to $\alpha_A$.
\end{lemma}
\begin{proof}
If $\widetilde{\alpha}_A$ is sufficiently $C^1$-close to $\alpha_A$, then $\sup_{a\in \mathbb{A}}\norm{\beta(a,\cdot)})_{C^0}$ is sufficiently small. Suppose $x$ and $y$ are on a stable leaf of $\hat{\alpha}_A(a,\cdot)$. Theorem \ref{th:1} shows that the coefficient  Hence the condition \eqref{for:58} of Proposition \ref{po:1} is satisfied.   From \eqref{for:7} we see that
\begin{align*}
  d\big(\gamma_{n+1}(x,y), \gamma_{n}(x,y)\big)&\leq \Pi_{j=0}^{n-1}\big(\norm{\text{Ad}(\iota(a))^{-1}}\cdot\norm{\text{Ad}(\beta(a,a^{j}\cdot x))}\big)\norm{\mathfrak{c}_n}\notag\\
  &\leq \epsilon^{n\kappa/3}\cdot\epsilon^{n\kappa/3}\norm{\mathfrak{c}_n}\leq 2c_1^{-1}\epsilon^{2n\kappa/3}\norm{\beta(a,\cdot)^\delta}_{C^0}.
\end{align*}
Using \eqref{for:5} and above estimates, we get
\begin{align*}
 &d(p_{\beta,a}(x,y), e)\leq\sum_{j=1}^\infty d\big(\gamma_{j}(x,y), \gamma_{j-1}(x,y)\big)\\
 &\leq\sum_{j=1}^n d\big(\gamma_{j}(x,y), \gamma_{j-1}(x,y)\big)+\sum_{j=n+1}^\infty d\big(\gamma_{j}(x,y), \gamma_{j-1}(x,y)\big)\\
 &\leq 2c_1^{-1}n\epsilon^{2n\kappa/3}\norm{\beta(a,\cdot)^\delta}_{C^0}+\sum_{j=n+1}^\infty c_1^{-2}c_2C(\beta,a) \lambda^{(j-1)\kappa/3}d(x,y)^\kappa\\
 &\leq 2c_1^{-1}n\epsilon^{2n\kappa/3}\norm{\beta(a,\cdot)^\delta}_{C^0}+C(\beta,a)' \lambda^{n\kappa/3}d(x,y)^\kappa.
 \end{align*}
 This shows the smallness of $d(p_{\beta,a}(x,y), e)$ when $\norm{\beta(a,\cdot)^\delta}_{C^0}$ is sufficiently
small, by first choosing n so that the second term is very small (note that the H\"{o}lder
norm $C(\beta,a)$ of $\beta$ is uniformly bounded, see Theorem \ref{th:1}), then limiting the size
of $\beta$ by limiting the size of the perturbation $\widetilde{\alpha}_A$. Similarly, we get the result if $x$ and $y$ are on a unstable leaf of $a$. Then we finish the proof.
\end{proof}

\section{Proof of Theorem \ref{thm:main} and \ref{thm:main:1}}
\label{sec:proofs}

 First, we  show the trivialization of the lifted cocycle $\beta$ over the perturbed action $\hat{\alpha}_A$ on $\mathcal{G}_0$. By Proposition \ref{prop:PCF-nec-suff}, it suffices to show that the periodic cycle functional (PCF) vanishes on all $\overline{\mc F}$-cycles.
By using Theorem \ref{thm:perturbed-proj} and Proposition \ref{po:3} we see that the function $F_\beta \of P^{h,h}$ is a continuous injective homomorphism from $\mathcal{C}_A/\mathcal{S}_A$ to $N$.
Then by Theorem \ref{th:3} $F_\beta \of P^{h,h}$ is trivial, which implies that $\beta$ is cohomologous to a constant cocycle via a transfer function on $\mathcal{G}_0$.
Then for any $\gamma$ inside $\Gamma'$ (resp. $\Gamma_\rho'$ or $\widetilde{\Gamma}$) and any $\mathcal{\overline{F}}$-path $\mathfrak{c}$ connecting $h$ and $h\gamma$, the value of $F_\beta(\mathfrak{c})$ only depends on the on the element of $\gamma$ this $\mathcal{\overline{F}}$-path represents by noting that $\hat{\alpha}_A$ commuting with the right $\Gamma'$ (resp. $\Gamma_\rho'$ or $\widetilde{\Gamma}$) actions. Therefore, these values provide a homomorphism from $\Gamma'$ (resp. $\Gamma_\rho'$ or $\widetilde{\Gamma}$) to $N$. Since $\Gamma'$ (resp. $\Gamma_\rho'$ or $\widetilde{\Gamma}$) is compactly generated (see $6.18$ of \cite{Raghunathan}), Lemma \ref{le:4} implies the smallness of the homomorphism. Then it follows from Lemma \ref{cor:no-lattice-morphisms} that such homomorphisms are trivial.  This show that $\beta$ on $\mathcal{G}_0$ takes constant values on those points that project to the same point on $M$, which implies the the trivialization of the PCF for $\beta$ on $M$ since the value of the PCF for $\beta$ on $M$ depending
only on the element $\gamma$ of $\Gamma'$ (resp. $\Gamma_\rho'$ or $\widetilde{\Gamma}$) this cycle represents.

Hence by Proposition \ref{prop:PCF-nec-suff}, $\beta$ is cohomologous to a small
constant twisted cocycle $i:A\rightarrow N$ via a
continuous transfer map $T:M\rightarrow N$ which can be chosen
close to identity in $C^0$ topology if the perturbation
$\widetilde{\alpha}_A$ small in $C^1$ topology.  Set $h'(x) = T(x)^{-1}\cdot x$ for any $x\in M$. Then we have

\begin{align*}
 h'(\bar{\alpha}_A(a,x)) & = T(\bar{\alpha}_A(a,x))^{-1}\beta(a,x)\alpha_A(a,x) \\
 & = T(\bar{\alpha}_A(a,x))^{-1}\left( T(\bar{\alpha}_A(a,x))i(a)\psi_a(T(x))^{-1}\right) a\cdot x \\
 & = i(a)aT(x)^{-1}x \\
 & = \alpha_{\mathcal{A}}(a, h'(x)),
\end{align*}
where $\alpha_{\mathcal{A}}(a,x):=i(a)a \cdot x$.

Next, we show that the map $h'$ constructed above is a homeomorphism and hence provides a topological conjugacy between $\overline{\alpha}_A$ and $\alpha_{A}$.
Since $h'$ is $C^0$ close to the identity on $M$ it is surjective and thus the action $\overline{\alpha}_A$ is semi-conjugate to the standard perturbation $\alpha_{A}$. Since the
map $h'$ is $C^0$ close to the identity, it can be lifted to continuous map close to the identity map on $\mathcal{G}_0$ commuting with the right $\Gamma'$ (resp. $\Gamma_\rho'$ or $\widetilde{\Gamma}$)  which we denote by $\hat{h}$ (using the same methods of Section \ref{sec:20}).

Furthermore, if we show that the lifted map is injective on the
cover it will follow that $h'$ on $M$ is injective and hence
a homeomorphism. On $\mathcal{G}_0$, $\alpha_{\mathcal{A}}$ is lifted to an action of $\alpha_{\tilde{\mathcal{A}}}$ or $\alpha_{\mathcal{A}'}$ (see \eqref{se:5} and \eqref{se:7} of Section \ref{sec:14}) for which we denote by $\alpha$ for simplicity. That $\hat{h}$ is a lift of a semiconjugacy implies that:
\begin{align*}
  \hat{h}(\hat{\alpha}_A(a,g))=\alpha(a, \hat{h}(g))\cdot c(g),\qquad g\in \mathcal{G}_0,
\end{align*}
where $c : \mc G_0 \to \exp(\mathfrak{z})$ is $C^0$ close to identity (hence this step may be skipped if $\mathfrak{z}= \set{0}$).

Set $\hat{\alpha}'_A(a,g)=\hat{\alpha}_A(a,g)c(g)$. Then $\hat{h}$ semi-conjugates the dynamics of $\hat{\alpha}'_A$ and $\alpha$, which implies that elements inside $\hat{\alpha}'_A$ preserve stable/unstable foliations of each other. Then we can also construct isomorphism between coarse Lyapunov cycles of $\hat{\alpha}'_A$ and $\alpha$ by canonical projections as before. Now, assume that $\hat{h}(g_1) = \hat{h}(g_2)$. Then $g_2 \in \mathcal{N}(g_1)$. Thus, we get that $\hat{h}(g_2)$ is the end point of $\widehat{P}^{x,\hat{h}(x)}(\mathfrak{p})$, where $\mathfrak{p}$ is any coarse Lyapunov path (of $\hat{\alpha}'_A$) from $g_1$ to $g_2$.
Then $\widehat{P}^{x,\hat{h}(x)}(\mathfrak{p})$ is a coarse Lyapunov cycle (of $\alpha$). This contradicts the fact that the canonical projections map coarse Lyapunov cycles of $\hat{\alpha}'_A$ to those of $\alpha$ and vice verse (see Theorem \ref{thm:perturbed-proj}). This
contradiction proves that $\hat{h}$ (and hence $h'$) is injective.

Now by letting $h:=h' \textbf{h}^{-1}$ (where $\mbf{h}$ is the Pugh-Shub-Wilkinson leaf conjugacy) we have a topological conjugacy between $\widetilde{\alpha}_A$ and $\alpha_A$. We need a Lemma:

\begin{lemma}
\label{lem:isometric}
If $i : \R^k \times \Z^l \to \mc G$ is a homomorphism, then sufficiently close homomorphisms $i'$ have $\Ad_{i'(a)}$ semisimple for every $a \in \R^k \times \Z^l$ if $A = i(\R^k \times \Z^l)$ detects every root of $\Delta_{A_0}$ for a corresponding Cartan subalgebra and $0 \not\in \Phi_{A,\rho}$ in the case of twisted symmetric space examples.
\end{lemma}

\begin{proof}
One easily sees that the assumptions survive small perturbations, so it suffices to show it for the original homomorphism $i$. As in the proofs of Propositions \ref{prop:diagonalize} and \ref{prop:diagonalize1}, after applying a necessary conjugation, we may write each $i(a) = s(a)n(a)$, where $s : A \to A_0$ is a homomorphism to some Cartan subgroup $A_0$, $n(a)$ is quasi-unipotent (has all $\Ad$-eigenvalues of modulus 1) and $[s(a),n(b)] = e$ for every $a,b \in \R^k \times \Z^l$. The condition that every root of $A_0$ is detected implies that $s(A)$ contains a regular element. Since the centralizer of a regular element in a semisimple group coincides with the centralizer of the Cartan subgroup, in the case of symmetric space examples, we are done.

For twisted symmetric space examples $\mc G = G \ltimes \R^N$, we apply the previous argument to conclude that $n(a) \in C \ltimes \R^N$, where $C$ is the compact part of the centralizer of $A_0$. If $n(a)$ contains a nontrivial $\R^N$ component, then it may be decomposed according the the weights of $\rho$. If $E_\mu$ is a weight space, and $v \in E_\mu$, then $[s(a),v] = (e^{\mu(a)} - 1) v$. Since there are no zero weights by assumption, the fact that $[s(a),n(b)] = e$ simplies that $n(b) \in C$. In particular, $i(a) = s(a)n(a)$ is always semisimple.

\end{proof}

The smoothness of $h$ along the coarse Lyapunov foliations
 of $\alpha_A$ follows as in \cite{ks97} by the general Katok-Spatzier theory of
non-stationary normal forms for partially hyperbolic abelian
actions, provided the homogeneous perturbation acts semisimply. Furthermore, the isometric distributions of $\alpha_A$ span the tangent space at every point. Then global smoothness follows from the subelliptic regularity theory \cite{Spatzier2}.

\bigskip

\noindent {\bf Acknowledgements.} The authors would like to thank their mutual advisor Anatole Katok for his introduction of the problem and valuable input and encouragement, as well as Aaron Brown, Ralph Spatzier and Federico Rodriguez-Hertz for helpful discussions on the subject.

\end{document}